\documentclass[12pt,oneside]{memoir}

\usepackage[british]{babel}
\usepackage{amsmath,amsfonts,amssymb,amsthm,amscd}
\usepackage{enumerate,array}
\usepackage{textcomp,url}
\usepackage{tikz}
\usepackage{eucal}

\usepackage{ucs}
\usepackage[utf8x]{inputenc}
\usepackage[T1]{fontenc}

\usepackage{epsfig}
\usepackage{subcaption}

\usepackage[noBBpl]{mathpazo}
\usepackage{mathabx}
\let\newbigast=\bigast
\renewcommand{\bigast}{\mathbin{\newbigast}}

\usepackage[matrix,arrow]{xy}

\setsecnumdepth{subsubsection}
\settocdepth{subsubsection}

\usepackage[pdftex,
bookmarks=true,
breaklinks=true,
bookmarksnumbered = true,
colorlinks= true,
urlcolor= green,
anchorcolor = yellow,
citecolor=blue,
pdfborder= {0 0 0}
]{hyperref}                  

\theoremstyle{plain}
\newtheorem{thm}{Theorem}
\newtheorem{lem}[thm]{Lemma}
\newtheorem{prop}[thm]{Proposition}

\theoremstyle{remark}
\newtheorem{rem}[thm]{Remark}
\newtheorem{rems}[thm]{Remarks}

\newtheorem*{exs}{Examples}

\newtheoremstyle{citing}{3pt}{3pt}{\itshape}{}{\bfseries}{.}{.5em}{\thmnote{#3}}

\theoremstyle{citing}
\newtheorem*{varthm}{}

\newcommand{\reqref}[1]{(\protect\ref{eq:#1})}
\newcommand{\resec}[1]{Section~\protect\ref{sec:#1}}
\newcommand{\rechap}[1]{Chapter~\protect\ref{chap:#1}}
\newcommand{\repr}[1]{Proposition~\protect\ref{prop:#1}}

\newcommand{\relem}[1]{Lemma~\protect\ref{lem:#1}}
\newcommand{\reth}[1]{Theorem~\protect\ref{th:#1}}
\newcommand{\req}[1]{equation~({\protect\ref{eq:#1}})}
\newcommand{\rerem}[1]{Remark~\protect\ref{rem:#1}}
\newcommand{\rerems}[1]{Remarks~\protect\ref{rem:#1}}

\makeatletter

\renewcommand{\p@enumii}{}
\renewcommand{\p@enumiii}{}
\renewcommand{\p@enumiv}{}

\def\@enum@{\list{\csname label\@enumctr\endcsname}
           {\usecounter{\@enumctr}\def\makelabel##1{
\normalfont\ignorespaces\emph{{##1}~}}
\setlength{\labelsep}{3pt}
\setlength{\parsep}{0pt}
\setlength{\itemsep}{0pt}
\setlength{\leftmargin}{0pt}
\setlength{\labelwidth}{0pt}
\setlength{\listparindent}{\parindent}
\setlength{\itemsep}{0pt}
\setlength{\itemindent}{0pt}
\topsep=3pt plus 1pt minus 1 pt}}

\def\@map#1#2[#3]{\mbox{$#1 \colon\thinspace #2 \to #3$}}
\def\map#1#2{\@ifnextchar [{\@map{#1}{#2}}{\@map{#1}{#2}[#2]}}

\makeatother

\newcommand{\rp}{\ensuremath{\mathbb{R}P^2}}

\newcommand{\Z}{\ensuremath{\mathbb Z}}
\newcommand{\N}{\ensuremath{\mathbb N}}
\newcommand{\Q}{\ensuremath{\mathbb Q}}
\newcommand{\K}{\ensuremath{\mathbb K}}
\newcommand{\St}[1][2]{\ensuremath{\mathbb S}^{#1}}

\newcommand{\FF}{\ensuremath{\mathbb F}}
\newcommand{\F}[1][n]{\ensuremath{\FF_{{#1}}}}

\renewcommand{\to}{\ensuremath{\longrightarrow}}
\renewcommand{\ker}[1]{\ensuremath{\operatorname{\text{Ker}}\left({#1}\right)}}
\newcommand{\coker}[2][]{\ensuremath{\operatorname{\text{Coker}}_{#1}\left({#2}\right)}}
\DeclareRobustCommand*{\up}[1]{\textsuperscript{#1}}
\newcommand{\brak}[1]{\ensuremath{\left\{ #1 \right\}}}
\newcommand{\ft}[1][n]{\ensuremath{\Delta_{#1}^2}}
\newcommand{\ord}[1]{\ensuremath{\left\lvert #1 \right\rvert}}

\newcommand{\aut}[1]{\ensuremath{\operatorname{\text{Aut}}\left({#1}\right)}}
\newcommand{\sn}[1][n]{\ensuremath{S_{#1}}}
\newcommand{\an}[1][n]{\ensuremath{A_{#1}}}

\newcommand{\lhra}{\mathrel{\lhook\joinrel\to}}

\newcommand{\wh}[1]{\ensuremath{\operatorname{Wh}(#1)}}
\newcommand{\skone}[1]{\ensuremath{\operatorname{SK}_{1}(\Z[{#1}])}}

\newcommand{\nkone}[2][1]{\ensuremath{\operatorname{\textit{NK}}_{#1}(\Z[#2])}}
\newcommand{\nkonealt}[2][1]{\ensuremath{\operatorname{\textit{NK}}_{#1}(#2)}}
\newcommand{\nkoneftwo}[2][1]{\ensuremath{\operatorname{\textit{NK}}_{#1}(\F[2][#2])}}

\newcommand{\nkonetwist}[3][1]{\ensuremath{\operatorname{\textit{NK}}_{#1}^{#3}(\Z[#2])}}
\newcommand{\nkonetwistftwo}[3][1]{\ensuremath{\operatorname{\textit{NK}}_{#1}^{#3}(\F[2][#2])}}

\newcommand{\ang}[1]{\ensuremath{\left\langle #1\right\rangle}}
\newcommand{\set}[2]{\ensuremath{\brak{#1 \,\mid\, #2}}}

\newcommand{\setangr}[2]{\ensuremath{\ang{#1 \,\left\lvert \, #2 \right.}}}

\newcommand{\setbigangr}[2]{\ensuremath{\bigl\langle #1 \,\bigl\lvert \, #2 \bigr. \bigr\rangle}}

\newcommand{\setr}[2]{\ensuremath{\brak{#1 \,\left\lvert \, #2 \right.}}}
\newcommand{\setl}[2]{\ensuremath{\brak{\left. #1 \,\right\rvert \, #2}}}
\newcommand{\setbigl}[2]{\ensuremath{\bigl\{\bigr. #1 \,\bigr\vert \, #2\bigr\}}}
\newcommand{\im}[1]{\ensuremath{\operatorname{Im}(#1)}}
\newcommand{\id}{\ensuremath{\operatorname{Id}}}

\renewcommand{\epsilon}{\varepsilon}
\renewcommand{\th}{\ensuremath{\up{th}}}

\newcommand{\tstar}{\ensuremath{\operatorname{T}^{\ast}}}
\newcommand{\ostar}{\ensuremath{\operatorname{O}^{\ast}}}
\newcommand{\istar}{\ensuremath{\operatorname{I}^{\ast}}}
\newcommand{\dic}[1]{\ensuremath{\operatorname{\text{Dic}}_{#1}}}
\newcommand{\dih}[1]{\ensuremath{\operatorname{\text{Dih}}_{#1}}}
\newcommand{\quat}[1][8]{\ensuremath{\mathcal{Q}_{#1}}}

\newcommand{\garside}[1][n]{\ensuremath{\Delta_{#1}}}
\newcommand{\ordelt}[1]{\ensuremath{o(#1)}}

\newcommand{\inn}[1]{\ensuremath{\operatorname{\text{Inn}}\left({#1}\right)}}

\newcommand{\twoheadlongrightarrow}{\relbar\joinrel\twoheadrightarrow}

\begin{document}

\begin{titlingpage}

\vspace*{-3cm}

\title{The lower algebraic $K$-theory of virtually cyclic subgroups of the braid groups of the sphere and of $\Z[B_{4}(\St)]$}

\author{John~Guaschi\vspace*{1mm}\\ 
Normandie Universit\'e, UNICAEN, CNRS,\\
Laboratoire de Math\'ematiques Nicolas Oresme UMR CNRS~\textup{6139},\\
CS 14032, 14032 Caen Cedex 5, France\\
e-mail:~\url{john.guaschi@unicaen.fr}\vspace*{4mm}\\
Daniel Juan-Pineda\vspace*{1mm}\\
Centro de Ciencias Matem\'aticas,\\
Universidad Nacional Aut\'onoma de M\'exico,\\
Campus Morelia, Morelia, Michoac\'an, Mexico 58089\\
e-mail:~\url{daniel@matmor.unam.mx}\vspace*{4mm}\\
Silvia Mill\'an L\'opez\vspace*{1mm}\\
Colegio de Bachilleres del Estado de Tlaxcala,\\
Calle Miguel N.~Lira No.~3 Col.~Centro,\\
Tlaxcala CP 90000, Mexico\\
e-mail:~\url{smillan@cobatlaxcala.edu.mx}}


\date{29th June 2018}

\maketitle

\enlargethispage{2cm}

\vspace*{-0.5cm}

\setlength{\absleftindent}{0em}
\setlength{\absrightindent}{0em}

\begin{abstract}
We study $K$-theoretical aspects of the braid groups $B_{n}(\St)$ on $n$ strings of the $2$-sphere, which by results of the second two authors, are known to satisfy the Farrell-Jones fibred isomorphism conjecture~\cite{JM}. In light of this, in order to determine the algebraic $K$-theory of the group ring $\Z[B_{n}(\St)]$, one should first compute that of its virtually cyclic subgroups, which were classified by D.~L.~Gon\c{c}alves and the first author~\cite{GG2}. We calculate the Whitehead and $K_{-1}$-groups of the group rings of the finite subgroups (dicyclic and binary polyhedral) of $B_{n}(\St)$ for all $4\leq n\leq 11$. Some new phenomena occur, such as the appearance of torsion for the $K_{-1}$-groups. We then go on to study the case $n=4$ in detail, which is the smallest value of $n$ for which $B_{n}(\St)$ is infinite. We show that $B_{4}(\St)$ is an amalgamated product of two finite groups, from which we are able to determine a universal space for proper actions of the group $B_{4}(\St)$. We also calculate the algebraic $K$-theory of the infinite virtually cyclic subgroups of $B_{4}(\St)$, including the Nil groups of the quaternion group of order $8$. This enables us to determine the lower algebraic $K$-theory of $\Z[B_{4}(\St)]$.
\end{abstract}

\end{titlingpage}

\frontmatter

%
%


\begin{KeepFromToc}
\tableofcontents
\end{KeepFromToc}

\mainmatter

\chapter{Introduction}
 
Given a group $G$, the $K$-theoretic fibred isomorphism conjecture of F.~T.~Farrell and L.~E.~Jones asserts that the algebraic $K$-theory of its integral group ring $\Z[G]$ may be computed from the knowledge of the algebraic $K$-theory groups of its virtually cyclic subgroups (see~\cite{FJC} or Appendix~\ref{chap:fic} for the statement). This conjecture has been verified for a number of classes of groups, such as discrete cocompact subgroups of virtually connected Lie groups~\cite{FJC}, finitely-generated Fuchsian groups~\cite{BJP}, Bianchi groups~\cite{BFJP}, pure braid groups of aspherical surfaces~\cite{AFR}, braid groups of aspherical surfaces~\cite{FR} and for some classes of mapping class groups~\cite{BJL}. In~\cite{LO}, Lafont and Ortiz presented explicit computations of the lower algebraic $K$-theory of hyperbolic $3$-simplex reflection groups, and then together with Magurn, for that of certain reflection groups~\cite{LMO}. Similar calculations were performed for virtually free groups in~\cite{JLMP}. 

Let $n\in \N$, let $M$ be a surface, and let $B_{n}(M)$ (resp.\ $P_{n}(M)$) denote the \emph{$n$-string braid group} (resp.\ \emph{$n$-string pure braid group}) of $M$~\cite{Bi,H}. Some basic information and facts about surface braid groups are given in Appendix~\ref{chap:braids}. The braid groups of the $2$-sphere $\St$ were first studied by Zariski, and then later by Fadell and Van Buskirk during the 1960's~\cite{FVB,Z}. If $M$ either is the $2$-sphere $\St$ or the projective plane $\rp$, the results of~\cite{AFR,FR} do not apply to its braid groups, the principal reason being that these groups possess torsion~\cite{FVB,VB}. The second two authors of this book proved that the conjecture of Farrell and Jones holds also for the braid groups of these two surfaces, which using the method prescribed by the conjecture, enabled them to carry out complete computations of the lower algebraic $K$-groups for $P_{n}(\St)$ and $P_{n}(\rp)$~\cite{JM}. One necessary ingredient in this process is the knowledge of the virtually cyclic subgroups of $P_{n}(M)$. For $n\geq 4$, $P_{n}(\St)$ has only one non-trivial finite subgroup, generated by the `full twist' braid, which is central and of order $2$, and from this, it is straightforward to see that $P_{n}(\St)$ has very few isomorphism classes of virtually cyclic subgroups. The classification of the isomorphism classes of the virtually cyclic subgroups of $P_{n}(\rp)$, which was established in~\cite{GG8} and used subsequently in~\cite{JM} to compute the $K$-theory groups of $\Z[P_{n}(\rp)]$, is rather more involved. 

Our aim in this manuscript is to implement similar $K$-theoretical computations for the group ring $\Z[B_{n}(\St)]$ of the full braid groups of $\St$. In order to do so, one must determine initially the virtually cyclic subgroups (finite, and then infinite) of $B_{n}(\St)$, and then compute the $K$-groups of these subgroups. If $n\leq 3$ then $B_{n}(\St)$ is finite, and so we shall assume in much of this manuscript that $n\geq 4$. The torsion of $B_{n}(\St)$ was determined in~\cite{GVB}, and its finite order elements were classified in~\cite{Mu}. It was shown by D.~L.~Gon\c{c}alves in collaboration with the first author that up to isomorphism, the finite subgroups of $B_{n}(\St)$ are cyclic, dicyclic or binary polyhedral (see~\cite{GG} or \reth{finitebn}). As for the corresponding pure braid groups, one must then determine the infinite virtually cyclic subgroups of $B_{n}(\St)$ with the aid of the characterisation due to Epstein and Wall of infinite virtually cyclic groups~\cite{Ep,JuLe,Wall1}. Up to isomorphism and with a few exceptions in the case that $n$ is a small even number, this was achieved in~\cite{GG2}. A taste of the results is given in \reth{vcodd} when $n$ is odd and in \reth{vcb4S2} when $n=4$.

In the ensuing quest to compute the lower algebraic $K$-theory of the group ring $\Z[B_{n}(\St)]$, we encountered a number of difficulties, among them:
\begin{enumerate}[(a)]
\item the family of virtually cyclic subgroups of $B_{n}(\St)$ is relatively large, and depends on $n$, contrasting sharply with the case of the pure braid groups analysed in~\cite{JM}. 
\item the lower algebraic $K$-theory of even the finite subgroups of $B_{n}(\St)$ is poorly understood, and the investigation of the $K$-groups of dicyclic and binary polyhedral groups presents additional technical obstacles compared to that of the dihedral and polyhedral groups that appear in~\cite{LMO,LO} for example.
\item\label{it:univspace} in order to apply the method of calculation suggested by the fibred isomorphism conjecture, one needs not only to compute the various Nil groups, but also to discover a suitable universal space for the family of virtually cyclic subgroups of $B_{n}(\St)$. In spite of the rich topological and geometric structures of the braid groups and their associated configuration spaces, this space has thus far proved to be elusive for $n\geq 5$. 
\end{enumerate}

\rechap{loweralgfinite} is devoted to the second point, that of the computation of the lower $K$-theory groups of many of the finite subgroups of $B_{n}(\St)$. In \resec{classvc}, we recall the classification up to isomorphism of the finite subgroups of $B_{n}(\St)$, and of its virtually cyclic subgroups when $n$ is odd or $n=4$. In \resec{ccbpg}, we compute the number of different types of conjugacy classes of the binary polyhedral groups in \repr{ribinpoly}. These results are used later in the chapter, when we determine the lower algebraic $K$-theory of the group rings of these groups. In Sections~\ref{sec:whitehead} and~\ref{sec:Kminusone}, we calculate the Whitehead and the $K_{-1}$-groups respectively of the integral group rings of many of the finite subgroups of $B_{n}(\St)$. To our knowledge, these sections contain a number of original results, as well as some new phenomena, such as the existence of torsion for some $K_{-1}$-groups, that did not appear in previous work~\cite{JLMP,LMO,LO}. This necessitates alternative techniques, notably the application of results of Yamada to determine local Schur indices~\cite{Y2,Y}, which enables us to calculate the torsion of our $K_{-1}$-groups. We believe that the methods that we use to calculate these $K$-groups for dicyclic groups of certain orders may be extended to dicyclic groups of other orders. The Whitehead groups are given in \repr{whbnS2}. The main results concerning the $K_{-1}$-groups are \reth{kminusonedic4prime} for dicyclic groups of order $4m$, where $m$ is an odd prime, \repr{kminusonequat2k} for the generalised quaternions (the dicyclic groups of order a power of $2$), \repr{kminusonebinpoly} for the binary polyhedral groups. We also compute the $K_{-1}$-groups of the dicyclic groups of order $24$, $36$ and $40$ in \repr{rankdicextra}, and of cyclic groups of order $2p^{q}$, $12$ and $20$, where $p$ is prime and $q\in \N$ in \repr{kminusonecyclic}. In \resec{Kzero}, we recall briefly the work of Swan pertaining to the calculation of the $\widetilde{K}_{0}$-groups of the group rings of the binary polyhedral groups, and of the dicyclic groups of order $4m$ for $m\leq 11$~\cite{Sw3}, and in \reth{K0cyc}, we compute $\widetilde{K}_0(\Z[G])$ where $G$ is a cyclic group of order $18$, $20$ or $22$. For dicyclic groups of higher order, the situation is complicated, and little seems to be known about the corresponding $\widetilde{K}_{0}$-groups. In \resec{lower411}, we sum up the results of many of our computations in Table~\ref{tab:finite}, which lists the lower $K$-theory groups of the finite subgroups of $B_{n}(\St)$ for all $4\leq n\leq 11$. From this table, we may also obtain the lower $K$-theory groups of $B_{n}(\St)$ in the cases where $B_{n}(\St)$ is finite, namely for $n\in \brak{1,2,3}$.

The aim of the remaining two chapters is to determine the lower algebraic $K$-theory of $\Z[B_{4}(\St)]$. We study the case of $B_4(\St)$ in detail and show how the algebraic and geometric features of this group interact, thus allowing us to compute its lower $K$-groups. In preparation for the explicit computations in \rechap{kthb4St}, in \rechap{b4s2}, we describe the ingredients for the corresponding computations in the case of \emph{infinite} virtually cyclic subgroups. We start by recalling some basic facts about the group $B_{4}(\St)$ in \resec{gensB4}. One striking property, which was proved in~\cite[Theorem~1.3(3)]{GG3}, is that it possesses a finite normal subgroup isomorphic to the quaternion group of order $8$. This enables us to show in \repr{b4amalg} that $B_{4}(\St)$ is an amalgamated product of the generalised quaternion group of order $16$ and the binary tetrahedral group, the amalgamation being along this normal subgroup, from which we deduce in \rerem{b4S2hyperbolic} that it is hyperbolic in the sense of Gromov. In \resec{maxvcb4}, we determine the isomorphism classes of the maximal virtually cyclic subgroups of $B_{4}(\St)$ in \reth{maxvcb4}, and in \resec{conjclmaxvc}, we show that there are an infinite number of conjugacy classes for each of the isomorphism classes of the infinite maximal virtually cyclic subgroups. These properties aid greatly, not just in the computation of the $K$-groups of the virtually cyclic subgroups of $B_{4}(\St)$ and of the corresponding Nil groups, but also to exhibit an appropriate universal space referred to in~(\ref{it:univspace}) above. 

In \rechap{kthb4St}, we bring together the results of the previous chapters to compute the lower $K$-groups of $B_{4}(\St)$. In \resec{calcingred}, we recall some facts and results about the lower $K$-theory of infinite virtually cyclic groups. In \resec{prelimkth}, we determine the lower $K$-groups of $B_{4}(\St)$ up to the computations of the associated Nil groups, and in \resec{nilgrpcomp}, we determine these Nil groups. One result that is interesting is its own right is \repr{nkq8} where we calculate the Bass Nil groups $\nkone[i]{\quat}$ for $i=0,1$ of the quaternion group $\quat$ of order $8$. Our calculations show that both $\wh{B_4(\St)}$ and $\widetilde{K}_0(\Z[B_4(\St)])$ are infinitely-generated Abelian groups, and contain infinite direct sums of Abelian $2$-groups. In contrast, we shall see that $K_{-1}(\Z[B_4(\St)])\cong\Z\oplus \Z_2$. Compared with the families of groups considered in~\cite{JLMP,LMO,LO}, the existence of torsion here is once more a new phenomenon. We summarise these results as follows.

\begin{thm}\label{th:B4}
The group $B_4(\St)$ has the following lower algebraic $K$-groups: 
\begin{gather*}
\wh{B_4(\St)}\cong\Z\oplus  \operatorname{Nil}_1,\\
\widetilde{K}_{0}(\Z[B_4(\St)])\cong\Z_2\oplus  \operatorname{Nil}_0,\quad\text{and}\\
K_{-1}(\Z[B_4(\St)])\cong\Z_2\oplus\Z,\\
\text{$K_{-i}(\Z[B_4(\St)])= 0$ for all $i\geq 2$},
\end{gather*}
where for $i=0,1$, the groups $\operatorname{Nil}_i$ are isomorphic to a countably-infinite direct sum of $\Z_2$, $\Z_4$ or $\Z_2\oplus \Z_4$.
\end{thm}

For $n\geq 5$, we cannot expect the group $B_{n}(\St)$ to enjoy properties, such as hyperbolicity, similar to those of $B_{4}(\St)$. Furthermore, we have not been able as yet to determine an appropriate model for the universal space for the family of virtually cyclic subgroups of $B_{n}(\St)$. There are some candidates suggested by the theory of Brunnian braids, but the corresponding subgroups are of large index, and do not seem to be terribly useful from a practical viewpoint. On the positive side, for small odd values of $n$, the family of virtually cyclic subgroups of $B_{n}(\St)$ is relatively small, and our techniques enable us to determine the corresponding $K$-groups of these subgroups. If we are able to find an appropriate universal space, we hope to be able to determine the $K$-groups of $B_{n}(\St)$ for other values of $n$.

\section*{Acknowledgements}

We wish to thank the following colleagues for helpful and fruitful discussions: Bruno Angl\`es, for help with the Galois theory of \resec{rankkminusone}; Eric Jespers, Gerardo Raggi and \'Angel del R\'io (and the \texttt{GAP} package~\texttt{Wedderga}~\cite{gap}), for aiding us with the Wedderburn decomposition of dicyclic and binary polyhedral groups; Jean-Fran\c{c}ois Lafont, Ivonne Ortiz and Stratos Prassidis, for conversations on $K$-theoretical aspects of our work during the early stages of the writing of this paper; and Chuck Weibel for valuable comments at various points. We would also like to thank the referees for their comments on the manuscript, and one referee in particular, who made a number of useful remarks and suggestions to make it easier to read.

The authors are grateful to the French-Mexican Laboratoire International Associ\'e LAISLA for partial financial support. D.~Juan-Pineda was partially supported by the CNRS, PAPIIT-UNAM and CONACYT (M\'exico), and J.~Guaschi was partially supported by the ANR project TheoGar n\up{o}~ANR-08-BLAN-0269-02. J.~Guaschi would like to thank the CNRS for having granted him a `d\'el\'egation' during the writing of part of this paper. J.~Guaschi and S.~Mill\'an also wish to thank CONACYT (M\'exico) for partial financial support through its programme `Estancias postdoctorales y sab\'aticas vinculadas al fortalecimiento de la calidad del posgrado nacional'.

\chapter[Lower algebraic $K$-theory of the finite subgroups of $B_{n}(\St)$]{Lower algebraic $K$-theory of the finite subgroups of $B_{n}(\St)$}\label{chap:loweralgfinite}

\chaptermark{$K$-theory of the finite subgroups of $B_{n}(\St)$}

\section{Classification of the virtually cyclic subgroups of $B_{n}(\St)$}\label{sec:classvc}

If $G$ is a group that satisfies the Farrell-Jones fibred isomorphism conjecture, the lower algebraic $K$-theory of the group ring $\Z [G]$ may be calculated in principle if one knows the lower algebraic $K$-theory of the group rings of the virtually cyclic subgroups of $G$ (see Appendix~\ref{chap:fic}). Recall that a group is said to be \emph{virtually cyclic} if it possesses a cyclic subgroup of finite index. Clearly any finite group is virtually cyclic. By results of Epstein and Wall~\cite{Ep,Wall1}, an infinite group is virtually cyclic if and only if it has two ends. This allows us to show that any infinite virtually cyclic group $G$ is isomorphic either to $F\rtimes \Z$ or to $G_{1}\bigast_{F} G_{2}$, where $F$ is a finite normal subgroup of $G$, and in the second case, $F$ is of index $2$ in both $G_{1}$ and $G_{2}$. Consequently, in order to determine the virtually cyclic subgroups of $G$, one must first discover its finite subgroups. Let $G=B_{n}(\St)$, and if $m\geq 2$, let $\dic{4m}$ denote the \emph{dicyclic} group of order $4m$, with presentation:
\begin{equation}\label{eq:presdic}
\dic{4m}=\setbigangr{x,y}{x^{m}=y^{2},\; yxy^{-1}=x^{-1}}.
\end{equation}
If $m$ is a power of $2$, then we shall also say that $\dic{4m}$ is a \emph{generalised quaternion} group, and denote it by $\quat[4m]$. Using a presentation of $B_{n}(\St)$, such as that given in \reth{fvb}, if $n\leq 3$, $B_{n}(\St)$ may be seen to be finite. The group $B_{1}(\St)$ is trivial, $B_{2}(\St)$ is isomorphic to $\Z_{2}$, and $B_{3}(\St)$ is isomorphic to $\dic{12}$, and its subgroups may be obtained easily. So in most of what follows, we shall assume that $n\geq 4$, in which case $B_{n}(\St)$ is infinite. The finite subgroups of $B_{n}(\St)$ were classified up to isomorphism in~\cite{GG} as follows.
\begin{thm}[{\cite[Theorem~1.3]{GG}}]\label{th:finitebn}
Let $n\geq 4$. The maximal finite subgroups of $B_n(\St)$ are isomorphic to one of the following groups:
\begin{enumerate}
\item\label{it:fina} $\Z_{2(n-1)}$ if $n\geq 5$,
\item\label{it:finb} $\dic{4n}$,
\item\label{it:finc} $\dic{4(n-2)}$ if $n=5$ or $n\geq 7$,
\item\label{it:find} the binary tetrahedral group, denoted by $\tstar$, if $n\equiv 4 \pmod 6$,
\item\label{it:fine} the binary octahedral group, denoted by $\ostar$, if $n\equiv 0,2\pmod 6$,
\item\label{it:finf} the binary icosahedral group, denoted by $\istar$, if $n\equiv 0,2,12,20\pmod{30}$.
\end{enumerate}
\end{thm}
More information on $\tstar,\ostar$ and $\istar$, to which we refer collectively as the \emph{binary polyhedral groups}, may be found in~\cite{AM,Co,CM,Wo}. It is well known that the subgroups of dicyclic and binary polyhedral groups are cyclic, dicyclic or binary polyhedral (see~\cite[Proposition~85]{GG2} for the binary polyhedral case). We recall from~\cite[page~759]{GG} that the finite subgroups of $B_{n}(\St)$ are periodic of period $1,2$ or $4$. Further, by~\cite[Proposition~1.5]{GG}, any two finite subgroups of $B_{n}(\St)$ that are isomorphic are also conjugate, with the exception of those subgroups that are isomorphic to $\Z_{4}$ and $\dic{4r}$ if $n$ is even and $r$ divides $n/2$ or $(n-2)/2$, in which case are two conjugacy classes in each isomorphism class. Consequently, any such subgroup $H$ of $B_{n}(\St)$ satisfies the following three conditions (see~\cite{AM} or~\cite[page~20]{T}):
\begin{enumerate}[(a)]
\item  the $p^2$-condition: for any prime divisor $p$ of $\ord{H}$ ($\ord{H}$ denotes the order of $H$), $H$ contains no subgroup isomorphic to $\Z_{p}\times \Z_{p}$.
\item  the $2p$-condition: for any prime divisor $p$ of $\ord{H}$, any subgroup of $H$ of order $2p$ is cyclic.
\item the Milnor condition: if $H$ has an element of order $2$, this element is unique (and so is central in $H$).
\end{enumerate}

\begin{rems}\mbox{}\label{rem:centre}
\begin{enumerate}[(i)]
\item\label{it:centre1} The $p^2$-condition implies that the Sylow $p$-subgroups of $H$ are either cyclic or generalised quaternion, the latter case occurring only if $p=2$.
\item\label{it:centre2} If $G$ is a dicyclic or binary polyhedral group, the centre $Z(G)$ is generated by the unique element of order $2$.
\end{enumerate}
\end{rems}

The second step in the process is to classify the infinite virtually cyclic subgroups of $B_{n}(\St)$. Up to isomorphism, and with a finite number of exceptions, this was achieved in~\cite{GG2}. The statement of the main result of~\cite{GG2} is somewhat long to explain here, but to give a flavour of the results, we state the classification when $n$ is odd, in which case the classification is complete for all values of $n$. We shall also recall the case $n=4$ later in \reth{vcb4S2}.
\begin{thm}[{\cite[Theorem~7]{GG2}}]\label{th:vcodd}
Let $n\geq 3$ be odd. Then up to isomorphism, the virtually cyclic subgroups of $B_{n}(\St)$ are as follows. %
\begin{enumerate}[(I)]
\item The isomorphism classes of the finite virtually cyclic subgroups of $B_{n}(\St)$ are:
\begin{enumerate}[(i)]
\item $\dic{4m}$, where $m\geq 3$ divides $n$ or $n-2$.
\item $\Z_{m}$, where $m\in \N$ divides $2n$, $2(n-1)$ or $2(n-2)$.
\end{enumerate}
\item If in addition $n\geq 5$, then the following groups are the isomorphism classes of the infinite virtually cyclic subgroups of $B_{n}(\St)$.
\begin{enumerate}[(i)]
\item $\Z_{m} \rtimes_{\theta} \Z$, where $\theta(1)\in\brak{\id, -\id}$, $m$ is a strict divisor of $2(n-i)$, for $i\in \brak{0,2}$, and $m\neq n-i$.
\item $\Z_{m} \times \Z$, where $m$ is a strict divisor of $2(n-1)$.
\item $\dic{4m} \times \Z$, where $m\geq 3$ is a strict divisor of $n-i$ for $i\in \brak{0,2}$.
\item $\Z_{4q}\bigast_{\Z_{2q}}\Z_{4q}$, where $q$ divides $(n-1)/2$.
\item $\dic{4q}\bigast_{\Z_{2q}}\dic{4q}$, where $q\geq 2$ is a strict divisor of $n-i$, and $i\in \brak{0,2}$.
\end{enumerate}
\end{enumerate}
\end{thm}

The aim of the rest of this chapter is to compute the lower algebraic $K$-theory of the group rings of many of the finite groups of $B_{n}(\St)$. In \resec{ccbpg}, we start by determining the number of different types of conjugacy classes in the binary polyhedral groups. The main result of that section, \repr{ribinpoly}, will be used in the rest of the chapter to determine the lower algebraic $K$-theory of the group rings of $\tstar$, $\ostar$ and $\istar$. In Sections~\ref{sec:whitehead},~\ref{sec:Kzero} and~\ref{sec:Kminusone}, we calculate respectively the Whitehead, $\widetilde{K}_{0}$- and $K_{-1}$-groups of the group rings of many groups that appear in the statement of \reth{finitebn}. This allows us in \resec{lower411} to determine the lower algebraic $K$-theory of the group rings of the isomorphism classes of the finite groups of  $B_{n}(\St)$ for all $4\leq n\leq 11$, the results being summarised in Table~\ref{tab:finite}.

\section{Conjugacy classes of binary polyhedral groups}\label{sec:ccbpg}

In this section, we compute the number of certain types of conjugacy classes of elements of the binary polyhedral groups. Some of these numbers will be used in the calculations of the lower algebraic $K$-theory of the group rings of these groups. Recall first that $\ostar$ is generated by the elements $X,P,Q$ and $R$, subject to the following relations~\cite[page~198]{Wo}:
\begin{equation}\label{eq:presostar}
\hspace*{-2mm}\left\{ 
\begin{aligned}
& X^{3}=1,\, P^2=Q^2=R^2,\, PQP^{-1}=Q^{-1},\, XPX^{-1}=Q,\, XQX^{-1}=PQ\\
& RXR^{-1}=X^{-1},\, RPR^{-1}=QP,\,RQR^{-1}=Q^{-1}.
\end{aligned}
\right.
\end{equation}
It follows that $\ostar$ contains $\tstar$ as an index $2$ subgroup generated by $X$, $P$ and $Q$ that are subject to the relations given in the first line of~\reqref{presostar}. The subgroup $\ang{P,Q}$ is isomorphic to $\quat$, and $X$ is of order $3$ and acts by conjugation on $\ang{P,Q}$ by permuting $P,Q$ and $PQ$ cyclically, so that $\tstar\cong \quat\rtimes \Z_{3}$. Further, $\ostar \setminus \tstar$ is comprised of $12$ elements of order $4$ and twelve of order $8$. We recall also that $\ord{\istar}=120$, that $\istar$ is comprised of the trivial element, one element of order $2$, thirty elements of order $4$, twenty elements of order $3$ and twenty of order $6$, twenty-four elements of order $5$ and twenty-four elements of order $10$. The group $\istar$ also contains subgroups isomorphic to $\tstar$. The following lemma will be useful in some of our computations.

\begin{lem}\label{lem:centbinpoly} 
Let $G$ be a dicyclic or binary polyhedral group, and let $g\in G$ be an element of order greater than or equal to $3$. Then the centraliser $C_{G}(g)$ of $g$ in $G$ is cyclic.
\end{lem}

\begin{proof}
Let $g\in G$ be of order at least $3$. Then $g\in Z(C_{G}(g))$, so $\ord{Z(C_{G}(g))}\geq 3$. The subgroups of $G$ are cyclic, dicyclic or binary polyhedral (see~\cite[Proposition~85]{GG2} for the binary polyhedral case). Then $Z(C_{G}(g))$ is cyclic because the centre of a dicyclic or binary polyhedral group is isomorphic to $\Z_{2}$ by \rerems{centre}(\ref{it:centre2}).
\end{proof}

If $G$ is a finite group and $d$ is a divisor of $\ord{G}$, let $\nu(d)$ be the number of elements of order $d$ in $G$, let $r_{0}(d)$ be the number of conjugacy classes of elements of order $d$ in $G$, let $r_{1}(d)$ be the number of conjugacy classes of unordered pairs $\brak{g,g^{-1}}$ of elements of order $d$ in $G$, and let $r_{2}(d)$ be the number of conjugacy classes of cyclic subgroups of $G$ of order $d$ in $G$. If $g,g'$ are elements of $G$ of the same order $d$ such that $\brak{g,g^{-1}}$ is conjugate to $\brak{g',g'^{-1}}$, then there exists $h\in G$ such that either $hgh^{-1}=g'$ or $hgh^{-1}=g'^{-1}$, so $h\ang{g}h^{-1}=\ang{g'}=\ang{g'^{-1}}$, and thus $r_{1}(d)\geq r_{2}(d)$. It thus follows that:
\begin{equation}\label{eq:ineqr}
\nu(d)\geq r_{0}(d)\geq r_{1}(d)\geq r_{2}(d).
\end{equation}
For small $d$, the inequality $r_{1}(d)\geq r_{2}(d)$ is an equality.

\begin{lem}\label{lem:r1dr2d}
Let $G$ be a finite group. Then $r_{1}(d)=r_{2}(d)$ for all $d\in \brak{1,2,3,4,6}$.
\end{lem}

\begin{proof}[Proof of \relem{r1dr2d}]
If $d\in\brak{1,2}$ and if $g\in G$ is of order $d$ then the pair $\brak{g,g^{-1}}$ reduces to $\brak{g}$ and then clearly $r_{1}(d)=r_{2}(d)$. So assume that $d\in\brak{3,4,6}$. From above, it suffices to show that $r_{1}(d)\leq r_{2}(d)$. Note that if $g\in G$ is of order $d$ then the elements of $\ang{g}$ of order $d$ are precisely $g$ and $g^{-1}$. If $g'\in G$ is also of order $d$ and $\ang{g}$ and $\ang{g'}$ are conjugate then $g$ is conjugate to $g'$ or $g'^{-1}$, so $\brak{g,g^{-1}}$ is conjugate to $\brak{g',g'^{-1}}$, which completes the proof of the lemma.
\end{proof} 

The following proposition summarises the values of $\nu(d), r_{0}(d), r_{1}(d)$ and $r_{2}(d)$ for each of the three binary polyhedral groups. It will be used in the calculations of Whitehead and $K_{-1}$-groups in \repr{whbnS2} and \repr{kminusonebinpoly} respectively.

\begin{prop}\label{prop:ribinpoly}\mbox{}
\begin{enumerate}
\item\label{it:ribinpolya} If $G=\tstar$, $\nu(d), r_{0}(d), r_{1}(d)$ and $r_{2}(d)$ are given by:
\renewcommand{\arraystretch}{1.2}
\setlength{\tabcolsep}{5pt}
\begin{center}
\begin{tabular}{|>{$}c<{$}||>{$}c<{$}|>{$}c<{$}|>{$}c<{$}|>{$}c<{$}|>{$}c<{$}|}
\hline
d & 1 & 2 & 3 & 4 & 6\\ \hline\hline
\nu(d) & 1 & 1 & 8 & 6 & 8\\ \hline
r_{0}(d) & 1 & 1 & 2 & 1 & 2\\ \hline
r_{1}(d) & 1 & 1 & 1 & 1 & 1\\ \hline
r_{2}(d) & 1 & 1 & 1 & 1 & 1\\ \hline
\end{tabular}
\end{center}
If $d\in \brak{3,6}$, and $g\in \tstar$ is of order $d$ then $g$ and $g^{-1}$ are representatives of the two conjugacy classes of elements of order $g$.

\item\label{it:ribinpolyb} If $G=\ostar$, $\nu(d), r_{0}(d), r_{1}(d)$ and $r_{2}(d)$ are given by:
\renewcommand{\arraystretch}{1.2}
\setlength{\tabcolsep}{5pt}
\begin{center}
\begin{tabular}{|>{$}c<{$}||>{$}c<{$}|>{$}c<{$}|>{$}c<{$}|>{$}c<{$}|>{$}c<{$}|>{$}c<{$}|}
\hline
d & 1 & 2 & 3 & 4 & 6 & 8\\ \hline\hline
\nu(d) & 1 & 1 & 8 & 18 & 8 & 12\\ \hline
r_{0}(d) & 1 & 1 & 1 & 2 & 1 & 2\\ \hline
r_{1}(d) & 1 & 1 & 1 & 2 & 1 & 2\\ \hline
r_{2}(d) & 1 & 1 & 1 & 2 & 1 & 1\\ \hline
\end{tabular}
\end{center}
If $d=4$, and $g_{1}$ and $g_{2}$ are elements of $\ostar$ of order $4$ such that $g_{1}\in \tstar$ and $g_{2}\notin \tstar$ then $g_{1}$ and $g_{2}$ are representatives of the two conjugacy classes of elements of order $4$. If $d=8$, and $g\in \tstar$ is of order $8$, $g$ and $g^{3}$ are representatives of the two conjugacy classes of elements of order $8$.

\item\label{it:ribinpolyc} If $G=\istar$, $\nu(d), r_{0}(d), r_{1}(d)$ and $r_{2}(d)$ are given by:
\renewcommand{\arraystretch}{1.2}
\setlength{\tabcolsep}{5pt}
\begin{center}
\begin{tabular}{|>{$}c<{$}||>{$}c<{$}|>{$}c<{$}|>{$}c<{$}|>{$}c<{$}|>{$}c<{$}|>{$}c<{$}|>{$}c<{$}|}
\hline
d & 1 & 2 & 3 & 4 & 5 & 6 & 10\\ \hline\hline
\nu(d) & 1 & 1 & 20 & 30 & 24 & 20 & 24\\ \hline
r_{0}(d) & 1 & 1 & 1 & 1 & 2 & 1 & 2\\ \hline
r_{1}(d) & 1 & 1 & 1 & 1 & 2 & 1 & 2\\ \hline
r_{2}(d) & 1 & 1 & 1 & 1 & 1 & 1 & 1\\ \hline
\end{tabular}
\end{center}
If $d=5$ (resp.\ $d=10$), and $g\in \tstar$ is of order $d$, $g$ and $g^{2}$ (resp.\ $g$ and $g^{3}$) are representatives of the two conjugacy classes of elements of order $d$.
\end{enumerate}
\end{prop}

\begin{proof}
Since the binary polyhedral groups have exactly one element of order $1$ and of order $2$, the elements of the columns for $d\in \brak{1,2}$ are all equal to $1$. So we suppose that $d\geq 3$.
\begin{enumerate}[(a)]
\item Let $G=\tstar$. We make use of the presentation of whose relations are given by the first line of~\reqref{presostar}. Let $d=3$. The subgroups of $\tstar$ of order $3$ are the Sylow $3$-subgroups of $\tstar$, so they are pairwise conjugate. Thus $r_{2}(3)=1$, and $r_{1}(3)=1$ by \relem{r1dr2d}. We now compute $r_{0}(3)$. Let $g\in \tstar$ be of order $3$. Its centraliser $C_{\tstar}(X)$ contains $g$ and the central element $P^{2}$ of $\tstar$ of order $2$, so $C_{\tstar}(g)$ contains the cyclic subgroup $\ang{gP^{2}}$ of order $6$. Now $\tstar$ contains no element of order greater than $6$, and $C_{\tstar}(g)$ is cyclic by \relem{centbinpoly}. It follows that $C_{\tstar}(g)=\ang{gP^{2}}$. By the orbit-stabiliser theorem, the conjugacy class of $g$ contains $4$ elements, and since $\tstar$ possesses $8$ elements of order $4$, we deduce that $r_{0}(3)=2$. The fact that $r_{1}(3)=1$ implies that $g$ and $g^{-1}$ belong to different conjugacy classes, so $g$ and $g^{-1}$ are representatives of the two conjugacy classes of elements of order $3$. By adjoining $P^{2}$ to $g$, we see that $r_{i}(6)=r_{i}(3)$ for all $i\in \brak{0,1,2}$. 

Now let $d=4$. The six elements of $\tstar$ of order $4$ are contained in the subgroup $\ang{P,Q}$ isomorphic to $\quat$. We have $QPQ^{-1}=P^{-1}$, so $P$ is conjugate to $P^{-1}$. Further, conjugation by $X$ permutes $P$, $Q$ and $PQ$, so $\tstar$ contains a single conjugacy class of elements of order $4$, $r_{0}(4)=1$, and thus $r_{1}(4)=r_{2}(4)=1$ by~\reqref{ineqr}.

\item Let $G=\ostar$. First let $d=3$. All of the elements of $\ostar$ of order $3$ are contained in its subgroup $\tstar$, and so $r_{0}(3)\leq 2$. From~\reqref{presostar}, we have the relation $RXR^{-1}=X^{-1}$, where $X$ is of order $3$. Since $X$ and $X^{-1}$ are representatives of the two conjugacy classes of elements of order $3$ in $\tstar$, it follows that there is a single conjugacy class of elements of order $3$ in $\ostar$, so $r_{0}(3)=1$, and $r_{1}(3)=r_{2}(3)=1$ by~\reqref{ineqr}. Once more, the values for $d=6$ are obtained by adjoining $P^{2}$ to $X$. 

Now let $d=4$. From the case $G=\tstar$, the six elements of order $4$ that belong to the subgroup $\tstar$ of $\ostar$ are pairwise conjugate, and since $\tstar$ is normal in $\ostar$, they form a complete conjugacy class of elements of order $4$. Now let $\mathcal{U}$ denote the set of twelve elements of order $4$ that belong to $\ostar \setminus \tstar$, and let $g\in \mathcal{U}$. Since $C_{\ostar}(g)$ is cyclic by \relem{centbinpoly} and contains $\ang{g}$, it follows that $\ord{C_{\ostar}(g)}\in \brak{4,8}$. Suppose that $\ord{C_{\ostar}(g)}=8$. Then $C_{\ostar}(g)\cong \Z_{8}$, and there exists $h\in C_{\ostar}(g)$ of order $8$ such that $g=h^{2}$. But since $\tstar$ is of index $2$ in $\ostar$, it follows that $h^{2}\in \tstar$, which contradicts the fact that $g\notin \tstar$. So $\ord{C_{\ostar}(g)}=4$, and $C_{\ostar}(g)=\ang{g}$. The orbit-stabiliser theorem implies that the conjugacy class of $g$ contains twelve elements, which must be the elements of $\mathcal{U}$. We thus conclude that there are two conjugacy classes in $\ostar$ of elements of order $4$, so $r_{0}(4)=2$. it also follows that if $g_{1}$ and $g_{2}$ are elements of $\ostar$ of order $4$ such that $g_{1}\in \tstar$ and $g_{2}\notin \tstar$ then $g_{1}$ and $g_{2}$ are representatives of these two conjugacy classes. If $g\in \ostar$ is of order $4$, then either it belongs to $\tstar$, and then $g^{-1}\in \tstar$, or it belongs to $\ostar\setminus \tstar$, and then $g^{-1}\in \ostar\setminus \tstar$. In both cases, it follows that $g$ and $g^{-1}$ are conjugate in $\ostar$. Thus $r_{1}(4)=2$, and hence $r_{2}(4)=2$ by \relem{r1dr2d}.

Finally, let $d=8$, let $g\in \ostar$ be of order $8$, and let $H$ be a Sylow $2$-subgroup that contains $g$. Then $\ord{H}=16$, and since $\ostar$ has no element of order $16$, it follows from \rerem{centre}(\ref{it:centre1}) that $H\cong \quat[16]$. The group $\quat[16]$ contains a unique cyclic subgroup of order $8$, and hence $H$ is the only Sylow $2$-subgroup that contains $g$. Since the Sylow $2$-subgroups are pairwise conjugate, it follows that the three cyclic subgroups of order $8$ are pairwise conjugate, and hence $r_{2}(8)=1$. The centraliser $C_{\ostar}(g)$ contains $\ang{g}$, and is cyclic by \relem{centbinpoly}, and the fact that $\ostar$ has no element of order greater than $8$ implies that $C_{\ostar}(g)=\ang{g}$. The orbit-stabiliser theorem implies the conjugacy class of $g$ contains $6$ elements, and since $\ostar$ contains $12$ elements of order $8$, we conclude that $r_{0}(8)=2$. 
Using the presentation~\reqref{presdic} of $\quat[16]$, we see that $ygy^{-1}=g^{-1}$ for all $y\in H\setminus \ang{g}$, so $g$ and $g^{-1}$ are conjugate, and it follows that $r_{1}(8)=2$ also. We also deduce that $g$ and $g^{3}$ are representatives of the two conjugacy classes of elements of order $8$.

\item Let $G=\istar$. First suppose that $d\in \brak{3,5}$, and let $g\in \istar$ be an element of order $d$. Then $C_{\ostar}(g)$ contains $\ang{g}$ and the unique element $\omega$ of $\istar$ of order $2$, and since $C_{\ostar}(g)$ is cyclic by \relem{centbinpoly}, we see as in the previous cases that $C_{\ostar}(g)=\ang{\omega g}$, and $\ord{C_{\ostar}(g)}=2d$. The orbit-stabiliser theorem then implies that the conjugacy class of $g$ contains $60/d$ elements. If $d=3$, this conjugacy class is the set of elements of order $3$, so $r_{0}(3)=1$, and hence $r_{1}(3)=r_{2}(3)=1$ by~\reqref{ineqr}. If $d=5$, the conjugacy class of $g$ contains $12$ elements, from which we conclude that there are two conjugacy classes $C_{1}$ and $C_{2}$ of elements of order $5$, so $r_{0}(5)=2$. The subgroups of $\istar$ of order $5$ are its Sylow $5$-subgroups, which are pairwise conjugate, so $r_{2}(5)=1$. This also implies that each such subgroup contributes two elements to each of $C_{1}$ and $C_{2}$. So if $g\in \ostar$ is of order $5$, it is conjugate to exactly one  element $h$ of $\ang{g} \setminus \brak{g}$. Note that $h\neq g^{2}$ (resp.\ $h\neq g^{-2}$) for otherwise $g$ would be conjugate to $g^{2}$ (resp.\ $g^{-2}$), then $g^{2}$ would be conjugate to $g^{-1}$ (resp.\ to $g$), and the conjugacy class of $g$ would contain at least three elements of $\ang{g}$, which is not possible. Hence $g$ is conjugate to $g^{-1}$ for every element $g\in \istar$ of order $5$, but is not conjugate to $g^{2}$ or to $g^{-2}$. This proves that the conjugacy class of $\brak{g,g^{-1}}$ is equal to that of $g$, and so $r_{1}(5)=2$, and that $g$ and $g^{2}$ are representatives of the two conjugacy classes of elements of order $5$. Once more, if $d\in \brak{3,5}$, then $r_{i}(2d)=r_{i}(d)$ for all $i\in \brak{0,1,2}$.

It remains to study the case $d=4$. Let $g\in \istar$ be an element of order $4$. Using once more the fact that $C_{\ostar}(g)$ contains $\ang{g}$ and is cyclic, we see that $C_{\ostar}(g)=\ang{g}$, and then that the conjugacy class of $g$ contains thirty elements, which is the number of elements of $\istar$ of order $4$. So $r_{0}(4)=1$, and $r_{1}(3)=r_{2}(3)=1$ by~\reqref{ineqr}.\qedhere
\end{enumerate}
\end{proof}

\section{Whitehead groups of the finite subgroups of $B_{n}(\St)$}\label{sec:whitehead}

If $G$ is a finite group, recall that its \emph{Whitehead group} $\wh{G}$ is a finitely-generated Abelian group, and so may be written in the form:
\begin{equation}\label{eq:whprod}
\wh{G}=\Z^r \oplus \skone{G},
\end{equation}
where $\skone{G}$ is isomorphic to the torsion subgroup of $\wh{G}$~\cite{Wall2}. The following proposition implies that to determine $\wh{G}$, where $G$ is a finite subgroup of $B_{n}(\St)$, it suffices to compute $r$.

\begin{prop}\label{prop:skonetriv}
Let $n\in \N$, and let $G$ be a finite subgroup of $B_{n}(\St)$. Then $\skone{G}$ is trivial.
\end{prop}

\begin{proof}
As we mentioned in \resec{classvc}, any finite subgroup $G$ of $B_{n}(\St)$ is cyclic, dicyclic or binary polyhedral. If $G$ is cyclic, dicyclic of order $8m$, $m\in \N$, or binary polyhedral the result follows from~\cite[Theorem~A, parts~(1),~(3),~(5),~(6) and~(7)]{U}. The only other possibility is when $G$ is dicyclic of order $4m$, where $m$ is odd. In this case, the Sylow $p$-subgroups of $G$ are cyclic, and from~\cite[page~20]{T}, $G$ admits a presentation of the Type~I groups of~\cite[Appendix]{U}. The result then follows from~\cite[Theorem~A, part~(2)]{U}.
\end{proof}

\begin{rem}
One may also prove \repr{skonetriv} by applying~\cite[Theorem~14.2(i) and Example~14.4]{O}.
\end{rem}

Consequently, if $G$ is a finite subgroup  of $B_{n}(\St)$, then by \req{whprod}, $\wh{G}$ is a free Abelian group, and it remains to calculate its rank. This is achieved in the following proposition.

\begin{prop}\label{prop:whbnS2}
Let $n\in \N$, let $G$ be a finite subgroup of $B_{n}(\St)$, and if $q\in \N$, let $\delta(q)$ denote the number of divisors of $q$. Then $\wh{G}\cong \Z^r$, where:
\begin{equation*}
r=
\begin{cases}
\left\lfloor \frac{m}{2} \right\rfloor+1-\delta(m) & \text{if $G\cong \Z_{m}$, $m\in\N$}\\
m+1-\delta(2m) & \text{if $G\cong \dic{4m}$, $m\geq 2$}\\
0 & \text{if $G\cong \tstar$}\\
1 & \text{if $G\cong \ostar$}\\
2 & \text{if $G\cong \istar$.}
\end{cases}
\end{equation*}
\end{prop}

\begin{proof}
Let $G$ be isomorphic to a finite subgroup of $B_{n}(\St)$. We recall once more that $G$ is cyclic, dicyclic or binary polyhedral. Let $r_{1}$ denote the number of conjugacy classes of unordered pairs $\brak{g,g^{-1}}$ in $G$, where $g\in G$, and let $r_{2}$ be the number of conjugacy classes of cyclic subgroups of $G$. By~\cite[page 39]{KL}, the rank $r$ of $\wh{G}$ is equal to $r_{1}-r_{2}$, and so: 
\begin{equation}\label{eq:sumr1r2}
r=\sum_{d \divides \ord{G}} \bigl( r_{1}(d)-r_{2}(d)\bigr).
\end{equation}
We treat the possibilities for $G$ separately.

\begin{enumerate}[(a)]
\item $G\cong \Z_{m}$, where $m\in \N$. Since $G$ is Abelian, $r_{1}$ is just the number of unordered pairs $\brak{g,g^{-1}}$ in $G$, where $g$ runs over the elements of $G$,  and $r_{2}$ is the number of cyclic subgroups of $G$. Since $g=g^{-1}$ if and only if $\ord{\ang{g}}\in \brak{1,2}$, we have that $r_{1}=\frac{m-1}{2}+1$ if $m$ is odd, and $r_{1}=\frac{m-2}{2}+2$ if $m$ is even. So $r_{1}=\left\lfloor \frac{m}{2} \right\rfloor+1$. Since the subgroups of $\Z_{m}$ are in bijection with the divisors of $m$, we have $r_{2}= \delta(m)$, so $r=r_{1}-r_{2}=\left\lfloor \frac{m}{2} \right\rfloor+1-\delta(m)$ as required.

\item $G\cong \dic{4m}$, where $m\geq 2$. Let $G=\ang{x} \coprod \ang{x}y$ be given by \req{presdic}. Since the elements of $\ang{x}y$ are of order $4$, it follows from \relem{r1dr2d} and \req{sumr1r2} that they do not contribute to $r$. So we just need to consider the contributions of the elements of $\ang{x}$ to $r_{1}$ and $r_{2}$. Using \req{presdic}, the conjugacy classes of the elements of $\ang{x}$ in $G$ are $\brak{x^{i},x^{-i}}$, where $0\leq i\leq m$. Since $\ang{x}$ is of order $2m$, as in the cyclic case, its elements contribute $m+1$ to the $r_{1}$-term, and $\delta(2m)$ to the $r_{2}$-term, and thus $r=m+1-\delta(2m)$.

\item If $G$ is binary polyhedral, the rank of $\wh{G}$ may be easily deduced using~\reqref{sumr1r2} and the tables of \repr{ribinpoly}.\qedhere
\end{enumerate}
\end{proof}

\section{$\widetilde{K}_{0}(\Z [G])$ for the finite subgroups of $B_{n}(\St)$}\label{sec:Kzero}

Let $G$ be a finite group. The calculation of $\widetilde{K}_{0}(\Z [G])$ is a difficult problem, even when the order of $G$ is small. It is known that $\widetilde{K}_{0}(\Z [G])$ is isomorphic to the \emph{ideal class group} $\operatorname{Cl}(\Z [G])$ of $\Z [G]$~\cite[Section~49.11]{CR2}. The following theorems summarise some results  about $\widetilde{K}_{0}(\Z [G])$ for certain finite groups.

\begin{thm}[{\cite[Corollary~50.17]{CR2}},~\cite{EH}]\label{th:kzerotrivial}
If $G$ is Abelian then $\widetilde{K}_{0}(\Z [G])$ is trivial if and only if $G$ is either cyclic of order $n$, $n\in \brak{1,2,\ldots,11,13,14,17,19}$, or is isomorphic to $\Z_{2}\oplus \Z_{2}$. 
If $G$ is non Abelian and  $\widetilde{K}_{0}(\Z [G])=1$ then $G$ is isomorphic to one of $\dih{2q}$, $q\geq 3$, $\an[4]$, $\sn[4]$ or $\an[5]$, $\dih{2q}$ being the dihedral group of order $2q$. 
\end{thm}

\begin{thm}[{\cite[Theorems~III and~IV, Corollary~10.12]{Sw3}}]\mbox{}\label{th:swan}
\begin{enumerate}[(a)]
\item $\widetilde{K}_{0}(\Z [\dic{4m}])\cong
\begin{cases}
\Z_{2} & \text{if $m\in\brak{2,3,4,5,7,8,11}$}\\
\Z_{2}^2 & \text{if $m=9$}\\
\Z_{2}^3 & \text{if $m\in\brak{6,10}$.}\\
\end{cases}$

\item $\widetilde{K}_{0}(\Z [\tstar])\cong \Z_{2}$, $\widetilde{K}_{0}(\Z [\ostar])\cong \Z_{2}^2$ and $\widetilde{K}_{0}(\Z [\istar])\cong \Z_{2}^3$.
\end{enumerate}
\end{thm}

In \resec{lower411}, we will determine the lower algebraic $K$-theory of the finite subgroups of $B_{n}(\St)$ for all $4\leq n\leq 11$. With this in mind, we now compute $\widetilde{K}_{0}(\Z [G])$ for some other finite cyclic groups. Before proving our results, we state the following result concerning the Bass cyclic units of the group ring $\Z[\Z_n]$.

\begin{thm}[{\cite[p.~403]{Ba2}}]\label{th:basscyc}
Let $G$ denote the cyclic group of order $n$. Let $n,k \in \N$, let $g\in \Z_{n}$ be an element of order $n$, and let $m$ be a multiple of $\phi(n)$. Then $k^m\equiv 1 \bmod{n}$. Further, the Bass cyclic units are defined by:
\begin{equation*}
 u_{k,m}(g)=\bigl(1+g+\cdots +g^{k-1}\bigr)^m+ \frac{1-k^m}{n}(1+g+\cdots + g^{n-1}),
\end{equation*}
where $k$ and $n$ are relatively prime, and they generate all the units of infinite order in $\Z[\Z_{n}]$.
\end{thm}

\begin{thm}\label{th:K0cyc}
Let $G=\Z_{n}$, where $n\in \brak{18,20,22}$. Then $\widetilde{K}_0(\Z[G])\cong \Z_3$ if $n\in \brak{18,22}$, and $\widetilde{K}_0(\Z[G])\cong \Z_2^{5}$ if $n=20$.
\end{thm}

\begin{proof}
Let $G=\Z_{n}$, where $n\in \brak{18,20,22}$. We make use of an appropriate Cartesian square and the associated Mayer-Vietoris sequence as follows. We begin with the Rim square associated to $\Z_{2}$:
\begin{equation}\label{eq:CS}
\begin{CD}
\Z[\Z_2]@>>> \Z\\
@VVV    @VVV\\
\Z@>>>\FF_2.
\end{CD}
\end{equation} 
By~\reqref{CS}, we obtain the following Cartesian square:
\begin{equation}\label{eq:CS1}
\begin{CD}
\Z[G]@>>> \Z[\Z_{n/2}]\\
@VVV    @VVV\\
\Z[\Z_{n/2}]@>>>\FF_2[\Z_{n/2}].
\end{CD}
\end{equation}
\repr{skonetriv} and~\reth{kzerotrivial} imply that $\widetilde{K}_0(\Z[\Z_{n/2}])=0$ and $SK_1(\Z[\Z_{n/2}])=0$, hence the Mayer-Vietoris sequence associated with~\reqref{CS1} becomes: 
\begin{equation*}
\hspace*{-2mm}\cdots\to U(\Z[\Z_{n/2}])\oplus U(\Z[\Z_{n/2}])\to U(\FF_2[\Z_{n/2}])\to \widetilde{K}_0(\Z[G])\to 0.
\end{equation*}
In the rest of this proof, $U(R)$ will denote the group of Bass cyclic units of an integral group ring $R$. We therefore need to understand the following homomorphism:
\begin{equation}\label{eq:coker}
U(\Z[\Z_{n/2}])\oplus U(\Z[\Z_{n/2}])\to U(\FF_2[\Z_{n/2}])
\end{equation}
that is induced by reduction modulo $2$. If $n=18$ (resp.\ $n=22$), the ring $\FF_2[\Z_{n/2}]$ is semi-simple, and
is isomorphic to $\FF_2\oplus\FF_2(\xi^3)\oplus \FF_2(\xi)$ (resp.\ to $\FF_2\oplus\FF_2(\xi)$), where $\xi$ is a primitive $(n/2)\up{th}$ root of unity, and is a root of the polynomial $x^6+x^3+1$ (resp.\ of $x^{10}+x^9+\cdots+x^2+x+1$). Both of these polynomials are irreducible in $\FF_2[x]$. Recall that $\FF_2(\xi)$ is a field with $64$ (resp.\ $1024$) elements~\cite{Cox}, and its group of units is cyclic of order $63$ (resp.\ of order $1023=31\ldotp 11\ldotp 3$). Suppose that $n=18$.
As we mentioned, we are taking $U(\Z[\Z_9])$ to be generated by the Bass cyclic units that are of infinite order in $\Z[\Z_9]$. These cyclic units are described by \reth{basscyc}, and are $u_{2,6},u_{4,6},u_{5,6}, u_{7,6}$ and $u_{8,6}$. The image of $u_{7,6}$ in $\FF_{2}(\xi)$ is:
\begin{align*}
(1+\xi+\cdots+\xi^{6})^{6}&= (\xi^{7}+\xi_{8})^{6}=\xi^{6}(1+\xi)^{6}=\xi^{6}(1+\xi^{2})^{3}\\
&=\xi^{6}(1+\xi^{2}+\xi^{4}+\xi^{6})=\xi^{6}+ \xi^{8}+\xi+\xi^{3}=1+\xi+\xi^{8}.
\end{align*}
So the image of $u_{7,6}^{2}$ in $\FF_{2}(\xi)$ is $1+\xi^{2}+\xi^{7}$, the image of $u_{7,6}^{6}$ in $\FF_{2}(\xi)$ is the image of $(u_{7,6}^{2})^{2} u_{7,6}^{2}$ in $\FF_{2}(\xi)$, which is equal to:
\begin{equation*}
(1+\xi^{4}+\xi^{5})(1+\xi^{2}+\xi^{7}) =1+\xi^{2}+\xi^{7}+ \xi^{4}+\xi^{6}+\xi^{2}+ \xi^{5}+\xi^{7}+\xi^{3}= \xi^{4}+ \xi^{5}.
\end{equation*}
Thus the image of $u_{7,6}^{7}$ in $\FF_{2}(\xi)$ is the image of $u_{7,6}^{6} u_{7,6}$ in $\FF_{2}(\xi)$, which is equal to:
\begin{align*}
(\xi^{4}+ \xi^{5})(1+\xi+\xi^{8})=\xi^{4}+\xi^{5}+\xi^{3}+\xi_{5}+\xi^{6}+\xi^{4}=1.
\end{align*}
Hence the image of $u_{7,6}$ in $\FF_{2}(\xi)$ is of order $7$, and the image of $u_{7,6} u_{8,6}$ in $\FF_{2}(\xi)$ is of order $21$. We now show that the three other cyclic units are each of order $21$.
\begin{enumerate}[(i)]
\item $u_{2,6}$: its image in $\FF_{2}(\xi)$ is $(1+\xi)^{6}=(1+\xi^{2})^{3}= 1+\xi^{2}+\xi^{4}+\xi^{6}$. So the image of $u_{2,6}^{3}$ in $\FF_{2}(\xi)$ is equal to the image of $u_{2,6}^{2}u_{2,6}$ in $\FF_{2}(\xi)$, which in turn is equal to:
\begin{multline*}
(1+\xi^{4}+\xi^{8}+\xi^{3})(1+\xi^{2}+\xi^{4}+\xi^{6})= 1+\xi^{2}+\xi^{4}+\xi^{6}+ \xi^{4}+\xi^{6}+\xi^{8}+\xi+\\
\xi^{8}+\xi+\xi^{3}+\xi^{5}+\xi^{3}+\xi^{5}+\xi^{7}+1= \xi^{2}+\xi^{7}.
\end{multline*}
Hence the image of $u_{2,6}^{6}$ in $\FF_{2}(\xi)$ is equal to the image of $(u_{2,6}^{3})^{2}$ in $\FF_{2}(\xi)$, which is equal to $\xi^{4}+\xi^{5}$, and the image of $u_{2,6}^{7}$ in $\FF_{2}(\xi)$ is equal to the image of $(u_{2,6}^{6})u_{2,6}$ in $\FF_{2}(\xi)$, which in turn is equal to:
\begin{equation*}
(1+\xi^{2}+\xi^{4}+\xi^{6})(\xi^{4}+\xi^{5})=\xi^{4}+\xi^{6}+\xi^{8}+\xi+ \xi^{5}+\xi^{7}+1+\xi^{2}=\xi^{3},
\end{equation*}
which is of order $3$. It follows that the image of $u_{2,6}$ in $\FF_{2}(\xi)$ is of order $21$.

\item The image of $u_{4,6}$ in $\FF_{2}(\xi)$ is:
\begin{align*}
(1+\xi+\xi^{2}+\xi^{3})^{6}&=(1+\xi^{2}+\xi^{4}+\xi^{6})^{3}\\
&=(1+\xi^{2}+\xi^{4}+\xi^{6})^{2}(1+\xi^{2}+\xi^{4}+\xi^{6})\\
&= (1+\xi^{4}+\xi^{8}+\xi^{3})(1+\xi^{2}+\xi^{4}+\xi^{6})\\
&= (\xi^{2}+\xi^{3}+\xi^{4})^{2}=\xi^{4}+\xi^{6}+\xi^{8}=\xi^{4}(1+\xi^{2}+\xi^{4}).
\end{align*}
Then the image of $u_{4,6}^{2}$ in $\FF_{2}(\xi)$ is $\xi^{8}(1+\xi^{4}+\xi^{8})= \xi^{8}+\xi^{3}+\xi^{7}=\xi^{3}(1+\xi^{4}+\xi^{5})$, and the image of $u_{4,6}^{3}=u_{4,6}^{2} u_{4,6}$ in $\FF_{2}(\xi)$ is:
\begin{align*}
\xi^{7}(1+\xi^{2}+\xi^{4})(1+\xi^{4}+\xi^{5})&= \xi^{7}(1+\xi^{4}+\xi^{5}+ \xi^{2}+\xi^{6}+\xi^{7}+ \xi^{4}+\xi^{8}+1)\\
&= \xi^{7}(\xi^{5}+\xi^{2}+\xi^{6}+\xi^{7}+\xi^{8})\\
&= \xi^{7}(\xi^{6}+\xi^{7})= \xi^{4}(1+\xi).
\end{align*}
Thus the image of $u_{4,6}^{6}=(u_{4,6}^{3})^{2}$ is $\xi(1+\xi^{7})$, and the image of $u_{4,6}^{7}=u_{4,6}^{6} u_{4,6}$ is:
\begin{align*}
\xi^{5}(1+\xi^{2}+\xi^{4})(1+\xi^{7})=\xi^{5}(1+\xi^{2}+\xi^{4}+ \xi^{7}+1+\xi^{2})=
\xi^{6}.
\end{align*}
It follows that the image of $u_{4,6}$ in $\FF_{2}(\xi)$ is of order $21$.

\item The image of $u_{5,6}$ in $\FF_{2}(\xi)$ is:
\begin{align*}
(1+\xi+\cdots+\xi^{4})^{6}&=(1+\xi^{2}+\xi^{4}+\xi^{6}+\xi^{8})^{3}= (\xi^{3}+\xi^{4}+ \xi^{5})^{3}= (1+\xi+\xi^{2})^{3}\\
&= (1+\xi+\xi^{2})^{2}(1+\xi+\xi^{2})= (1+\xi^{2}+\xi^{4})(1+\xi+\xi^{2})\\
&= (1+\xi+\xi^{2}+\xi^{2}+\xi^{3}+\xi^{4}+ \xi^{4}+\xi^{5}+\xi^{6})= \xi(1+\xi^{4}).
\end{align*}
So the image of $u_{5,6}^{3}$ in $\FF_{2}(\xi)$ is $\xi^{3}(1+\xi^{4}+\xi^{8}+\xi^{3})= \xi^{3}+x^{7}+\xi^{2}+\xi^{6}=1+\xi^{2}+\xi^{7}$, and the image of $u_{5,6}^{6}=(u_{5,6}^{3})^{2}$ is $1+\xi^{4}+\xi^{5}$. Hence the image of $u_{5,6}^{7}=u_{5,6}^{6}u_{5,6}$ is:
\begin{align*}
(1+\xi^{4}+\xi^{5})\xi(1+\xi^{4})&=\xi(1+\xi^{4}+\xi^{5}+ \xi^{4}+\xi^{8}+1)= \xi(\xi^{5}+\xi^{8})=\xi^{3}.
\end{align*}
It follows that the image of $u_{5,6}$ in $\FF_{2}(\xi)$ is of order $21$.
\end{enumerate}
We conclude that the image in $\FF_{2}(\xi)$ of the subgroup generated by the cyclic units is of order $21$. Hence the cokernel in~\reqref{coker} is of order $3$, and from that equation we see that $\widetilde{K}_0(\Z[\Z_{18}])\cong \Z_{3}$, thus proving the result in the case $n=18$. The proofs in the cases $n=20$ and $n=22$ are similar. First suppose that $n=22$. As above, we see that the Bass cyclic units are of the form $u_{k,10}$ for $k=2,3,4,\ldots,10$, and making use of a Mathematica~\cite{math}  routine written by Jos\'e Hernandez (CCM-UNAM), to whom we are grateful, one may check that the image in $\FF_2(\xi)$ of the subgroup generated by the cyclic units is of order $341$, and that the cokernel of~\reqref{coker} is of order $3$, so once more, $\widetilde{K}_0(\Z[\Z_{22}])\cong \Z_{3}$. Finally, suppose that $n=20$. In this case, the ring $\FF_2[\Z_{10}]$ is not semi-simple. It is isomorphic to $\FF_2[x]/(x^5-1)^2$, and its group of units is isomorphic to the direct product $U(\FF_2[\Z_{10}])=(\Z_2)^5\times \Z_3\times \Z_5$. Applying the Mathematica routine once more to the cyclic units $u_{3,4}, u_{7,4}$ and $u_{9,4}$, we see that the image of the group generated by these units in~\reqref{coker} is of order $15$, and hence the cokernel of~\reqref{coker} is isomorphic to $\Z_2^{5}$. This completes the proof of the theorem. 
\end{proof}

\section{$K_{-1}(\Z [G])$ for the finite subgroups of $B_{n}(\St)$}\label{sec:Kminusone}

Let $G$ be a finite subgroup of $B_{n}(\St)$. In order to determine $K_{-1}(\Z [G])$, we shall use the following special case of a result of Carter. Similar results have recently been obtained independently by B.~Magurn in~\cite{Mag} for generalised quaternion and binary polyhedral groups.

First we recall that a simple Artinian ring $A$ is isomorphic to $\mathcal{M}_{n}(D)$ for some positive integer $n$ and some skew field $D$. Further, $D$ is finite dimensional over its centre $E$, the dimension being a square $[D:E]$, and the \emph{Schur index} of $A$ is equal to $\sqrt{[D:E]}$~\cite[Section~27]{CR}. 
\begin{thm}[{\cite[Theorem~1]{C1}}]\label{th:carter}
Let $G$ be a finite group of order $q$. Then 
\begin{equation}\label{eq:carterkminusone}
K_{-1}(\Z [G])\cong \Z^r \oplus \Z_{2}^s,
\end{equation}
where $r$ is given by
\begin{equation}\label{eq:rankkminusone}
r=1-r_{\Q}+ \sum_{p\divides \ord{G}} \bigl( r_{\Q_{p}}-r_{\F[p]}\bigr),
\end{equation}
$r_{\Q}$ (resp.\ $r_{\Q_{p}}$, $r_{\F[p]}$) denotes the number of isomorphism classes of irreducible $\Q$- (resp.\ $\Q_{p}$-, $\F[p]$-) representations of $G$, and $s$ is equal to the number of simple components of $\Q[G]$ that have even Schur index $m$ but have odd local Schur indices $m_{Q}$ at every finite prime $Q$ of the centre which divides $q$.
\end{thm}
So to calculate $K_{-1}(\Z [G])$, we must determine the quantities $r_{F}$ for the various fields appearing in \req{rankkminusone}, as well as the number $s$. For the finite subgroups $G$ of $B_{n}(\St)$, we divide this calculation into two parts. In \resec{torsionkminusone}, we determine $r$, which yields the torsion of $K_{-1}(\Z [G])$. In \resec{rankkminusone}, we compute $s$, which is the rank of $K_{-1}(\Z [G])$. We then obtain $K_{-1}(\Z [G])$ from~\reqref{carterkminusone}.

\subsection{Torsion of $K_{-1}(\Z [G])$ for finite subgroups of $B_{n}(\St)$}\label{sec:torsionkminusone}

Let $G$ be a finite subgroup of $B_{n}(\St)$, and let $s$ be as defined in \req{carterkminusone}. As remarked in~\cite[page~1928]{C1}, a consequence of \reth{carter} is that $K_{-1}(\Z [G])$ is torsion free if $G$ is Abelian. In particular, if $G$ is cyclic, then $s=0$. If $G$ is non cyclic, then as we shall see, $K_{-1}(\Z [G])$ may have torsion. Although \req{carterkminusone} clearly allows for this possibility, this appears to be a new phenomenon, and contrasts with the calculations given in~\cite{LMO,LO} for example. We thus require new techniques to calculate the torsion of $K_{-1}(\Z [G])$. If $G$ is dicyclic, we make use of results due to Yamada concerning the computation of the (local) Schur indices of the simple components of $\Q[G]$~\cite{Y}. If $G$ is binary polyhedral, then one may apply induction/restriction techniques and the Mackey formula. 

Assume first that $G\cong \dic{4m}$ is dicyclic, where $m\geq 3$ is odd. If $m$ is an odd prime then we determine $K_{-1}(\Z [\dic{4m}])$. In principle, our method should apply to any odd value of $m$, not just for prime values. If $m$ is odd, the Wedderburn decomposition over $\Q$ of the algebra $\Q [\dic{4m}]$ is given in~\cite[Example~7.40]{CR}:
\begin{align}
\Q [\dic{4m}] & \cong \Q [\dih{2m}] \oplus\, \Q(i) \oplus \Biggl( \bigoplus_{d_{0}\mid m,\; d_{0}>1} \mathbb{H}_{2d_{0}}\Biggr)\notag\\
& \cong \Q^2 \oplus \Biggl(\bigoplus_{d \mid m, \; d>2} \mathcal{M}_{2}\left( \Q\left(\zeta_{d}+\zeta_{d}^{-1} \right)\right)\Biggr) \oplus \Q(i) \oplus \Biggl(\bigoplus_{d_{0}\mid m,\; d_{0}>1} \mathbb{H}_{2d_{0}}\Biggr),\label{eq:decompQG}
\end{align}
where $\zeta_{d}$ is a primitive $d\up{th}$ root of unity, and
\begin{equation}\label{eq:defhd}
\mathbb{H}_{d}=E_{d}\oplus E_{d}i \oplus E_{d}j \oplus E_{d}k
\end{equation}
is the quaternion skew field with centre $E_{d}=\Q \left(\zeta_{d}+\zeta_{d}^{-1}\right)$. In particular, if $m=\mu$ is prime then
\begin{equation}\label{eq:weddic4p}
\Q [\dic{4\mu}]\cong \Q^2 \oplus \mathcal{M}_{2}\left( E_{\mu} \right) \oplus \Q(i) \oplus \mathbb{H}_{2\mu}.
\end{equation}
Note that the number of components in \req{weddic4p} is equal to the number of conjugacy classes of cyclic subgroups of $\dic{4\mu}$, and that the components are in one-to-one correspondence with the irreducible $\Q$-representations of $\dic{4\mu}$. The first four components of \req{weddic4p} are matrix rings over fields, and so their Schur index is equal to one. By \req{carterkminusone}, the torsion of $K_{-1}(\Z [\dic{4\mu}])$ is then either trivial or equal to $\Z_{2}$ depending on the Schur and local Schur indices of the remaining component $\mathbb{H}_{2\mu}$. We now determine precisely this torsion using results of Yamada~\cite{Y2,Y}. 

\begin{prop}\label{prop:torsionkminusone}
If $\mu$ be an odd prime, the torsion of $K_{-1}(\Z[\dic{4\mu}])$ is trivial if $\mu\equiv 3 \bmod 4$, and is equal to $\Z_{2}$ if $\mu\equiv 1 \bmod 4$.
\end{prop}

\begin{proof}
We apply the results of~\cite{Y2,Y}, and refer the reader to these papers for the notation used in this proof. If $n\in \N$ and $w\in \Z$ is coprime with $n$, then $w \operatorname{mod}^{\times} n$ will denote $w$ as an element of the multiplicative group of integers modulo $n$. With the notation of~\cite[Proposition~4]{Y}, we have $m=2\mu$, $r=2\mu-1$, $s=2$, $h=\mu$ and $u$ is the order of $2\mu-1 \operatorname{mod}^{\times} 2\mu$, so $u=s=2$. From~\cite[Example~3, Section~6]{Y}, there are representations of $\dic{4\mu}$ of the form $U_{\alpha,0}^{(2)}$, where $0\leq \alpha \leq 2\mu-1$. Such representations are defined in~\cite[equation~(8), p.~214]{Y} and induced by linear characters. Using~\cite[Proposition~5]{Y}, the representation $U_{1,0}^{(2)}$ gives rise to an irreducible representation of $\Q[\dic{4\mu}]$, and the last part of~\cite[Example~3, Section~6]{Y} implies that its Schur index is equal to two. Since the Schur index of each of the first four components of \req{weddic4p} is equal to one, it follows that the simple component $\mathbb{H}_{2\mu}$ of $\Q[\dic{4\mu}]$ corresponds to $U_{1,0}^{(2)}$. 

We now apply~\cite[Proposition~9]{Y} to $U_{1,0}^{(2)}$. Within our framework, the enveloping algebra $\operatorname{env}_{\Q}\left(U_{1,0}^{(2)}\right)$ with respect to $\Q$ is isomorphic to the simple component $\mathbb{H}_{2\mu}$, and the centre $E_{2\mu}$ of $\mathbb{H}_{2\mu}$ is isomorphic to $\Q\left(\chi_{1,0}^{(2)}\right)$, $\chi_{1,0}^{(2)}$ being the character of $U_{1,0}^{(2)}$~\cite[Introduction]{Y2}. With the notation of~\cite[Proposition~9]{Y}, we have $d_{1}=\frac{2\mu}{\gcd{(2\mu,1)}}=2\mu$ and $v_{1}=\frac{2\mu}{\gcd{(2\mu,\mu)}}=2$. Let $\mathfrak{p}$ be a finite prime of the centre $E_{\mu}$ of $\mathbb{H}_{2\mu}$ that divides $4\mu$. Then $\mathfrak{p}$ divides $2\mu$, and since $\mu$ is odd, $\mathfrak{p}$ divides $p$, where $p\in \brak{2,\mu}$. We distinguish these two possibilities, the notation being that of~\cite[Proposition~9]{Y}.
\begin{enumerate}[(a)]
\item Suppose that $\mathfrak{p}$ divides $2$. Then we have $p=2$, $b=z=1$, $a=1$ and $t'$ is the order of $2\mu-1 \operatorname{mod}^{\times} \mu$, so $t'=2$. Thus $e_{\mathfrak{p}}=1$, and hence $c_{\mathfrak{p}}=\Lambda_{\mathfrak{p}}=1$.

\item Suppose that $\mathfrak{p}$ divides $\mu$. Then $p=\mu$, $b=0$, $z=2$, $a=1$, $\ang{2\mu-1 \operatorname{mod}^{\times} 2}=\ang{\mu \operatorname{mod}^{\times} 2}=\brak{1}$, $f=\widetilde{f}=t'=1$, $q=\mu$, $e_{\mathfrak{p}}=2$, $c_{\mathfrak{p}}=\gcd{(2,\mu-1)}=2$ and
\begin{equation*}
\Lambda_{\mathfrak{p}}=\frac{2}{\gcd{\left(2,\frac{\mu-1}{2}\right)}}=
\begin{cases}
1 & \text{if $\mu\equiv 1 \bmod 4$}\\
2 & \text{if $\mu\equiv 3 \bmod 4$}
\end{cases}
\end{equation*}
by~\cite[Proposition~9(II)]{Y}.
\end{enumerate}

Thus if $\mu\equiv 1 \bmod 4$, the simple component $\mathbb{H}_{2\mu}$ of $\Q [G]$ whose Schur index is equal to two satisfies the property that its local Schur indices at every finite prime of the centre are odd. Hence the integer $s$ of \req{carterkminusone} is equal to one, so the torsion of $K_{-1}(\Z[G])$ is $\Z_{2}$. If $\mu\equiv 3 \bmod 4$ then $\Lambda_{\mathfrak{p}}=2$ for any finite prime $\mathfrak{p}$ that divides $\mu$, so $s=0$, and hence $K_{-1}(\Z[G])$ is torsion free.\qedhere
\end{proof}

As another example, we calculate the torsion of $K_{-1}(\Z[\dic{4m}])$ in the case where $m$ is a power of $2$ (so $\dic{4m}$ is a generalised quaternion group).

\begin{prop}\label{prop:torsquat}
The torsion of $K_{-1}(\Z[\quat[2^k]])$ is trivial if $k=3$, and is equal to $\Z_{2}$ if $k\geq 4$.
\end{prop}

\begin{proof}
Let $k\geq 3$. Then $\dic{2^{k}}=\quat[2^{k}]$. From~\cite[Example 7.40, case~1]{CR},                   
\begin{equation}\label{eq:wedderquat}
\Q [\quat[2^k]]\cong \Q [\dih{2^{k-1}}]\oplus \mathbb{H}_{2^{k-1}},
\end{equation}
where $\mathbb{H}_{2^{k-1}}$ is the quaternion skew field defined by \req{defhd}. Using~\cite[Example~7.39]{CR}, each simple component of $\Q [\dih{2^{k-2}}]$ is a matrix ring over a field, and so its Schur index is equal to one. As in the proof of \repr{torsionkminusone}, one may show that the Schur index of the remaining simple component $\mathbb{H}_{2^{k-1}}$ of \req{wedderquat} is equal to two, and that this component corresponds to the irreducible representation $U_{1,0}^{(2)}$. To study the local Schur index $\Lambda_{\mathfrak{p}}$ of each finite prime $\mathfrak{p}$ dividing the centre $E_{2^{k-1}}$ of the simple component $\mathbb{H}_{2^{k-1}}$, we again apply~\cite[Proposition~9]{Y2}. With the same notation, we have $m=2^{k-1}$, $r=2^{k-1}-1$, $u=s=2$, $h=2^{k-2}$, $d_{1}=2^{k-1}$ and $v_{1}=2$. If $\mathfrak{p}$ does not divide $2^{k-1}$ then $\Lambda_{\mathfrak{p}}=1$ by~\cite[Proposition~9(I)]{Y2}. So suppose that $\mathfrak{p}$ divides $2^{k-1}$. With the notation of~\cite[Proposition~9(II)]{Y2}, $b=z=1$ and $p=2$. If $k=3$ then we are in the exceptional case of~\cite[Proposition~9(II)]{Y2}, so $\Lambda_{\mathfrak{p}}=2$. Thus there exists a finite prime of the centre $E_{2^{k-1}}$ of $\mathbb{H}_{2^{k-1}}$ dividing $2^k$ with even local Schur index, and it follows from \reth{carter} that the torsion of $K_{-1}(\Z[\quat])$ is trivial. Assume then that $k\geq 4$. So $f=\widetilde{f}=t'=1$ and $e_{\mathfrak{p}}=q=2$, thus $c_{\mathfrak{p}}=\Lambda_{\mathfrak{p}}=1$. Then the simple component $\mathbb{H}_{2^{k-1}}$ of $\Q [\quat[2^k]]$ whose Schur index is equal to two satisfies the property that its local Schur indices at every finite prime of the centre dividing $2^k$ are odd. Hence the integer $s$ of \req{carterkminusone} is equal to one, and thus the torsion of $K_{-1}(\Z[\quat[2^k]])$ is equal to $\Z_{2}$ as required.
\end{proof}

Now let $G$ be a binary polyhedral group. We recall that a group is said to be \emph{$2$-hyper-elementary} if it is a semi-direct product of a cyclic normal subgroup of odd order and a $2$-group. Since $G$ is not itself $2$-hyper-elementary, induction/restriction techniques may be used to calculate the torsion of $K_{-1}(\Z [G])$. 

\begin{prop}\label{prop:torsbinpoly}
The torsion of $K_{-1}(\Z [G])$ is trivial if $G\cong\tstar$, and is equal to $\Z_{2}$ if $G\cong \ostar$ or $G \cong \istar$.
\end{prop}

\begin{proof}
Let $G$ be a binary polyhedral group. Applying~\cite[Theorem~3(iii) and page~1936]{C1}, we have the composition 
\begin{equation}\label{eq:resind}
\oplus_{H}\, K_{-1}(\Z[H])\stackrel{\mathsf{ind}}{\to} K_{-1}(\Z[G]) \stackrel{\mathsf{res}}{\to} \oplus_{H}\, K_{-1}(\Z[H]),
\end{equation}
where $\mathsf{ind}$ and $\mathsf{res}$ are the usual induction and restriction maps that are surjective and injective respectively when restricted to the corresponding torsion subgroups, and $H$ runs over the conjugacy classes of the $2$-hyper-elementary subgroups of $G$~\cite[Theorem~3(iii) and p.~1936]{C1}. Restricting to these torsion subgroups, we see that the torsion of $K_{-1}(\Z[G])$ injects into that of $\oplus_{H}\, K_{-1}(\Z[H])$. The non-trivial $2$-hyper-elementary subgroups of $\tstar$ are $\Z_{2}$, $\Z_{3}$, $\Z_{4}$, $\Z_{6}$ and $\quat$, those of $\ostar$ are $\Z_{2}$, $\Z_{3}$, $\Z_{4}$, $\Z_{6}$, $\Z_{8}$, $\dic{12}$, $\quat$ and $\quat[16]$, and those of $\istar$ are $\Z_{2}$, $\Z_{4}$, $\Z_{6}$, $\Z_{10}$, $\quat$, $\dic{12}$ and $\dic{20}$~(see~\cite[Lemma~14.3]{Sw3} and~\cite[Proposition~85]{GG2}). If $m\in \N$, the group algebra $\Q[\Z_{m}]$ splits~\cite[Example~7.38]{CR}, so the torsion of $K_{-1}(\Z [\Z_{m}])$ is trivial~\cite[page~1928]{C1}. Further, by Propositions~\ref{prop:torsionkminusone} and~\ref{prop:torsquat}, the torsion of $K_{-1}(\Z [\quat])$ and of $K_{-1}(\Z [\dic{12}])$ is also trivial, and setting $L=\quat[16]$ (resp.\ $L=\dic{20}$) if $G=\ostar$ (resp.\ $G=\istar$), the torsion of $K_{-1}(\Z[L])$ is $\Z_{2}$. The injectivity of $\textsf{res}$ in \req{resind} implies that the torsion of $K_{-1}(\Z [\tstar])$ is trivial, which gives the result in this case. 

So let $G=\ostar$ or $\istar$, and let $L$ be as defined above. Now $G$ possesses a single conjugacy class of subgroups isomorphic to $L$~\cite[Lemma~14.3]{Sw3}, and since $L$ is the only subgroup of $G$ for which the torsion of $K_{-1}(\Z[L])$ is non trivial, we need only to consider the restriction of \req{resind} to the factor $H=L$:
\begin{equation}\label{eq:lgl}
K_{-1}(\Z[L])\stackrel{\mathsf{ind}}{\to} K_{-1}(\Z[G]) \stackrel{\mathsf{res}}{\to} K_{-1}(\Z[L]).
\end{equation}
It thus suffices to show that the restriction of~\reqref{lgl} to the corresponding torsion subgroups is the identity. Now $K_{-1}(\cdot)$ is a Mackey functor~\cite[Theorem 11.2]{O}, so we may apply Mackey's formula that describes the composition~\reqref{lgl} as the sum of the maps:
\begin{equation}\label{eq:mackey}
K_{-1}(\Z[L])\stackrel{\mathsf{res}}{\to} K_{-1}(\Z[x_{i}^{-1}Lx_{i}\cap L]) \stackrel{c_{x_{i}}}{\to} K_{-1}(\Z[L]),
\end{equation}
where $G=\coprod Lx_{i}L$ is a double coset decomposition of $G$, and the map $c_{x_{i}}$ is induced by the homomorphism $x_{i}^{-1}Lx_{i}\cap L \to L$ defined by $y\longmapsto x_{i}yx_{i}^{-1}$~\cite[Section~11a]{O}. Let $N_{G}(L)$ denote the normaliser of $L$ in $G$. If $x_{i}\notin N_{G}(L)$ then the torsion of $K_{-1}(\Z[x_{i}^{-1}Lx_{i}\cap L])$ is trivial, and the corresponding map~\reqref{mackey} contributes zero to the restriction of~\reqref{lgl} to the torsion subgroups. If on the other hand, $x_{i}\in N_{G}(L)$, the corresponding map~\reqref{mackey} is an isomorphism. Now $L\subset N_{G}(L)\subset G$, and since $L$ is not normal in $G$ and $G$ has no proper subgroup that strictly contains $L$~\cite[Proposition~85]{GG2}, it follows that $N_{G}(L)=L$. So there is only one double coset representative $x_{i}$ that belongs to $N_{G}(L)$, and for this $x_{i}$, it follows that the restriction of~\reqref{lgl} to the torsion subgroups is equal to the restriction of the isomorphism~\reqref{mackey} to the torsion subgroups. Since the torsion of $K_{-1}(\Z[L])$ is $\Z_{2}$, the same conclusion holds for $K_{-1}(\Z[G])$.
\end{proof}

\begin{rems}\mbox{}
\begin{enumerate}[(a)]
\item The induction/restriction arguments in the proof of \repr{torsbinpoly} were inspired by those given in~\cite[Paragraph~14]{Sw3} for the $\widetilde{K}_{0}$-groups.

\item Let $G=\ostar$ or $\istar$. We sketch an alternative proof of the fact that $K_{-1}(\Z[G])$ has non-trivial torsion that uses~\cite[Proposition~4.11]{Sw3}. The embedding of $G$ in the Hamilton quaternions $\mathbb{H}$~\cite[Chapter~7]{Co} induces an algebra homomorphism $\map{\psi_{G}}{\Q[G]}[\mathbb{H}]$. By~\cite[Proposition~4.11 and its proof]{Sw3}, $\psi_{G}(\Z[G])$ is a maximal order $\Gamma_{G}$ that is completely described in~\cite[page~79]{Sw3}, from which one may prove that $\im{\psi_{G}}$ is equal to $\mathbb{H}_{d}$, where $d=8$ (resp.\ $d=5$) if $G=\ostar$ (resp.\ $G=\istar$), in other words, $\mathbb{H}_{d}$ appears as a factor in the Wedderburn decomposition of $\Q[G]$. On the other hand, from \req{wedderquat} (resp.\ \req{weddic4p}), we know that $\mathbb{H}_{d}$ also appears in the Wedderburn decomposition of $\Q[\quat[16]]$ (resp.\ $\Q[\dic{20}]$), and from the proof of \repr{torsquat} (resp.\ \repr{torsionkminusone}), that it contributes a $\Z_{2}$-term to the torsion of $\K_{-1}(\Z [\quat[16]])$ (resp.\ $\K_{-1}(\Z [\dic{20}])$). It follows then from~\cite[Theorem~1]{C1} that $K_{-1}(\Z[G])$ has non-trivial torsion.

\item Using the \texttt{GAP}~package \texttt{Wedderga}~\cite{gap}, one may obtain the complete Wedderburn decomposition for the binary polyhedral groups:
\begin{align}
\Q [\tstar] &\cong \Q\oplus \Q(\zeta_{3}) \oplus \mathcal{M}_{3}(\Q) \oplus\mathbb{H}_{4}\oplus \mathbb{H}(\Q(\zeta_{3}))\label{eq:weddertstar}\\
\Q [\ostar] &\cong \Q^{2}\oplus \mathcal{M}_{2}(\Q) \oplus 2\mathcal{M}_{3}(\Q) \oplus \mathbb{H}_{8} \oplus \mathcal{M}_{2}(\widehat{\mathbb{H}}),\,\text{and}\label{eq:wedderostar}\\
\Q [\istar] &\cong \Q\oplus \mathcal{M}_{4}(\Q) \oplus \mathbb{H}_{5} \oplus \mathcal{M}_{2}(\widehat{\mathbb{H}}) \oplus \mathcal{M}_{5}(\Q)\oplus \mathcal{M}_{3}(\mathbb{H}(\Q))\oplus \!\mathcal{M}_{3}(\Q(\sqrt{5})),\label{eq:wedderistar}
\end{align}
where $\widehat{\mathbb{H}}$ is the quaternion algebra $(-1,-3)/\Q$. This algebra admits a basis $\brak{1,i,j,k}$ as a $\Q$-vector space, and the algebra multiplication satisfies $ij=-ji=k$, $i^{2}=-1$ and $j^{2}=-3$. Somewhat surprisingly, we were not able to find the decompositions~\reqref{weddertstar}--\reqref{wedderistar} in the literature.
\end{enumerate}
\end{rems}

In order to prove \reth{tablas} and to obtain Table~\ref{tab:finite} given on page~\pageref{tab:finite}, we will need to calculate $K_{-1}(\Z [G])$ for some other dicyclic groups, namely $G=\dic{4\mu}$, where $\mu\in \brak{6,9,10}$. We now compute the torsion of $K_{-1}(\Z[\dic{4\mu}])$ for $\mu=9$, as well as the case where $\mu=2\tau$, where $\tau$ is an odd prime, which includes the cases $\mu=6$ and $\mu=10$.

\begin{prop}\label{prop:torsiondicextra}\mbox{}
\begin{enumerate}[(a)]
\item If $\mu=2\tau$, where $\tau$ is an odd prime, then the torsion of $K_{-1}(\Z [\dic{4\mu}])$ is $\Z_{2}$.
\item The group $K_{-1}(\Z[\dic{36}])$ is torsion free. 
\end{enumerate}
\end{prop}

\begin{proof}\mbox{}
\begin{enumerate}[(a)]
\item Let $\mu=2\tau$, where $\tau$ is an odd prime. From~\cite[Example~7.40]{CR} or~\cite[pp.~75--76]{Sw3}, and using~\reqref{decompQG} and the notation of \resec{torsionkminusone}, we have:
\begin{align}
\Q [\dic{4\mu}] &\cong \Q [\dic{8\tau}]\notag\\
&\cong \Q^4 \oplus \Biggl(\,\bigoplus_{d \in \brak{\tau,2\tau}} \mathcal{M}_{2}\left( \Q\left(\zeta_{d}+\zeta_{d}^{-1} \right)\right)\Biggr) \oplus \mathbb{H}_{4} \oplus \mathbb{H}_{4\tau}.\label{eq:wedder4mu}
\end{align} 
The first three factors of \req{wedder4mu} are matrix rings over fields, so their Schur index is equal to one. Further, the factor $\mathbb{H}_{4}$ appears in the Wedderburn decomposition of the $\Q$-algebra $\Q[\tstar]$, and since $K_{-1}(\Z[\tstar])$ is torsion free by~\repr{torsbinpoly}, $\mathbb{H}_{4}$ does not contribute to the torsion of $K_{-1}(\Z[\dic{4\mu}])$. It remains to determine the Schur and local Schur indices of the remaining factor $\mathbb{H}_{4\tau}$. Once more, 
we follow the proof of \repr{torsionkminusone}, and we use the results and notation of~\cite{Y2,Y}, taking $m=4\tau$, $r=4\tau-1$, $h=2\tau$, and $u=s=2$. Using~\cite[Proposition~5]{Y}, the representation $U_{1,0}^{(2)}$ gives rise to an irreducible representation of $\Q[\dic{4\mu}]$, and the last part of~\cite[Example~3, Section~6]{Y} implies that its Schur index is equal to two. Since the Schur index of each of the first four components of \req{wedder4mu} is equal to one, it follows that the simple component $\mathbb{H}_{4\tau}$ of $\Q[\dic{4\mu}]$ corresponds to $U_{1,0}^{(2)}$. With the notation of~\cite[Proposition~9]{Y}, we have $d_{1}=4\tau$ and $v_{1}=2$. Let $\mathfrak{p}$ be a finite prime of the centre of $\mathbb{H}_{4\tau}$ that divides $4\tau$. Then $\mathfrak{p}$ divides $2$ or $\tau$. If $\mathfrak{p} \divides 2$, then $p=2$. We are not in the exceptional case of~\cite[Proposition~9(II)]{Y} since the order of $r$ in $\Z_{\tau}^{\ast}$ is equal to $2$. Further, $a=2$, and $t'$ is equal to the order of $4\tau-1$ in $\Z_{\tau}^{\ast}$, so $t'=2$. If $\mathfrak{p} \divides \tau$, then $p=\tau$, $a=1$, and $t'$ is equal to the order of $4\tau-1$ in $\Z_{4}^{\ast}$, so $t'=2$ also. So in both cases $e_{\mathfrak{p}}=s/t'=1$, hence $c_{\mathfrak{p}}=\Lambda_{\mathfrak{p}}=1$. Thus the simple component $\mathbb{H}_{4\tau}$ of $\Q [\dic{4\mu}]$ whose Schur index is equal to two satisfies the property that its local Schur indices at every finite prime of the centre are odd. Hence the integer $s$ of \req{carterkminusone} is equal to one, and therefore the torsion of $K_{-1}(\Z[\dic{4\mu}])$ is $\Z_{2}$.

\item By~\reqref{decompQG}, the Wedderburn decomposition of the $\Q$-algebra $\Q[\dic{36}]$ is given by:
\begin{equation}\label{eq:wedder36}
\Q[\dic{36}]\cong \Q^{2} \oplus \mathcal{M}_{2}(E_{3})\oplus \mathcal{M}_{2}(E_{9}) \oplus \Q(i) \oplus \mathbb{H}_{6} \oplus \mathbb{H}_{18}.
\end{equation} 
The first four factors of \req{wedder36} are matrix rings over fields, so their Schur index is equal to one. Further, the factor $\mathbb{H}_{6}$ also appears in the Wedderburn decomposition of the $\Q$-algebra $\Q[\dic{12}]$, and since $K_{-1}(\Z[\dic{12}])$ is torsion free by~\repr{torsionkminusone}, $\mathbb{H}_{6}$ does not contribute to the torsion of $K_{-1}(\Z[\dic{36}])$. It thus suffices to determine the Schur and local Schur indices of the remaining factor $\mathbb{H}_{18}$. Following the proof of \repr{torsionkminusone}, we obtain $m=d_{1}=18$, $\mu=h=9$, $r=17$ and $u=s=v_{1}=2$, and the representation $U_{1,0}^{(2)}$ gives rise to an irreducible representation of $\Q[\dic{36}]$ whose Schur index is equal to two. Since the Schur index of each of the first four components of \req{weddic4p} is equal to one, it follows that the simple component $\mathbb{H}_{18}$ of $\Q[\dic{36}]$ corresponds to $U_{1,0}^{(2)}$. If $\mathfrak{p}$ is a finite prime of the centre $E_{9}$ of $\mathbb{H}_{18}$ that divides $36$, then $\mathfrak{p}$ divides $6$, and hence $\mathfrak{p}$ divides $p$, where $p\in \brak{2,3}$. If $\mathfrak{p} \divides 3$, then $p=q=3$, $b=0$, $z=a=e_{\mathfrak{p}}=c_{\mathfrak{p}}=2$, $t'=\widetilde{f}=f=1$, and $\Lambda_{\frak{p}}=2$.
Thus the Schur index of the simple component $\mathbb{H}_{18}$ of the decomposition~\reqref{wedder36} of $\Q[\dic{36}]$ is equal to $2$, but its local Schur indices at every finite prime of the centre of $\mathbb{H}_{18}$ are not always odd. It follows from \reth{carter} that $K_{-1}(\Z[\dic{36}])$ is torsion free.\qedhere
\end{enumerate}
\end{proof}

\subsection{The rank of $K_{-1}(\Z [G])$ for the finite subgroups of $B_{n}(\St)$}\label{sec:rankkminusone}

Let $G$ be a finite subgroup of $B_{n}(\St)$. To calculate the rank of $K_{-1}(\Z [G])$, we shall apply \req{rankkminusone}. In each case, we will thus need to calculate the number $r_{F}$ of distinct irreducible $F[G]$-modules, where $F$ is equal respectively to $\Q$, $\Q_{p}$ and $\F[p]$. Before doing so, we recall the requisite theory (see \cite[pages~492 and~508]{CR} or \cite[pages~25--26]{O}).

Let $F$ be a field of characteristic $p\geq 0$, where $p$ is prime if $p>0$. If $G$ is a finite group of exponent $m$, let:
\begin{equation*}
\widehat{m}=\begin{cases}
m & \text{if $p=0$}\\
m/p^{a} & \text{if $p>0$, where $a$ is the largest power of $p$ that divides $m$.}
\end{cases}
\end{equation*}
Let $F(\zeta_{\widehat{m}})$ be a field extension of $F$ by a primitive $\widehat{m}\up{th}$ root of unity, which we denote by $\zeta_{\widehat{m}}$. Then $F(\zeta_{\widehat{m}})$ is a Galois extension of $F$, whose Galois group, denoted by $\operatorname{\text{Gal}}(F(\zeta_{\widehat{m}})/F)$, is given by:
\begin{equation*}
\operatorname{\text{Gal}}(F(\zeta_{\widehat{m}})/F)\!=\!\setr{\map{\phi}{F(\zeta_{\widehat{m}})}[F(\zeta_{\widehat{m}})]\!}{\text{$\phi$ is an automorphism and $\phi(z)=z$ for all $z\in F$}}\!.
\end{equation*}
Each automorphism $\sigma\in \operatorname{\text{Gal}}(F(\zeta_{\widehat{m}})/F)$ is uniquely determined by its action on $\zeta_{\widehat{m}}$, and is given by $\sigma(\zeta_{\widehat{m}})=\zeta_{\widehat{m}}^t$, where $t$ is an integer that is uniquely defined modulo $\widehat{m}$. Hence $t$ corresponds to an element of the multiplicative group of units $\Z_{\widehat{m}}^{\ast}$, and there is an injective group homomorphism:
\begin{equation}\label{eq:defphi}
\left\{ \begin{aligned}
\phi \colon\thinspace \operatorname{\text{Gal}}(F(\zeta_{\widehat{m}})/F) &\to \Z_{\widehat{m}}^{\ast}\\
\sigma &\longmapsto t,
\end{aligned}\right.
\end{equation}
defined by $\phi(\sigma)=t$. We now recall the definition of $F$-conjugacy class.
\begin{enumerate}
\item If $f,g$ are elements of $G$, we say that they are \emph{$F$-conjugate} if there exists $t\in \im{\phi}$ and $\alpha\in G$ such that $f^t=\alpha g \alpha^{-1}$. The $F$-conjugacy relation is an equivalence relation on $G$, and the $F$-equivalence class of $f$ in $G$ will be denoted by $[f]_{F}$. \item Let
\begin{equation*}
G_{p}'=\set{g\in G}{\gcd{(p,\ordelt{g})}=1},
\end{equation*}
be the set of \emph{$p$-regular elements} of $G$, where $\ordelt{g}$ denotes the order of $g\in G$. An $F$-conjugacy class of $G$ is said to be \emph{$p$-regular} if it is contained in $G_{p}'$.
\end{enumerate}
If $f\in G$ then we denote its usual conjugacy class by $[f]$. 
\begin{rems}\label{rem:fconjclass}
\begin{enumerate}\mbox{}
\item It follows from the definition that
\begin{equation}\label{eq:conjclass}
[f]_{F}= \bigcup_{t\in \im{\phi}} \left[f^t\right],
\end{equation}
in other words, an $F$-conjugacy class is a union of normal conjugacy classes. In particular, $[f]_{F}\supset [f]$. Further, the number of $F$-conjugacy classes of elements of order $n$ is bounded above by the number of usual conjugacy classes of elements of order $n$.

\item\label{it:fconjclassii} If $F=\Q$ then $\phi$ is an isomorphism~\cite[Theorem~1.5]{O}, and $f,g\in G$ are $F$-conjugate if and only if $\ang{f}$ and $\ang{g}$ are conjugate subgroups of $G$.
\end{enumerate}
\end{rems}

By the Witt-Berman Theorem, we have the following result that will be used to compute $r_{F}$ for our groups.
\begin{thm}[{\cite[Theorems~21.5 and~21.25]{CR}}]\label{th:wittberman}
Let $G$ be a finite group, and let $F$ be a field of characteristic $p\geq 0$, where $p$ is prime if $p>0$. 
\begin{enumerate}[(a)]
\item If $p=0$, then $r_{F}$ is equal to the number of $F$-conjugacy classes in $G$.
\item If $p>0$, then $r_{F}$ is equal to the number of $p$-regular $F$-conjugacy classes in $G$.
\end{enumerate}
\end{thm}

We also need the following results concerning the structure of the Galois groups.
\begin{thm}[\cite{Se}]\label{th:bruno1}
Suppose that $n$ is odd or divisible by $4$. Then $\Q_{p}(\zeta_{n})/\Q_{p}$ is a Galois extension of $\Q_{p}$, and its Galois group $G$ is as follows.
\begin{enumerate}[(a)]
\item\label{it:bruno1a} If $p$ does not divide $n$ then $G$ is cyclic, and there exists an element $\sigma\in G$, the Frobenius element of the extension satisfying $\sigma (\zeta_{n})= \zeta_{n}^p$ that generates $G$. Further, the order of $\sigma$ is the order of $p$ considered as an element of $\Z_{n}^{\ast}$.

\item\label{it:bruno1b} If $n=p^m$, $m\geq 1,$ then $G$ is of order $p^{m-1}(p-1),$ and we have a group isomorphism $G\cong  \Z_{p^m}^{\ast}$. Hence $G$ is cyclic if $p$ is odd or if $p=m=2$, and is isomorphic to the direct product $\Z_{2}$ (generated by the class of $-1$) and $\Z_{p^{m-2}}$ (generated by the class of $5$) if $p=2$ and $m\geq 3$.

\item\label{it:bruno1c} Suppose that $n=p^m n_1$, where $n_{1}\geq 2$ and $p$ does not divide $n_1$. Let $\zeta_1$ be a  primitive $n_1\up{th}$ root of unity, and let $\rho$ be a primitive $p^m\up{th}$ root of unity. Then $G\cong \operatorname{\text{Gal}}(\Q_{p}(\zeta_1)/\Q_{p}) \times \operatorname{\text{Gal}}(\Q_{p}(\rho)/\Q_{p})$.
\end{enumerate}
\end{thm}

\begin{thm}[\cite{Cox}]\label{th:bruno2}
Let $k$ denote the order of $p$ considered as an element of $\Z_{n}^{\ast}$. Then the group $\operatorname{\text{Gal}}(\F[p](\zeta_{n})/\F[p])$ is isomorphic to $\Z_{k}$.
\end{thm}

We suppose in what follows that $G$ is dicyclic of order $4m$. We first apply the above results in order to determine the rank of $K_{-1}(\Z[G])$ where $m$ is an odd prime. We then go on to to study the case where $m$ is a power of $2$.

\begin{thm}\label{th:kminusonedic4prime}
Let $m$ be an odd prime, and let $\lambda$ be the number of $\Q_{2}$-conjugacy classes (or equivalently $\F[2]$-conjugacy classes) of the elements of $\dic{4m}$ of order $m$. Then
\begin{equation*}
K_{-1}(\Z[\dic{4m}])\cong
\begin{cases}
\Z^{\lambda}\oplus \Z_{2} & \text{if $m\equiv 1\bmod 4$}\\
\Z^{\lambda} & \text{if $m\equiv 3\bmod 4$.}
\end{cases}
\end{equation*}
\end{thm}

\begin{proof}
Let $G=\dic{4m}$ be given by the presentation~\reqref{presdic}. By \repr{torsionkminusone} and \req{carterkminusone}, it suffices to show that the rank of $K_{-1}(\Z[\dic{4m}])$ is equal to $\lambda$. The group $G$ has one element each ($e$ and $x^m$ respectively) of order $1$ and $2$, $(m-1)$ elements of order $2m$, of the form $x^{i}$, $i$ odd, $1\leq i\leq 2m-1$, and $i\neq m$, $(m-1)$ elements of order $m$, of the form $x^{i}$, $i$ even, $2\leq i\leq 2m-2$, and $2m$ elements of order $4$, of the form $y,xy,\ldots,x^{2m-1}y$. The elements of order $1$ and $2$ each form a single (usual) conjugacy class, those of order $4$ form $2$ conjugacy classes, $\set{x^{i}y}{\text{$0\leq i\leq 2m-2$, $i$ even}}$ and $\set{x^{i}y}{\text{$1\leq i\leq 2m-1$, $i$ odd}}$, while those of order $m$ and $2m$ form $(m-1)$ conjugacy classes of the form $\brak{x^{i}, x^{-i}}$ for $i=1,\ldots,m-1$. Since $r_{\Q}$ is equal to the number of simple components in the Wedderburn decomposition of $\Q[\dic{4m}]$, it follows from \req{weddic4p} that $r_{\Q}=5$. This may also be obtained by observing that the subgroups of $\dic{4m}$ of order $4$ are its Sylow $2$-subgroups, and so $\dic{4m}$ possesses a single conjugacy class of subgroups of order $4$.

We must thus calculate $r_{\Q_{p}}$ and $r_{\F[p]}$ for $p\in \brak{2,m}$, which we do using \reth{wittberman}. Since there is a unique conjugacy class of elements of order $1$ and $2$, these elements contribute $1$ to each of $r_{\Q_{p}}$ and $r_{\F[p]}$, except in the case of $r_{\F[2]}$, where the element of order $2$ is not $2$-regular, so contributes zero. We thus focus on the elements of order $4,m$ and $2m$. According to \cite[page~26]{O}, it suffices to analyse the $F$-conjugacy classes of the elements of order $4,m$ and $2m$ adjoining an $n\up{th}$ root of unity to $F$ for $n=4,m,2m$, where $F=\Q_2$ or $\Q_m$.

\begin{enumerate}[\textbullet]

\item $\Q_{2}$-conjugacy classes of the order $4$ elements: by \reth{bruno1}(\ref{it:bruno1b}), the monomorphism $\map{\phi}{\operatorname{\text{Gal}}(\Q_{2}(\zeta_{4})/\Q_{2})}[\Z_{4}^{\ast}]$ is an isomorphism and $\im{\phi}=\brak{1,3}$. By \req{conjclass}, $[y]_{\Q_{2}}=[y] \cup [y^3]=[y]\cup [x^my]$ as $y^3=y^2\ldotp y= x^my$, and so there is a single $\Q_{2}$-class of order $4$ elements because $m$ is odd.

\item $\Q_{2}$-conjugacy classes of the elements of order $m$ and $2m$: by hypothesis, the number of $\Q_{2}$-conjugacy classes of the elements of order $m$ is equal to $\lambda$. \reth{bruno1}(\ref{it:bruno1c}) implies that the number of $\Q_{2}$-conjugacy classes of the elements of order $2m$ is also equal to $\lambda$.
\end{enumerate}

We conclude that $r_{\Q_{2}}=2\lambda+3$.

\begin{enumerate}[\textbullet]
\item $\Q_{m}$-conjugacy classes of the order $4$ elements: by \reth{bruno1}(\ref{it:bruno1a}), we have a monomorphism $\map{\phi}{\operatorname{\text{Gal}}(\Q_{m}(\zeta_{4})/\Q_{m})}[\Z_{4}^{\ast}]$, $\operatorname{\text{Gal}}(\Q_{m}(\zeta_{4})/\Q_{m})$ is cyclic, and its order is equal to that of $m$ considered as an element of $\Z_{4}^{\ast}$. If $m\equiv 3\bmod{4}$ then $\phi$ is an isomorphism and $\im{\phi}=\brak{1,3}$. By \req{conjclass}, $[y]_{\Q_{m}}=[y] \cup [y^3]=[y]\cup [x^my]$ as $y^3=y^2\ldotp y= x^my$, and so there is a single $\Q_{m}$-class of order $4$ elements since $m$ is odd. If $m\equiv 1 \bmod{4}$ then $\operatorname{\text{Gal}}(\Q_{m}(\zeta_{4})/\Q_{m})$ is trivial and $\im{\phi}=\brak{1}$. In this case, the $\Q_{m}$-conjugacy classes coincide with the usual conjugacy classes, so there are two $\Q_{m}$-conjugacy classes of elements of order $4$.

\item $\Q_{m}$-conjugacy classes of the elements of order $m$: by \reth{bruno1}(\ref{it:bruno1b}), the monomorphism $\map{\phi}{\operatorname{\text{Gal}}(\Q_{m}(\zeta_{m})/\Q_{m})}[\Z_{m}^{\ast}]$ is an isomorphism and $\im{\phi}=\brak{1,\ldots,m-1}$. By \req{conjclass},
$\displaystyle [x^{2}]_{\Q_{m}}=\bigcup_{i=1}^{m-1} [x^{2i}]$, so there is a single $\Q_{m}$-class of order $m$ elements.

\item $\Q_{m}$-conjugacy classes of the elements of order $2m$:  as $m$ is an odd prime, we have that $\Q_m(\zeta_{2m})=\Q_m(\zeta_m)$, so $\map{\phi}{\operatorname{\text{Gal}}(\Q_{m}(\zeta_{m})/\Q_{m})}[\Z_{m}^{\ast}]\cong\Z_{m-1}$ is an isomorphism, and we conclude that there is a single $\Q_{m}$-class of order $2m$ elements.
\end{enumerate}

It thus follows that $r_{\Q_{m}}=6$ if $m\equiv 1 \bmod{4}$, and $r_{\Q_{m}}=5$ if $m\equiv 3 \bmod{4}$.

\begin{enumerate}[\textbullet]
\item $2$-regular {$\F[2]$}-conjugacy classes: we have $G_{2}'=\brak{e,x^2,x^4,\ldots,x^{2m-2}}$, which splits as the disjoint union of $(m+1)/2$ (usual) conjugacy classes in $\dic{4m}$, comprised of $\brak{e}$, and $\bigl\{x^{2i},x^{2(m-i)}\bigr\}$ for $i=1,\ldots, (m-1)/2$. We thus need to study the $\F[2]$-conjugacy classes of the elements of order $m$.
By \reth{bruno2}, we have $\map{\phi}{\operatorname{\text{Gal}}(\F[2](\zeta_{m})/\F[2])}[\Z_{m}^{\ast}]$, where $\operatorname{\text{Gal}}(\F[2](\zeta_{m})/\F[2])$ is cyclic, of order that of $2$ considered as an element of $\Z_{m}^{\ast}$, and $\im{\phi}=\ang{2}$.

We return for a moment to the $\Q_{2}$-conjugacy classes of the elements of order $m$. Replacing $\phi$ by $\phi_{1}$ to distinguish it from the monomorphism $\phi$ of the previous paragraph, by \reth{bruno1}(\ref{it:bruno1a}), we have $\map{\phi_{1}}{\operatorname{\text{Gal}}(\Q_{2}(\zeta_{m})/\Q_{2})}[\Z_{m}^{\ast}]$, and $\operatorname{\text{Gal}}(\Q_{2}(\zeta_{m})/\Q_{2})$ is cyclic, of order that of $2$ considered as an element of $\Z_{m}^{\ast}$. Thus $\im{\phi_{1}}=\ang{2}$ also. In particular, the $\F[2]$-conjugacy class of an element of $\dic{4m}$ of order $m$ is equal to its $\Q_{2}$-conjugacy class, and thus the number of $\F[2]$-conjugacy classes of elements of order $m$ is equal to $\lambda$. We deduce that $r_{\F[2]}=\lambda+1$.

\item $m$-regular $\F[m]$-conjugacy classes: we have
\begin{equation*}
G_{m}'=\bigl\{e,x^m,y,xy,x^2y,\ldots,x^{2m-2}y,x^{2m-1}y\bigr\}.
\end{equation*}
The four (usual) conjugacy classes in $\dic{4m}$ are:
\begin{equation*}
\text{$\brak{e}$, $\brak{x^m}$, $\bigl\{y,x^2y,\ldots,x^{2m-2}y\bigr\}$ and $\bigl\{xy,x^3y,\ldots,x^{2m-1}y\bigr\}$.}
\end{equation*}
It is thus necessary to study the $\F[m]$-conjugacy classes of the latter two classes, which are those of the elements of $\dic{4m}$ of order $4$. By \reth{bruno2}, we have the monomorphism $\map{\phi}{\operatorname{\text{Gal}}(\F[m](\zeta_{4})/\F[m])}[\Z_{4}^{\ast}]$, and $\operatorname{\text{Gal}}(\F[m](\zeta_{4})/\F[m])$ is cyclic, of order that  of $m$ considered as an element of $\Z_{4}^{\ast}$. As in the case of the $\Q_{m}$-conjugacy classes of the order $4$ elements, if $m\equiv 3\bmod{4}$, there is a single $\F[m]$-class of order $4$ elements, while if $m\equiv 1 \bmod{4}$, the $\F[m]$-conjugacy classes coincide with the usual conjugacy classes, and so there are two $\F[m]$-conjugacy classes of order $4$ elements. Hence $r_{\F[m]}=4$ if $m\equiv 1\bmod{4}$ and $r_{\F[m]}=3$ if $m\equiv 3\bmod{4}$.
\end{enumerate}

So by \req{rankkminusone}, the rank $r$ of $K_{-1}(\Z[\dic{4m}])$ is given by:
\begin{align*}
r &=1-r_{\Q}+(r_{\Q_{2}}-r_{\F[2]})+(r_{\Q_{m}}-r_{\F[m]})\\
&=
\begin{cases}
1-5+(2\lambda+3)-(\lambda+1)+(6-4) & \text{if $m\equiv 1 \bmod{4}$}\\
1-5+(2\lambda+3)-(\lambda+1)+(5-3) & \text{if $m\equiv 3 \bmod{4}$}
\end{cases}\\
&=\lambda.\qedhere
\end{align*}
\end{proof}

If $m$ is an odd prime, the proof of \reth{kminusonedic4prime} indicates that the number $\lambda$ of $\Q_{2}$-conjugacy classes of the elements of $\dic{4m}$ of order $m$ is related to the order of the subgroup $\ang{2}$ in $\Z_{m}^{\ast}$. The question of when $2$ generates $\Z_{m}^{\ast}$ is open and constitutes a special case of Artin's primitive root conjecture. The following proposition shows that it is also interesting for us to know whether $-1$ belongs to $\ang{2}$, and enables us to determine the rank of $K_{-1}(\Z[\dic{4m}])$ solely in terms of $\ord{\ang{2}}$.

\begin{prop}\label{prop:lambdacalc}
Let $m$ and $\lambda$ be defined as in the statement of \reth{kminusonedic4prime}. Then 
\begin{equation*}
\lambda=
\begin{cases}
(m-1)\bigl/\ord{\ang{2}}\bigr. & \text{if $-1\in \ang{2}$}\\
(m-1)\bigl/2\ord{\ang{2}}\bigr. & \text{if $-1\notin \ang{2}$.}
\end{cases}
\end{equation*}
\end{prop}

\begin{exs}\mbox{}
\begin{enumerate}[(a)]
\item Suppose that $m$ is a Fermat number, of the form $2^{2^{s}}+1$, where $s\in\N$. Then $\ord{\ang{2}}=2^{s+1}$ and $-1\in \ang{2}$, so the rank of $K_{-1}(\Z[\dic{4m}])$ is equal to $\lambda=2^{2^{s}-s-1}$. For example, if $m=257$ then $\lambda=16$ and $K_{-1}(\Z[\dic{1\,028}])\cong \Z_{2}\oplus \Z^{16}$.
\item Suppose that $m$ is a Mersenne prime, of the form $2^{p}-1$, where $p$ is prime. Then $\ord{\ang{2}}=p$ and $-1\notin\ang{2}$, so the rank of $K_{-1}(\Z[\dic{4m}])$ is equal to $\lambda=\frac{2^{p}-2}{2p}=\frac{2^{p-1}-1}{p}$. For example, if $m=127$ then $\lambda=9$ and $K_{-1}(\Z[\dic{508}])\cong \Z^{9}$, and if $m=8\,191$ then $\lambda=315$ and $K_{-1}(\Z[\dic{32\,728}])\cong \Z^{315}$.
\end{enumerate}
\end{exs}

\begin{proof}[Proof of \repr{lambdacalc}]
Using \req{presdic}, the elements of $\dic{4m}$ of order $m$ are of the form $x^{2i}$, $1\leq i\leq m-1$, and $[x^{2i}]=\brak{x^{2i},x^{-2i}}$, in particular, they form $(m-1)/2$ distinct (usual) conjugacy classes in $\dic{4m}$. Let $1\leq i \leq m-1$. Since $m$ is prime, there exist $\tau,\mu\in \Z$ such that $\tau i+\mu m=1$. One may check easily that the maps $[x^2]_{\Q_{2}}\to [x^{2i}]_{\Q_{2}}$ and $[x^{2i}]_{\Q_{2}} \to [x^2]_{\Q_{2}}$, defined respectively by $w\longmapsto w^i$ and $z\longmapsto z^\tau$, are mutual inverses, and hence $[x^{2i}]_{\Q_{2}}$ has the same number of elements as $[x^{2}]_{\Q_{2}}$. Thus the number of $\Q_{2}$-conjugacy classes of the elements of order $m$, which is equal to $\lambda$, is just $(m-1)$ divided by the cardinal of $[x^{2}]_{\Q_{2}}$. \reth{bruno1}(\ref{it:bruno1a}) and \req{conjclass} imply that:
\begin{align*}
[x^{2}]_{\Q_{2}} &=\bigcup_{t\in \ang{2}} [x^{2t}]= \setbigl{x^{2i}}{i\in \ang{2}}\cup \setbigl{x^{-2i}}{i\in\ang{2}}\\
&=\setbigl{x^{2i}}{i\in \ang{2}}\cup \setbigl{x^{2i}}{i\in-\ang{2}}.
\end{align*}
Now $-\ang{2}$ is the $\ang{2}$-coset of $-1$ in $\Z_{m}^{\ast}$, so $\setl{x^{2i}}{i\in \ang{2}}$ and $\setl{x^{2i}}{i\in-\ang{2}}$ have the same cardinality $\ord{\ang{2}}$, and are either equal or disjoint. Since $-1\in -\ang{2}$, they are equal if and only if $-1\in \ang{2}$. This being the case, the cardinality of $[x^{2}]_{\Q_{2}}$ is equal to $\ord{\ang{2}}$, and $\lambda=(m-1)/\ord{\ang{2}}$. If $-1\notin \ang{2}$, the two cosets $\ang{2}$ and $-\ang{2}$ are disjoint, thus the cardinality of $[x^{2}]_{\Q_{2}}$ is equal to $2\ord{\ang{2}}$, and $\lambda=(m-1)/2\ord{\ang{2}}$ as required.
\end{proof}

The methods used above allow us in theory to calculate $K_{-1}(\Z[\dic{4m}])$ for any $m\geq 2$, not just for $m$ an odd prime. As another example, consider the case where $m$ is a power of $2$, so $G\cong \quat[2^{k}]$ is the generalised quaternion group of order $2^{k}$, where $m=2^{k-2}$.

\begin{prop}\label{prop:kminusonequat2k}
$K_{-1}(\Z[\quat[2^k]])$ is trivial if $k=3$, and is isomorphic to $\Z_{2}$ if $k\geq 4$.
\end{prop}

\begin{proof}
By \reth{carter} and \repr{torsquat}, it suffices to show that for all $k\geq 3$, the rank of $K_{-1}(\Z[\quat[2^k]])$ is zero, which we do using \reth{wittberman}. We must calculate $r_{\Q}$, $r_{\Q_{2}}$ and $r_{\F[2]}$. Using the presentation~\reqref{presdic} of $\quat[2^k]$, we see that $\quat[2^k]=\ang{x} \coprod \ang{x}y$, and that the elements of $\ang{x}y$ are all of order $4$. So $G_{2}'$ consists of the identity element, whence $r_{\F[2]}=1$. 

We now determine the number $r_{\Q}$ of $\Q$-conjugacy classes, which by \rerem{fconjclass}(\ref{it:fconjclassii}), is equal to the number of conjugacy classes of cyclic subgroups in $\quat[2^{k}]$. The elements of $\quat[2^k]$ are of order $2^l$, $0\leq l\leq k-1$, and if $l\neq 2$ then the elements of order $2^l$ are contained entirely within $\ang{x}$. Thus there is just one subgroup of order $2^l$ for each such $l$, and so these subgroups contribute $k-1$ to $r_{\Q}$. Suppose then that $l=2$. Using the relations 
\begin{equation}\label{eq:relsconj}
\text{$y(x^iy)y^{-1}=x^{-i}y$, $x(x^iy)x^{-1}=x^{i+2}y$ and $(x^{i}y)^{-1}=x^{i+2^{k-2}}y$}
\end{equation}
in $\quat[2^{k}]$, we see that there are at most three conjugacy classes of cyclic subgroups of order $4$, represented by the subgroups $\bigl\langle x^{2^{k-3}} \bigr\rangle$, $\ang{y}$ and $\ang{xy}$. Since $\bigl\langle x^{2^{k-3}}\bigr\rangle$ is contained in the normal subgroup $\ang{x}$ of $\quat[2^k]$, it cannot be conjugate to the two other subgroups, and using relations~\reqref{relsconj}, we see that $\ang{y}$ and $\ang{xy}$ are non conjugate. We thus conclude that $r_{\Q}=k+2$. This number may also be obtained by counting the number of simple components in the Wedderburn decomposition \reqref{decompQG} of $\Q[\quat[2^{k}]]$.

Finally we calculate $r_{\Q_{2}}$. Consider the elements of $\quat[2^{k}]$ of order $2^l$, where $0\leq l\leq k-1$. If $l\in\brak{0,1}$ then there is just one element of order $2^l$, and so the contribution to $r_{\Q_{2}}$ is one in each case. If $l=2$ then $\operatorname{\text{Gal}}(\Q_{2}(\zeta_{2^2})/\Q_{2})\cong \Z_{2^2}^{\ast}=\brak{1,3}$ by \reth{bruno1}(\ref{it:bruno1b}). Hence for every element $z$ of $\quat[2^k]$ of order $4$, $[z]_{\Q_{2}}=[z]\cup [z^3]=[z]\cup [z^{-1}]=[z]$ since in $\quat[2^k]$, every element is conjugate to its inverse. Thus the elements of $\quat[2^{k}]$ of order $4$ contribute $3$ to $r_{\Q_{2}}$. Suppose then that $l\geq 3$. The elements of order $2^l$ are contained in $\ang{x}$, are elements of the subgroup $\bigl\langle x^{2^{k-l-1}} \bigr\rangle$ of the form $x^{2^{(k-l-1)}r}$, where $\gcd{(r,2^l)}=1$, and so are of the form $x^{2^{(k-l-1)}r}$, where $r\in \bigl\{1,3,\ldots, 2^l-1\bigr\}$. On the other hand, applying \reth{bruno1}(\ref{it:bruno1b}), we see that $\operatorname{\text{Gal}}(\Q_{2}(\zeta_{2^l})/\Q_{2})\cong \Z_{2^l}^{\ast}$. Now $\Z_{2^l}^{\ast}=\bigl\{1, 3,\ldots, 2^l-1\bigr\}$, and thus
\begin{equation*}
\bigl[x^{2^{k-l-1}}\bigr]_{\Q_{2}}= \bigl[x^{2^{k-l-1}} \bigr] \cup \bigl[x^{3(2^{k-l-1})} \bigr] \cup \cdots \cup \bigl[x^{(2^l-1)(2^{k-l-1})}\bigr].
\end{equation*}
From above, this is precisely the set of all elements of order $2^l$, and hence for each $3\leq l\leq k-1$, the elements of order $2^l$ contribute one to $r_{\Q_{2}}$. Summing over all possible values of $l$ yields $r_{\Q_{2}}=k+2$, and applying \req{rankkminusone}, we obtain $r=1-r_{\Q}+r_{\Q_{2}}-r_{\F[2]}=0$, which proves the proposition.
\end{proof}

We now turn to the calculation of $K_{-1}(\Z[G])$, where $G$ is a binary polyhedral group.

\begin{prop}\label{prop:kminusonebinpoly}
$K_{-1}(\Z [\tstar])\cong \Z$, $K_{-1}(\Z [\ostar])\cong \Z_{2}\oplus \Z$ and $K_{-1}(\Z [\istar])\cong\Z_{2}\oplus \Z^2$.
\end{prop}

\begin{proof}
Let $G$ be a binary polyhedral group. By \repr{torsbinpoly}, it suffices to calculate the rank of $K_{-1}(\Z[G])$, which we do using~\reqref{rankkminusone} and \reth{wittberman}. From \rerem{fconjclass}(\ref{it:fconjclassii}) and the notation of \resec{ccbpg}, $\displaystyle r_{\Q}=\sum_{d \divides \ord{G}} r_{2}(d)$, and it follows from \repr{ribinpoly} that $r_{\Q}=5$ if $G=\tstar$, and $r_{\Q}=7$ if $G=\ostar$ or $\istar$. These values of $r_{\Q}$ may also be obtained from the corresponding Wedderburn decompositions given in~\reqref{weddertstar}--\reqref{wedderistar}.
 
\begin{enumerate}[(a)]
\item\label{it:rankkmotstar} We first calculate the rank of $K_{-1}(\Z[\tstar])$, where a presentation of $\tstar=\ang{P,Q,X}$ is given by the first line of \req{presostar}. 
\begin{enumerate}[\textbullet]
\item The set $G_{3}'$ consists of the union of the elements of $\tstar$ of order $1$, $2$ and $4$. By \repr{ribinpoly}(\ref{it:ribinpolya}), if $m\in \brak{1,2,4}$, the elements of order $m$ form a single conjugacy class, and thus form a single $\Q_{p}$-conjugacy class for $p\in \brak{2,3}$, whence $r_{\F[3]}=3$. 

\item The set $G_{2}'$ consists of the identity and the $8$ elements of $\tstar$ of order $3$, and by \repr{ribinpoly}(\ref{it:ribinpolya}), there are two conjugacy classes of the elements of order $3$, of which $X$ and $X^{-1}$ are representatives. By \reth{bruno2}, we have an isomorphism $\map{\phi}{\operatorname{\text{Gal}}(\F[2](\zeta_{3})/\F[2])}[\Z_{3}^{\ast}]$, and $\im{\phi}=\brak{1,2}$. Thus $[X]_{\F[2]}=[X]\cup [X^{2}]= [X]\cup [X^{-1}]$. It follows that there is a single $\F[2]$-conjugacy class of elements of order $3$, and so $r_{\F[2]}=2$.

\item Since there is a single conjugacy class of elements of order $d$, where $d\in \brak{1,2,4}$, it remains to determine the number of $\Q_{p}$-conjugacy classes, $p\in\brak{2,3}$, of the elements of $\tstar$ of order $3$ and $6$. We first calculate the number of $\Q_{p}$-conjugacy classes of the elements of order $3$. By \reth{bruno1}(\ref{it:bruno1a}) and~(\ref{it:bruno1b}), $\map{\phi}{\operatorname{\text{Gal}}(\Q_{p}(\zeta_{3})/\Q_{p})}[\Z_{3}^{\ast}]$ is an isomorphism, $\im{\phi}=\brak{1,2}$, and $[X]_{\Q_{p}}=[X]\cup [X^2]=[X]\cup [X^{-1}]$, which is the union of the two (usual) conjugacy classes of elements of order $3$. We have the same result for the elements of order $6$ of $\tstar$, since they are obtained from those of order $3$ by adjoining the central element of $\tstar$ of order $2$. So for all $d\in \brak{1,2,3,4,6}$ and $p\in \brak{2,3}$, there is a single $\Q_{p}$-conjugacy class of the elements of order $d$, and hence $r_{\Q_{2}}=r_{\Q_{3}}=5$.

\item By \req{rankkminusone}, the rank $r$ of $K_{-1}(\Z [\tstar])$ is equal to $r=1-r_{\Q}+ r_{\Q_{2}}-r_{\F[2]}+r_{\Q_{3}}-r_{\F[3]}=1$, and thus $K_{-1}(\Z [\tstar])\cong \Z$.
\end{enumerate}

\item We now calculate the rank of $\K_{-1}(\Z [\ostar])$. 
\begin{enumerate}[\textbullet]
\item Recall first that $\tstar$ is a subgroup of $\ostar$ of index $2$, and that $\ostar \setminus \tstar$ consists of twelve elements of order $4$ and of order $8$. So $G_{3}'$ is contained in $\tstar$, and as in case~(\ref{it:rankkmotstar}) we obtain $r_{\F[2]}=2$. Further, the elements of $\ostar$ of order $1,2,3$ and $6$ each give rise to a single $\Q_{p}$-conjugacy class of $\ostar$ for $p\in\brak{2,3}$. It remains to calculate the number of $\Q_{p}$-conjugacy classes of the elements of order $4$ and $8$, as well as $r_{\F[3]}$. 

\item To calculate the number of $\Q_{p}$-conjugacy classes of the elements of order $8$, recall from \repr{ribinpoly}(\ref{it:ribinpolyb}) that there are two conjugacy classes of elements of order $8$, for which representatives are $g$ and $g^{3}$, where $g$ is any element of $\ostar$ of order $8$. By \reth{bruno1}(\ref{it:bruno1a}) and~(\ref{it:bruno1b}), $\map{\phi}{\operatorname{\text{Gal}}(\Q_{p}(\zeta_{8})/\Q_{p})}[\Z_{8}^{\ast}]$ satisfies $\im{\phi}\supset\brak{{1,3}}$, hence $[g]_{\Q_{p}}\supset [g]\cup [g^3]$, and there is a single $\Q_{p}$-conjugacy class of elements of order $8$.

\item To calculate the number of $\Q_{p}$-conjugacy classes of the elements of order $4$, recall from \repr{ribinpoly}(\ref{it:ribinpolyb}) that there are two conjugacy classes of elements of order $4$, $C_{1}$ and $C_{2}$, where $C_{1}$ (resp.\ $C_{2}$) is the intersection of the set of elements of $\ostar$ of order $4$ with $\tstar$ (resp.\ with $\ostar \setminus \tstar$). In particular, if $g\in C_{1} \cup C_{2}$ then $g$ and $g^{-1}$ are conjugate. By \reth{bruno1}(\ref{it:bruno1a}) and~(\ref{it:bruno1b}), $\map{\phi}{\operatorname{\text{Gal}}(\Q_{p}(\zeta_{4})/\Q_{p})}[\Z_{4}^{\ast}]$ is an isomorphism, $\im{\phi}=\brak{1,3}$, and $[g]_{\Q_{p}}= [g]\cup [g^{-1}]=[g]$ for all $g\in \ostar$ of order $4$. So the number of $\Q_{p}$-conjugacy classes of elements of order $4$ is equal to $2$. 

\item From the above computations, if $d\in \brak{1,2,3,6,8}$ and $p\in \brak{2,3}$, there is a single $\Q_{p}$-conjugacy class of elements of order $d$, and there are two $\Q_{p}$-conjugacy class of elements of order $4$, whence $r_{\Q_{p}}=7$.

\item To calculate $r_{\F[3]}$, first note that $G_{3}'$ consists of the union of the elements of $\ostar$ of order $1,2,4$ and $8$, and that there is a single conjugacy class of elements of order $1$ and $2$. Let $m\in \brak{4,8}$. By \reth{bruno2}, $\map{\phi}{\operatorname{\text{Gal}}(\F[3](\zeta_{m})/\F[3])}[\Z_{m}^{\ast}]$ satisfies $\im{\phi}=\brak{{1,3}}$, and we see that the number of $\F[3]$-conjugacy classes of elements of order $m$ is just the number of $\Q_{p}$-conjugacy classes of these elements, \emph{i.e.} there are two $\F[3]$-conjugacy classes of elements of order $4$, and one $\F[3]$-conjugacy class of elements of order $8$. We conclude that $r_{\F[3]}=5$.

\item By \req{rankkminusone}, the rank $r$ of $K_{-1}(\Z [\ostar])$ is equal to $r=1-r_{\Q}+ r_{\Q_{2}}-r_{\F[2]}+r_{\Q_{3}}-r_{\F[3]}=1$, and thus $K_{-1}(\Z [\ostar])\cong \Z_{2} \oplus \Z$.
\end{enumerate}

\item Finally, we determine the rank of $\K_{-1}(\Z [\istar])$. 
\begin{enumerate}[\textbullet]
\item By \repr{ribinpoly}(\ref{it:ribinpolyc}), $r_{2}(l)=1$ for all $l\in \brak{1,2,3,4,6}$, so there is a single conjugacy class of elements of order $l$, and there are two conjugacy classes of elements of order $5$ and $10$. Hence it suffices to study the various $F$-conjugacy classes for the elements of order $5$ and $10$. 

\item We compute the number of the elements $\Q_{p}$-conjugacy classes of elements of order $5$ for $p\in \brak{2,3,5}$. By \repr{ribinpoly}(\ref{it:ribinpolyc}), if $g\in \istar$ is of order $5$, $g$ and $g^{2}$ are representatives of the two conjugacy classes of elements of order $5$. Using \reth{bruno1}(\ref{it:bruno1a}) and~(\ref{it:bruno1b}), the homomorphism $\map{\phi}{\operatorname{\text{Gal}}(\Q_{p}(\zeta_{5})/\Q_{p})}[\Z_{5}^{\ast}]$ is an isomorphism, and so there is a single $\Q_{p}$-conjugacy class of elements of order $5$ in $\istar$ for all $p\in \brak{2,3,5}$. By adjoining the central element of $\istar$ of order $2$ to $g$, it follows that the same is true for the elements of order $10$, from which it follows that $r_{\Q_{p}}=7$ for all $p\in \brak{2,3,5}$.

\item To compute the number of $\F[2]$- and $\F[3]$-conjugacy classes of $\istar$, note that $G_{2}'$ is the union of the elements of $\istar$ of order $1,3$ and $5$, and $G_{3}'$ is the union of the elements of $\istar$ of order $1,2,4,5$ and $10$. By \repr{ribinpoly}(\ref{it:ribinpolyc}), there is a single conjugacy class in $\istar$ of elements of order $1,2,3$ and $4$. By \reth{bruno2}, for $p\in\brak{2,3}$, the homomorphism $\operatorname{\text{Gal}}(\F[p](\zeta_{5})/\F[p])\to \Z_{5}^{\ast}$ is an isomorphism, and $\im{\phi}=\brak{1,2,3,4}$. Thus there is a single $\F[p]$-conjugacy class of the $5$-regular elements of order $5$. By adjoining the central element of $\istar$ of order $2$ to $g$, it follows that the same is true for the elements of order $10$ in the case $p=3$. We conclude that $r_{\F[2]}=3$ and $r_{\F[3]}=5$.

\item To compute the number of $\F[5]$-conjugacy classes of $\istar$, the set $G_{5}'$ is the union of the elements of $\istar$ of order $1,2,3,4$ and $6$. Since there is a single conjugacy class in $\istar$ of elements of each of these orders, it follows that $r_{\F[5]}=5$.

\item By \req{rankkminusone}, the rank $r$ of $K_{-1}(\Z [\istar])$ is equal to $r=1-r_{\Q}+ r_{\Q_{2}}-r_{\F[2]}+r_{\Q_{3}}-r_{\F[3]}+r_{\Q_{5}}-r_{\F[5]}=2$, and thus $K_{-1}(\Z [\istar])\cong \Z_{2} \oplus \Z^{2}$.\qedhere
\end{enumerate}
\end{enumerate}
\end{proof}

As we mentioned in \resec{rankkminusone}, in order to prove \reth{tablas} and to obtain Table~\ref{tab:finite}, we need to compute $K_{-1}(\Z [\dic{4\mu}])$ for $\mu\in \brak{6,9,10}$. The torsion of these groups was already determined in \repr{torsiondicextra}. To end this section, we calculate their rank.

\begin{prop}\label{prop:rankdicextra}
If $\mu\in \brak{6,9,10}$, the rank of $K_{-1}(\Z[\dic{4\mu}])$ is equal to $2$.
\end{prop}

\begin{proof}\mbox{}
\begin{enumerate}[(a)]
\item We first consider the cases where $\mu\in \brak{6,10}$, so $\mu/2$ is an odd prime. Making use of the presentation of the form~\reqref{presdic} of $\dic{4\mu}$, the following table summarises the elements of each order of $\dic{4\mu}$.
\renewcommand{\arraystretch}{1.2}
\setlength{\tabcolsep}{5pt}
\begin{center}
\begin{tabular}{|>{$}c<{$}||>{$}c<{$}|>{$}c<{$}|}
\hline
\text{order $d$} & \text{elements of order $d$} & \text{number of elements}\\ \hline\hline
1 & e & 1\\ \hline
2 & x^{\mu} & 1\\ \hline
4 & x^{\mu/2},\, x^{3\mu/2},\, x^{i}y, \, i\in \brak{0,1,\ldots,2\mu-1} & 2\mu+2\\ \hline
\mu/2 & x^{4},x^{8},\ldots, x^{2(\mu-2)} & (\mu-2)/2\\ \hline
\mu & x^{2},x^{6},\ldots,x^{\mu-4}, x^{\mu+4},x^{\mu+8},\ldots, x^{2(\mu-2)} & (\mu-2)/2\\ \hline
2\mu & x^{i},\, i\in \brak{1,3,\ldots, 2\mu-1}\setminus \brak{\mu/2, 3\mu/2} & \mu-2\\ \hline
\end{tabular}
\end{center}
We compute the number of $F$-conjugacy classes for each of the fields $F$ that appear in~\reqref{rankkminusone}.

\begin{enumerate}[\textbullet]
\item From the above table, $\dic{4\mu}$ possesses a single cyclic subgroup of order $r$ for all $r\in \brak{1,2,\mu/2,\mu,2\mu}$, and using~\reqref{presdic}, it has three conjugacy classes of elements of order $4$, namely $\bigl\{x^{\mu/2},x^{3\mu/2}\bigr\}$, $\setl{x^{i}y}{i\in \brak{1,3,\ldots,2\mu-1}}$ and $\setl{x^{i}y}{i\in \brak{0,2,\ldots,2\mu-2}}$. So $\dic{4\mu}$ has three conjugacy classes of (cyclic) subgroups of order $4$. We conclude that $\dic{4\mu}$ has eight conjugacy classes of cyclic subgroups, hence $r_{\Q}=8$. 

\item The set of $2$-regular elements of $\dic{4\mu}$ consists of $e$ and the $(\mu-2)/2$ elements of order $\mu/2$. Since the order of $\overline{2}$ in $\Z_{\mu/2}^{\ast}$ is equal to $(\mu-2)/2$, which is the order of $\Z_{\mu/2}^{\ast}$, the injective homomorphism $\map{\phi}{\operatorname{\text{Gal}}(\F[2](\zeta_{\mu/2})/\F[2])}[\Z_{\mu/2}^{\ast}]$ is an isomorphism. Using~\reqref{conjclass}, it follows that the $\F[2]$-conjugacy class of $x^{4}$ is equal to $\bigl\{x^{4},x^{8},\ldots, x^{2(m-2)}\bigr\}$, and thus $r_{\F[2]}=2$. 

\item The set of $\mu/2$-regular elements of $\dic{4\mu}$ consists of $e$, $x^{\mu}$, which is of order $2$, and the $2\mu+2$ elements of order $4$. The image of the injective homomorphism $\map{\phi}{\operatorname{\text{Gal}}(\F[\mu/2](\zeta_{4})/\F[\mu/2])}[\Z_{4}^{\ast}]$ is contained in $\brak{1,3}$, and so is equal to $\brak{1}$ or $\brak{1,3}$. But $(x^{\mu/2})^{3}=x^{3\mu/2}$, and for all $i\in \brak{0,1,\ldots, 2\mu-1}$, $(x^{i}y)^{3}=(x^{i}y)^{-1}=y^{-1}x^{-i}=y^{-1}x^{-i}y\ldotp y^{2} \ldotp y=x^{i+\mu}y$. It follows that if $z\in \dic{4\mu}$ is of order $4$, $[z]=[z^{3}]$, and by~\reqref{conjclass}, we have $[z] \subset [z]_{\F[\mu/2]}\subset [z] \subset [z^{3}]=[z]$, so $[z]_{\F[\mu/2]}=[z]$. Thus the $\F[2]$-conjugacy classes of the $\mu/2$-regular elements of $\dic{4\mu}$ of order $4$ coincide with the usual conjugacy classes, whence $r_{\F[\mu/2]}=5$.
\end{enumerate}

We now compute $r_{\Q_{2}}$. To do so, we need to determine the number of $\Q_{2}$-conjugacy classes of the elements of order $4,\mu/2,\mu$ and $2\mu$.

\begin{enumerate}[\textbullet]
\item We calculate the number of $\Q_{2}$-conjugacy classes of the elements of order $4$. By~\reth{bruno1}(\ref{it:bruno1b}), the injective homomorphism $\map{\phi}{\operatorname{\text{Gal}}(\Q_{2}(\zeta_{4})/\Q_{2})}[\Z_{4}^{\ast}]$ is an isomorphism, and $\im{\phi}=\brak{\overline{1}, \overline{3}}$. As in the analysis of the $\mu/2$-regular elements of order $4$, it follows that $[z]_{\Q_{2}}=[z]$ for every element $z\in \dic{4\mu}$ of order $4$, and so the $\Q_{2}$-conjugacy classes of the elements of order $4$ coincide with the usual conjugacy classes, and hence there are three $\Q_{2}$-conjugacy classes of  elements of order $4$. 

\item We determine the number of $\Q_{2}$-conjugacy classes of the elements of order $\mu/2$, $\mu$ and $2\mu$. Let $j\in \brak{0,1,2}$. Then the injective homomorphism $\map{\phi}{\operatorname{\text{Gal}}(\Q_{2}(\zeta_{2^{j}\mu/2})/\Q_{2})}[\Z_{2^{j}\mu/2}^{\ast}]$ is an isomorphism using~\reth{bruno1}(\ref{it:bruno1b}) and~(\ref{it:bruno1c}) because $\operatorname{\text{Gal}}(\Q_{2}(\zeta_{2}/\Q_{2})$ is trivial and the group $\operatorname{\text{Gal}}(\Q_{2}(\zeta_{4}/\Q_{2})$ is of order $2$. Thus there is a single $\Q_{2}$-conjugacy class of elements of order $2^{j} \mu/2$ for all $j\in \brak{0,1,2}$.

\item It follows from these calculations that there is a single $\Q_{2}$-conjugacy class of elements of order $r$ for all $r\in \brak{1,2,\mu/2,\mu,2\mu}$, and three $\Q_{2}$-conjugacy classes of  elements of order $4$, so $r_{\Q_{2}}=8$.
\end{enumerate}

We now compute $r_{\Q_{\mu/2}}$. To do so, we need to determine the number of $\Q_{\mu/2}$-conjugacy classes of the elements of order $4,\mu/2,\mu$ and $2\mu$.

\begin{enumerate}[\textbullet]
\item Let us determine the number of $\Q_{\mu/2}$-conjugacy classes of the elements of order $\mu/2$ and $\mu$. If $j\in \brak{0,1}$, using~\reth{bruno1}(\ref{it:bruno1b}) and~(\ref{it:bruno1c}), we see that the injective homomorphism $\map{\phi}{\operatorname{\text{Gal}}(\Q_{\mu/2}(\zeta_{2^{j}\mu/2})/\Q_{\mu/2})}[\Z_{2^{j}\mu/2}^{\ast}]$ is an isomorphism because $\operatorname{\text{Gal}}(\Q_{\mu/2}(\zeta_{2}/\Q_{2})$ is trivial. Thus there is a single $\Q_{\mu/2}$-conjugacy class of elements of order $2^{j} \mu/2$ for all $j\in \brak{0,1}$.

\item To calculate the number of $\Q_{\mu/2}$-conjugacy classes of the elements of order $2\mu$, consider the injective homomorphism $\map{\phi}{\operatorname{\text{Gal}}(\Q_{\mu/2}(\zeta_{2\mu})/\Q_{\mu/2})}[\Z_{2\mu}^{\ast}]$. By~\reth{bruno1}(\ref{it:bruno1c}), the group
$\operatorname{\text{Gal}}(\Q_{\mu/2}(\zeta_{2\mu})/\Q_{\mu/2})$ is isomorphic to the direct product
\begin{equation*}
\operatorname{\text{Gal}}(\Q_{\mu/2}(\zeta_{\mu/2})/\Q_{\mu/2})\times \operatorname{\text{Gal}}(\Q_{\mu/2}(\zeta_{4})/\Q_{\mu/2}),
\end{equation*}
which by \reth{bruno1}(\ref{it:bruno1b}) is isomorphic to $\Z_{\mu/2}^{\ast}\times \Z_{2}$ (resp.\ $\Z_{\mu/2}^{\ast}$) if $\mu/2 \equiv 3 \bmod{4}$ (resp.\ if $\mu/2 \equiv 1 \bmod{4}$). We now distinguish the two cases $\mu=6$ and $\mu=10$.

\begin{enumerate}[--]
\item If $\mu=6$, $\phi$ is an isomorphism, and there is a single $\Q_{3}$-conjugacy class of elements of order $12$.

\item If $\mu=10$, $\operatorname{\text{Gal}}(\Q_{5}(\zeta_{20})/\Q_{5})$ is isomorphic to $\Z_{4}$ by~\reth{bruno1}(\ref{it:bruno1b}) and~(\ref{it:bruno1c}), so the image of $\phi$ is a subgroup of $\Z_{20}^{\ast}$. Now $\Z_{20}^{\ast}$ is isomorphic to $\Z_{2}\times \Z_{4}$, so it possesses two subgroups isomorphic to $\Z_{4}$. A calculation shows that these two subgroups are of the form $\brak{\overline{1},\overline{3},\overline{7},\overline{9}}$ and $\brak{\overline{1},\overline{9},\overline{13},\overline{17}}$. Using the table given at the beginning of the proof and~\reqref{conjclass} and the fact that $x^{k}$ is conjugate to $x^{20-k}$ for all $k\in \brak{1,3,7,9,11,13,17,19}$ by~\reqref{presdic}, it follows in either case that there is a single $\Q_{5}$-conjugacy class of elements of order $20$.
\end{enumerate}

\item It follows from these calculations that there is a single $\Q_{\mu/2}$-conjugacy class of elements of order $r$ for all $r\in \brak{1,2,\mu/2,\mu,2\mu}$, and three $\Q_{\mu/2}$-conjugacy classes of elements of order $4$, so $r_{\Q_{\mu/2}}=8$ for $\mu\in \brak{6,10}$. 

\item Using~\reqref{rankkminusone}, we conclude that $r=1-r_{\Q}+r_{\Q_{2}}-r_{\F[2]}+r_{\Q_{\mu/2}}-r_{\F[\mu/2]}= 1-8+(8-2)+(8-5)=2$ as required.
\end{enumerate}

\item Now suppose that $\mu=9$. Using the presentation of the form~\reqref{presdic} of $\dic{4\mu}$, the following table summarises the elements of each order of $\dic{4\mu}$.
\renewcommand{\arraystretch}{1.2}
\setlength{\tabcolsep}{5pt}
\begin{center}
\begin{tabular}{|>{$}c<{$}||>{$}c<{$}|>{$}c<{$}|}
\hline
\text{order $d$} & \text{elements of order $d$} & \text{number of elements}\\ \hline\hline
1 & e & 1\\ \hline
2 & x^9 & 1\\ \hline
3 & x^{6},\, x^{12} & 2\\ \hline
4 & x^{i}y, \, i\in \brak{0,1,\ldots, 17} & 18\\ \hline
6 & x^{3},\, x^{15} & 2\\ \hline
9 & x^{2i},\, i\in \brak{1,2,4,5,7,8} & 6\\ \hline
18 & x^{i},\, i\in \brak{1,5,7,11,13,17} & 6\\ \hline
\end{tabular}
\end{center}
In order to apply~\reqref{rankkminusone}, we compute the number of $F$-conjugacy classes for each of the fields $F$ that appear in that equation.

\begin{enumerate}[\textbullet]
\item By~\reqref{presdic}, there are two conjugacy classes of the elements of order $4$, $\setl{x^{i}y}{i\in \brak{0,2,\ldots,16}}$ and $\setl{x^{i}y}{i\in \brak{1,3,\ldots,17}}$, and the remaining conjugacy classes are of the form $\brak{x^{i}, x^{18-i}}$ for $i\in \brak{0,1,\ldots, 9}$. Recall that $r_{\Q}$ is given by the number of factors in \req{wedder36}, so $r_{\Q}=7$ (this may also by verifying that there is a single conjugacy class of cyclic subgroups of order $d$ for each $d\in \brak{1,2,3,4,6,9,18}$). 

\item If $d\in \brak{1,2,3,6}$, there is a single conjugacy class of elements of order $d$. Thus there is a single $\Q_{p}$-conjugacy class of elements of order $d$ by~\reqref{conjclass}, where $p\in \brak{2,3}$. Similiarly, if $d\in \brak{1,2}$ (resp.\ $d\in \brak{1,3}$), there is a single $\F[3]$-conjugacy class (resp.\ $\F[2]$-conjugacy class) of elements of order $d$. So it suffices to determine:
\begin{enumerate}[(i)]
\item the number of $\F[2]$-conjugacy classes of elements of order $9$.
\item the number of $\F[3]$-conjugacy classes of elements of order $4$.
\item the number of $\Q_{p}$-conjugacy classes of elements of order $d$, where $d\in \brak{4,9,18}$, and $p\in\brak{2,3}$.
\end{enumerate}
\end{enumerate}
We consider these cases in turn.
 
\begin{enumerate}[\textbullet]

\item By \reth{bruno2}, the injective homomorphism $\map{\phi}{\operatorname{\text{Gal}}(\F[2](\zeta_{9})/\F[2])}[\Z_{9}^{\ast}]$ is an isomorphism, and so there is a single $\F[2]$-conjugacy class of elements of order $9$. Now the set $G_{2}'$ of $2$-regular elements of $\dic{36}$ is given by the union of the elements of order $1$, $3$ and $9$, and since there is a single (usual) conjugacy class of elements of order $1$ and $3$, we conclude that $r_{\F[2]}=3$.

\item By \reth{bruno2}, the injective homomorphism $\map{\phi}{\operatorname{\text{Gal}}(\F[3](\zeta_{4})/\F[3])}[\Z_{4}^{\ast}]$ is an isomorphism, and so there is a single $\F[3]$-conjugacy class of elements of order $4$. Now the set $G_{3}'$ of $3$-regular elements of $\dic{36}$ is given by the union of the elements of order $1$, $2$ and $4$, and since there is a single (usual) conjugacy class of elements of order $1$ and $2$, we conclude that $r_{\F[3]}=3$.

\item $\Q_{2}$-conjugacy classes of elements of order $4$: by \reth{bruno1}(\ref{it:bruno1b}), the injective homomorphism $\map{\phi}{\operatorname{\text{Gal}}(\Q_{2}(\zeta_{4})/\Q_{2})}[\Z_{4}^{\ast}]$ is an isomorphism, and $\im{\phi}=\brak{\overline{1}, \overline{3}}$. Thus $[y]_{\Q_{2}}=[y]\cup [y^{3}]=[y] \cup [x^{9}y]$, where $y$ is the element of $\dic{36}$ appearing in~\reqref{presdic}, so $[y]_{\Q_{2}}$ is the union of the two conjugacy classes of elements of order $4$. Consequently, there is a single $\Q_{2}$-conjugacy class of elements of order $4$.

\item $\Q_{2}$-conjugacy classes of elements of order $9$: by \reth{bruno1}(\ref{it:bruno1a}), the injective homomorphism $\map{\phi}{\operatorname{\text{Gal}}(\Q_{2}(\zeta_{9})/\Q_{2})}[\Z_{9}^{\ast}]$ is an isomorphism, and $\im{\phi}=\brak{\overline{1},\overline{2},\overline{4}, \overline{5}, \overline{7}, \overline{8}}$. It follows from the above table of elements of $\dic{36}$ and~\reqref{conjclass} that there is a single $\Q_{2}$-conjugacy class of elements of order $9$.

\item $\Q_{2}$-conjugacy classes of elements of order $18$: by \reth{bruno1}(\ref{it:bruno1c}), 
\begin{equation*}
\operatorname{\text{Gal}}(\Q_{2}(\zeta_{18})/\Q_{2}) \cong\operatorname{\text{Gal}}(\Q_{2}(\zeta_{9})/\Q_{2}) \times \operatorname{\text{Gal}}(\Q_{2}(\zeta_{2})/\Q_{2})\cong \operatorname{\text{Gal}}(\Q_{2}(\zeta_{9})/\Q_{2}),
\end{equation*}
which is cyclic of order $6$. Thus the injective homomorphism $\map{\phi}{\operatorname{\text{Gal}}(\Q_{2}(\zeta_{18})/\Q_{2})}[\Z_{18}^{\ast}]$ is an isomorphism, and $\im{\phi}= \brak{\overline{1},\overline{5}, \overline{7}, \overline{11}, \overline{13},\overline{17}}$. We conclude from the above table of elements of $\dic{36}$ and~\reqref{conjclass} that there is a single $\Q_{2}$-conjugacy class of elements of order $18$.

\item From the above computations, for all $d\in \brak{1,2,3,4,6,9,18}$, there is a single $\Q_{2}$-conjugacy class of elements of order $d$, and hence $r_{\Q_{2}}=7$.

\item $\Q_{3}$-conjugacy classes of elements of order $4$: by \reth{bruno1}(\ref{it:bruno1a}), the injective homomorphism $\map{\phi}{\operatorname{\text{Gal}}(\Q_{3}(\zeta_{4})/\Q_{3})}[\Z_{4}^{\ast}]$ is an isomorphism. As in the case of the $\Q_{2}$-conjugacy classes of elements of order $4$, we see that there is a single $\Q_{3}$-conjugacy class of elements of order $4$.

\item $\Q_{3}$-conjugacy classes of elements of order $9$: by \reth{bruno1}(\ref{it:bruno1b}), the injective homomorphism $\map{\phi}{\operatorname{\text{Gal}}(\Q_{3}(\zeta_{9})/\Q_{3})}[\Z_{9}^{\ast}]$ is an isomorphism (both groups are of order $6$), and $\im{\phi}=\brak{\overline{1}, \overline{2}, \overline{4}, \overline{5}, \overline{7}, \overline{8}}$. As in the case of the $\Q_{2}$-conjugacy classes of elements of order $9$, we see that there is a single $\Q_{3}$-conjugacy class of elements of order $9$.

\item $\Q_{3}$-conjugacy classes of elements of order $18$: by \reth{bruno1}(\ref{it:bruno1c}), 
\begin{equation*}
\operatorname{\text{Gal}}(\Q_{3}(\zeta_{18})/\Q_{3}) \cong\operatorname{\text{Gal}}(\Q_{3}(\zeta_{9})/\Q_{3}) \times \operatorname{\text{Gal}}(\Q_{3}(\zeta_{2})/\Q_{3})\cong \operatorname{\text{Gal}}(\Q_{3}(\zeta_{9})/\Q_{3}),
\end{equation*}
which as we saw above is cyclic of order $6$. It follows that the injective homomorphism $\map{\phi}{\operatorname{\text{Gal}}(\Q_{3}(\zeta_{18})/\Q_{3})}[\Z_{18}^{\ast}]$ is an isomorphism. As in the case of the $\Q_{2}$-conjugacy classes of elements of order $18$, we see that there is a single $\Q_{3}$-conjugacy class of elements of order $18$.

\item From the above computations, for all $d\in \brak{1,2,3,4,6,9,18}$, there is a single $\Q_{3}$-conjugacy class of elements of order $d$, and hence $r_{\Q_{3}}=7$.

\item Hence the rank of $K_{-1}(\Z[\dic{36}])$ is given by $r=1-r_{\Q}+r_{\Q_{2}}-r_{\F[2]}+r_{\Q_{3}}-r_{\F[3]}= 1-7+7-3+7-3= 2$ as required.\qedhere
\end{enumerate}
\end{enumerate}
\end{proof}

We complete this section by computing $K_{-1}(\Z[G])$, where $G$ is a cyclic group of order $p^{q}$ or $2p^{q}$, where $p$ is prime and $q\in \N$, or of order $12$ or $20$. These results will also be used in the proof of \reth{tablas}. 

\begin{prop}\label{prop:kminusonecyclic}
Let $q\in \N$, and let $p$ be a prime number.
\begin{enumerate}[(a)]
\item\label{it:kminusonecyclica} The group $K_{-1}(\Z[\Z_{p^{q}}])$ is trivial.
\item\label{it:kminusonecyclicb} If $p$ is odd then $K_{-1}(\Z[\Z_{2p^{q}}])\cong \Z^{r}$, where $r=\sum_{j=1}^{q} \Bigl[\Z_{p^{j}}^{\ast}: \ang{\overline{2}}_{\Z_{p^{j}}^{\ast}} \Bigr]$, and where $\ang{\overline{2}}_{\Z_{p^{j}}^{\ast}}$ denotes the subgroup of $\Z_{p^{j}}^{\ast}$ generated by $\overline{2}$.
\item\label{it:kminusonecyclicc} The group $K_{-1}(\Z[\Z_{12}])$ is isomorphic to $\Z^{2}$, and the group $K_{-1}(\Z[\Z_{20}])$ is isomorphic to $\Z^{3}$.
\end{enumerate} 
\end{prop}

\begin{proof}
As we mentioned at the beginning of \resec{torsionkminusone}, if $G$ is Abelian then the group $K_{-1}(\Z[G])$ is torsion free. So if $G$ is one of the given groups, by~\reqref{carterkminusone}, it suffices to calculate the rank $r$ of $K_{-1}(\Z[G])$. 

\begin{enumerate}[(a)]
\item Let $p$ be prime, and let $q\in \N$. Since $\Z_{p^{q}}$ is cyclic, $r_{\Q}$ is equal to the number of divisors of $p^{q}$, hence $r_{\Q}=q+1$. The elements of $\Z_{p^{q}}$ are of order $p^{j}$, where $j\in \brak{0,1,\ldots,q}$, and the set $G_{p}'$ of $p$-regular elements of $\Z_{p^{q}}$ is equal to $\brak{e}$, hence $r_{\F[p]}=1$. We now determine the number of $\Q_{p}$-conjugacy classes. If $1\leq j\leq q$, by \reth{bruno1}(\ref{it:bruno1b}), the injective homomorphism $\map{\phi}{\operatorname{\text{Gal}}(\Q_{p}(\zeta_{p^{j}})/\Q_{p})}[\Z_{p^{j}}^{\ast}]$
is an isomorphism, so $\Z_{p^{q}}$ possesses a single $\Q_{p}$-conjugacy of elements of order $p^{j}$, and thus $r_{\Q_{p}}=q+1$. Hence $r=1-r_{\Q}+r_{\Q_{p}}-r_{\F[p]}=0$ by~\reqref{rankkminusone}, and $K_{-1}(\Z[\Z_{p^{q}}])$ is trivial.

\item\label{it:2oddprime} Let $p$ be an odd prime, let $q\in \N$, and let $x$ be a generator of $\Z_{2p^{q}}$. In order to apply~\reqref{rankkminusone}, we compute $r_{\Q}$, $r_{\F[2]}$, $r_{\F[p]}$, $r_{\Q_{2}}$ and $r_{\Q_{p}}$.
\begin{enumerate}[\textbullet]
\item Since $\Z_{2p^{q}}$ is cyclic, we have $r_{\Q}=2(q+1)$.
\item  The set $G_{p}'$ of $p$-regular elements of $\Z_{2p^{q}}$ is equal to $\bigl\{e,x^{p^{q}}\bigr\}$, thus $r_{\F[p]}=2$. 

\item We now determine $r_{\Q_{p}}$. Since $\Z_{2p^{q}}$ possesses a single element of order $m$, where $m\in \brak{1,2}$, it suffices to compute the number of $\Q_{p}$-conjugacy class of elements of order $2^{\epsilon}p^{j}$, where $\epsilon\in \brak{0,1}$, and $1\leq j\leq q$. By \reth{bruno1}(\ref{it:bruno1b}), 
for all $j\in \brak{1,\ldots,q}$, the injective homomorphism $\map{\phi}{\operatorname{\text{Gal}}(\Q_{p}(\zeta_{p^{j}})/\Q_{p})}[\Z_{p^{j}}^{\ast}]$ is an isomorphism, and so there is a single $\Q_{p}$-conjugacy class of elements of order $p^{j}$. By \reth{bruno1}(\ref{it:bruno1c}), $\operatorname{\text{Gal}}(\Q_{p}(\zeta_{2p^{j}})/\Q_{p})\cong \operatorname{\text{Gal}}(\Q_{p}(\zeta_{p^{j}})/\Q_{p}) \times \operatorname{\text{Gal}}(\Q_{p}(\zeta_{2})/\Q_{p})\cong \operatorname{\text{Gal}}(\Q_{p}(\zeta_{p^{j}})/\Q_{p})$, and by \reth{bruno1}(\ref{it:bruno1b}), the injective homomorphism $\map{\phi}{\operatorname{\text{Gal}}(\Q_{p}(\zeta_{2p^{j}})/\Q_{p})}[\Z_{2p^{j}}^{\ast}]$ is an isomorphism. Thus for all $1\leq j\leq q$, there is a single $\Q_{p}$-conjugacy class of elements of order $2p^{j}$. It follows that for every divisor $m$ of $2p^{q}$, there is a single $\Q_{p}$-conjugacy class of elements of order $m$, whence $r_{\Q_{p}}=2(q+1)$.

\item To compute $r_{F}$, where $F=\Q_{2}$ or $\F[2]$, we first make the following general remark. Since $\Z_{2p^{q}}$ is Abelian, for all $f\in\Z_{2p^{q}}$, the (usual) conjugacy class $[f]$ of $f$ is equal to $\brak{f}$. With the notation of \resec{rankkminusone}, it follows from~\reqref{conjclass} that the cardinality of the $F$-conjugacy class $[f]_{F}$ is equal to $\ord{\im{\phi}}$, where $\phi$ is as defined in~\reqref{defphi}. Since the $F$-conjugacy classes are pairwise disjoint, if $f\in \Z_{2p^{q}}$ we conclude that there are $[\Z_{\widehat{m}}^{\ast}: \im{\phi}]$ $F$-conjugacy classes of elements whose order is that of $f$. With this in mind, we compute $r_{\Q_{2}}$ and $r_{\F[2]}$.

\item To calculate $r_{\Q_{2}}$, first observe that since $\Z_{2p^{q}}$ possesses a single element of order $m$, where $m\in \brak{1,2}$, it suffices to compute the number of $\Q_{2}$-conjugacy class of elements of order $2^{\epsilon}p^{j}$, where $\epsilon\in \brak{0,1}$, and $1\leq j\leq q$. \reth{bruno1}(\ref{it:bruno1a}) imples that the image of the injective homomorphism $\map{\phi}{\operatorname{\text{Gal}}(\Q_{2}(\zeta_{p^{j}})/\Q_{2})}[\Z_{p^{j}}^{\ast}]$ is equal to $\ang{\overline{2}}_{\Z_{p^{j}}^{\ast}}$. By the above remark, it follows that the number of $\Q_{2}$-conjugacy classes of elements of order $p^{j}$ is equal to $\Bigl[\Z_{p^{j}}^{\ast}: \ang{\overline{2}}_{\Z_{p^{j}}^{\ast}} \Bigr]$. By \reth{bruno1}(\ref{it:bruno1c}), $\operatorname{\text{Gal}}(\Q_{2}(\zeta_{2p^{j}})/\Q_{2})\cong \operatorname{\text{Gal}}(\Q_{2}(\zeta_{p^{j}})/\Q_{2}) \times \operatorname{\text{Gal}}(\Q_{2}(\zeta_{2})/\Q_{2})\cong \operatorname{\text{Gal}}(\Q_{2}(\zeta_{p^{j}})/\Q_{2})$. So the image of the injective homomorphism $\map{\phi}{\operatorname{\text{Gal}}(\Q_{2}(\zeta_{2p^{j}})/\Q_{2})}[\Z_{p^{j}}^{\ast}]$ is of order $\Bigl\lvert \ang{\overline{2}}_{\Z_{p^{j}}^{\ast}}\Bigr\rvert$. Now $\bigl\lvert \Z_{2p^{q}}^{\ast}\bigr\rvert= \bigl\lvert \Z_{p^{q}}^{\ast}\bigr\rvert$, and it follows from the above remark that the number of $\Q_{2}$-conjugacy classes of elements of order $2p^{j}$ is also equal to $\Bigl[\Z_{p^{j}}^{\ast}: \ang{\overline{2}}_{\Z_{p^{j}}^{\ast}} \Bigr]$. Since the elements of $\Z_{2p^{q}}$ are of order $1$, $2$, $p^{j}$ or $2p^{j}$, where $1\leq j\leq q$, we deduce that $r_{\Q_{2}}= 2+2 \sum_{j=1}^{q} \Bigl[\Z_{p^{j}}^{\ast}: \ang{\overline{2}}_{\Z_{p^{j}}^{\ast}} \Bigr]$.

\item To calculate $r_{\F[2]}$, observe that the set of $2$-regular elements of $\Z_{2p^{q}}$ is equal to the union of $e$ with the elements of order $p^{j}$, where $j\in \brak{1,\ldots,q}$. By \reth{bruno2}, for all $j\in \brak{1,\ldots,q}$, the image of the injective homomorphism $\map{\phi}{\operatorname{\text{Gal}}(\F[2](\zeta_{p^{j}})/\F[2])}[\Z_{p^{j}}^{\ast}]$ is equal to $\ang{\overline{2}}_{\Z_{p^{j}}^{\ast}}$. Using the above remark once more, we see that the number of $2$-regular $\F[2]$-conjugacy classes of elements of order $p^{j}$ is given by $\Bigl[\Z_{p^{j}}^{\ast}: \ang{\overline{2}}_{\Z_{p^{j}}^{\ast}} \Bigr]$, and thus $r_{\F[2]}= 1+ \sum_{j=1}^{q} \Bigl[\Z_{p^{j}}^{\ast}: \ang{\overline{2}}_{\Z_{p^{j}}^{\ast}} \Bigr]$. 

\item To conclude, by~\reqref{rankkminusone}, we have $r=1-r_{\Q}+r_{\Q_{2}}-r_{\F[2]}+r_{\Q_{p}}-r_{\F[p]} =\sum_{j=1}^{q} \Bigl[\Z_{p^{j}}^{\ast}: \ang{\overline{2}}_{\Z_{p^{j}}^{\ast}} \Bigr]$ as required.
\end{enumerate}

\item Let $m=4p$, where $p\in \brak{3,5}$. In order to apply~\reqref{rankkminusone}, we now proceed to determine $r_{\Q}$, $r_{\F[2]}$, $r_{\F[p]}$, $r_{\Q_{2}}$ and $r_{\Q_{p}}$.

\begin{enumerate}[\textbullet]
\item  Since $\Z_{m}$ is cyclic, $r_{\Q}$ is equal to the number of divisors of $4p$, so $r_{\Q}=6$. 

\item The set $G_{2}'$ of the $2$-regular elements of $\Z_{4p}$ consists of the trivial element and the elements of order $p$. By \reth{bruno2}, the injective homomorphism $\map{\phi}{\operatorname{\text{Gal}}(\F[2](\zeta_{p})/\F[2])}[\Z_{p}^{\ast}]$ is an isomorphism, so there is a single $\F[2]$-conjugacy class of elements of order $p$, and hence $r_{\F[2]}=2$.

\item The set $G_{p}'$ of the $p$-regular elements of $\Z_{4p}$ consists of the trivial element, the unique element of order $2$, and the two elements of order $4$. By \reth{bruno2}, the injective homomorphism $\map{\phi}{\operatorname{\text{Gal}}(\F[p](\zeta_{4})/\F[p])}[\Z_{4}^{\ast}]$ is an isomorphism if $p=3$, and the image of $\phi$ is equal to $\brak{\overline{1}}$ if $p=5$. As in the remark regarding the $F$-conjugacy classes used in part~(\ref{it:2oddprime}), we conclude that there is a single $\F[p]$-conjugacy class of elements of order $4$ if $p=3$, and two $\F[p]$-conjugacy classes of elements of order $4$ if $p=5$. It follows that $r_{\F[p]}=3$ if $p=3$ and $r_{\F[p]}=4$ if $p=5$.

\item We now compute $r_{\Q_{2}}$. Since there is a single element of order $1$ and $2$, it suffices to determine the number of $\Q_{2}$-conjugacy classes of elements of order $4$ and of order $2^{i}p$, where $i\in \brak{0,1,2}$. By \reth{bruno1}(\ref{it:bruno1c}) (resp.\ \reth{bruno1}(\ref{it:bruno1a})), the injective homomorphism $\map{\phi}{\operatorname{\text{Gal}}(\Q_{2}(\zeta_{4})/\Q_{2})}[\Z_{4}^{\ast}]$ (resp.\ $\map{\phi}{\operatorname{\text{Gal}}(\Q_{2}(\zeta_{p})/\Q_{2})}[\Z_{p}^{\ast}]$) is an isomorphism, so there is a single $\Q_{2}$-conjugacy class of elements of order $4$ (resp.\ of order $p$). If $i\geq 1$ then $\operatorname{\text{Gal}}(\Q_{2}(\zeta_{2^{i}p})/\Q_{2})\cong \operatorname{\text{Gal}}(\Q_{2}(\zeta_{p})/\Q_{2}) \times \operatorname{\text{Gal}}(\Q_{2}(\zeta_{2^{i-1}})/\Q_{2})$, and so the injective homomorphism $\map{\phi}{\operatorname{\text{Gal}}(\Q_{2}(\zeta_{2^{i}p})/\Q_{2})}[\Z_{2^{i}p}^{\ast}]$ is an isomorphism by \reth{bruno1}(\ref{it:bruno1a}) and~(\ref{it:bruno1c}). Hence there is a single $\Q_{2}$-conjugacy class of elements of order $2^{i}p$. So for any divisor $d$ of $4p$, there is a single $\Q_{2}$-conjugacy class of elements of order $d$, hence $r_{\Q_{2}}=6$.

\item We now compute $r_{\Q_{p}}$. Since there is a single element of order $1$ and $2$, it suffices to determine the number of $\Q_{p}$-conjugacy classes of elements of order $4$ and of order $2^{i}p$, where $i\in \brak{0,1,2}$. By \reth{bruno1}(\ref{it:bruno1a}), the injective homomorphism $\map{\phi}{\operatorname{\text{Gal}}(\Q_{p}(\zeta_{4})/\Q_{p})}[\Z_{4}^{\ast}]$ is an isomorphism if $p=3$, and the image of $\phi$ is equal to $\brak{\overline{1}}$ if $p=5$. As in the remark regarding the $F$-conjugacy classes used in part~(\ref{it:2oddprime}), we conclude that there is a single $\Q_{p}$-conjugacy class of elements of order $4$ if $p=3$, and two $\Q_{p}$-conjugacy classes of elements of order $4$ if $p=5$. Since the injective homomorphism $\map{\phi}{\operatorname{\text{Gal}}(\Q_{p}(\zeta_{p})/\Q_{p})}[\Z_{p}^{\ast}]$ is an isomorphism by \reth{bruno1}(\ref{it:bruno1b}), there is a single $\Q_{p}$-conjugacy class of elements of order $p$. Further, since $\operatorname{\text{Gal}}(\Q_{p}(\zeta_{2p})/\Q_{p})\cong \operatorname{\text{Gal}}(\Q_{p}(\zeta_{p})/\Q_{p})$ by \reth{bruno1}(\ref{it:bruno1c}), it follows from \reth{bruno1}(\ref{it:bruno1b}) that there is a single $\Q_{p}$-conjugacy class of elements of order $2p$. Finally, by \reth{bruno1}(\ref{it:bruno1c}), since $\operatorname{\text{Gal}}(\Q_{p}(\zeta_{4p})/\Q_{p})\cong \operatorname{\text{Gal}}(\Q_{p}(\zeta_{p})/\Q_{p}) \times \operatorname{\text{Gal}}(\Q_{p}(\zeta_{4})/\Q_{p})$, it follows from \reth{bruno1}(\ref{it:bruno1a}) and~(\ref{it:bruno1b}) that the injective homomorphism $\map{\phi}{\operatorname{\text{Gal}}(\Q_{p}(\zeta_{4p})/\Q_{p})}[\Z_{4p}^{\ast}]$ is an isomorphism if $p=3$, and the image of $\phi$ is isomorphic to $\Z_{p}^{\ast}$ if $p=5$. As in the remark regarding the $F$-conjugacy classes used in part~(\ref{it:2oddprime}), we conclude that there is a single $\Q_{p}$-conjugacy class of elements of order $4p$ if $p=3$, and two $\Q_{p}$-conjugacy classes of elements of order $4p$ if $p=5$. Hence $r_{\Q_{p}}=6$ if $p=3$, and $r_{\Q_{p}}=8$ if $p=5$.

\item Using~\reqref{rankkminusone} and the above computations, the rank $r$ of $K_{-1}(\Z[\Z_{4p}])$ is given by $r=1-r_{\Q}+r_{\Q_{2}}-r_{\F[2]}+r_{\Q_{p}}-r_{\F[p]}$, so $r=2$ if $p=3$ and $r=3$ if $p=5$ as required.\qedhere
\end{enumerate}
\end{enumerate}
\end{proof}

\section{The lower algebraic $K$-theory of the finite subgroups of $B_{n}(\St)$ for $4\leq n\leq 11$}\label{sec:lower411}

In this section, we bring together the results of the previous sections to compute the lower algebraic $K$-theory of the finite subgroups of $B_{n}(\St)$ for $4\leq n\leq 11$. The results are summarised in the following theorem.

\begin{thm}\label{th:tablas}
For $4\leq n\leq 11$, the lower algebraic $K$-theory of the finite subgroups of $B_{n}(\St)$ is as given in Table~\ref{tab:finite}.
\renewcommand{\arraystretch}{1.2}
\begin{table}[!htbp]
\begin{center}
\begin{tabular}{|c||c|c|c|c|}
\hline
Finite group $G$ & $\wh{G}$ & $K_{-1}(\Z[G])$ & $\widetilde{K}_{0}(\Z[G])$ & values of $n\geq 4$ for which\\
 & & & & $G$ is realised in $B_{n}(\St)$\\
\hline\hline

$\Z_{m}$, $m\in \brak{1,2,3,4}$ & $0$ & $0$ &  $0$  & all \\ \hline

$\Z_{5}$ & $\Z$ & $0$  &  $0$ & $n\equiv 0,1,2\bmod{5}$\\ \hline

$\Z_{6}$ & $0$ & $\Z$  &   $0$  & all \\ \hline

$\Z_{7}$ & $\Z^2$ & $0$  &   $0$  & $n\equiv 0,1,2\bmod{7}$ \\ \hline

$\Z_{8}$ & $\Z$ & $0$  &   $0$ & $n\nequiv 3 \bmod{4}$\\ \hline

$\Z_{9}$ & $\Z^2$ & $0$ &  $0$ & $n\equiv 0,1,2\bmod{9}$\\ \hline

$\Z_{10}$ & $\Z^{2}$ & $\Z$  &  $0$   & $n\equiv 0,1,2\bmod{5}$\\ \hline

$\Z_{11}$ & $\Z^{4}$ &  $0$ &  $0$ & $n\equiv 0,1,2\bmod{11}$\\ \hline

$\Z_{12}$ & $\Z$ & $\Z^2$ &  $\Z_{2}$ & $n\equiv 0,1,2\bmod{6}$\\ \hline

$\Z_{14}$ & $\Z^4$ & $\Z^{2}$ &  $0$   & $n\equiv 0,1,2\bmod{7}$\\ \hline

$\Z_{16}$ & $\Z^4$ & $0$ &  $\Z_{2}$  & $n\equiv 0,1,2\bmod{8}$\\ \hline

$\Z_{18}$ & $\Z^4$ & $\Z^{2}$ &  $\Z_3$& $n\equiv 0,1,2\bmod{9}$\\ \hline

$\Z_{20}$ & $\Z^{5}$ & $\Z^{3}$  &$\Z_2^5$  & $n\equiv 0,1,2\bmod{10}$\\ \hline

$\Z_{22}$ & $\Z^{8}$ & $\Z$  &$\Z_3 $ & $n\equiv 0,1,2\bmod{11}$\\ \hline\hline

$\quat$ & $0$ & $0$ & $\Z_{2}$ & $n$ even\\ \hline

$\dic{12}$ & $0$ & $\Z$  &  $\Z_{2}$  & $n\equiv 0,2\bmod{3}$\\ \hline

$\quat[16]$ & $\Z$ & $\Z_{2}$ & $\Z_{2}$ & $n$ even\\ \hline

$\dic{20}$ & $\Z^2$ & $\Z_{2}\oplus \Z$  &  $\Z_{2}$ & $n\equiv 0,2\bmod{5}$\\ \hline

$\dic{24}$ & $\Z$ & $\Z_{2}\oplus \Z^{2}$ & $\Z_{2}^{3}$ & $n\equiv 0,2\bmod{6}$\\ \hline

$\tstar$ & $0$ & $\Z$ & $\Z_{2}$ & $n$ even\\ \hline

$\dic{28}$ & $\Z^4$ & $\Z$  &  $\Z_{2}$  & $n\equiv 0,2\bmod{7}$\\ \hline

$\quat[32]$ & $\Z^{4}$ & $\Z_{2}$ & $\Z_{2}$ & $n\equiv 0,2\bmod{8}$\\ \hline

$\dic{36}$ & $\Z^4$ & $\Z^{2}$  &  $\Z_{2}^2$  & $n\equiv 0,2\bmod{9}$\\ \hline

$\dic{40}$ & $\Z^{5}$ & $\Z_{2}\oplus \Z^{2}$ & $\Z_{2}^{3}$ & $n\equiv 0,2\bmod{10}$\\ \hline

$\dic{44}$ & $\Z^{8}$ & $\Z$ & $\Z_{2}$ & $n\equiv 0,2\bmod{11}$\\ \hline

$\ostar$ & $\Z$ &  $\Z_{2} \oplus \Z$ & $\Z_{2}^2$ & $n\equiv 0,2\bmod{6}$\\ \hline

\end{tabular}
\caption{The lower algebraic $K$-theory of the finite subgroups of $B_{n}(\St)$, $n\leq 11$.}\label{tab:finite}
\end{center}
\end{table}
\end{thm}

\begin{rem}\label{rem:negkfinite}
It was proved in~\cite[Theorem 3]{C3} that $K_{-i}(\Z[G])=0$ for any finite group $G$ and for all $i\geq 2$. 
\end{rem}

\begin{rem}
Although the results of \reth{tablas} deal with the lower algebraic $K$-theory of the finite subgroups of $B_{n}(\St)$ for $4\leq n\leq 11$, these groups also occur as subgroups of $B_{n}(\St)$ for larger values of $n$. These values are given in the last column of Table~\ref{tab:finite}, and are obtained from  \reth{finitebn}.
\end{rem}

\begin{proof}[Proof of \reth{tablas}] 
By \reth{finitebn}, the isomorphism classes of the maximal finite subgroups of $B_{n}(\St)$, where $n$ runs over the elements of $\brak{4,\ldots,11}$, are $\dic{4m}$, where $m\in \brak{3,4,\ldots,11}$, $\Z_{2q}$, where $q\in \brak{4,5,\ldots, 10}$, $\tstar$ and $\ostar$. If we remove the condition of maximality of these subgroups, then we must also add $\quat$ and $\Z_{r}$ to the list, where $r\in\brak{1,2,\ldots,7,9,11,22}$. We thus obtain the groups of the first column of Table~\ref{tab:finite} that are subgroups of $B_{n}(\St)$ for the values of $n$ given in the last column. We divide the proof into several parts.

\begin{enumerate}[(a)]
\item If $G$ is one of the groups appearing in Table~\ref{tab:finite} then $\wh{G}$ is obtained by applying \repr{whbnS2}.

\item We determine $\widetilde{K}_{0}(\Z[G])$, where $G$ is one of the groups appearing in Table~\ref{tab:finite}. By \reth{kzerotrivial}, $\widetilde{K}_{0}(\Z[G])$ is trivial if $G$ is isomorphic to $\Z_{n}$, where $n\in \brak{1,\ldots,11,14}$, and is non trivial if $G$ is isomorphic to $\Z_{n}$, where $n\in \brak{12,16,18,20,22}$. By~\cite[page~126, line~16]{Sw3}, if $G$ is isomorphic to $\Z_{12}$ then $\widetilde{K}_{0}(\Z[\Z_{12}])\cong \Z_{2}$. Suppose that $G=\Z_{16}$. From~\cite[Page~416]{KM}, there is a surjective homomorphism from $\widetilde{K}_{0}(\Z[\Z_{16}])$ to $\prod_{\nu=1}^{4} \widetilde{K}_{0}(\Z[\zeta_{\nu}])$, where $\zeta_{\nu}$ is a primitive $2^{\nu}$\up{th} root of unity, and whose kernel, denoted by $W_{3}$ in~\cite{KM}, is isomorphic to $\Z_{2}$. From~\reth{kzerotrivial}, $\widetilde{K}(\Z[\zeta_{\nu}])$ is trivial for all $n=1,\ldots,4$, and so $\widetilde{K}_{0}(\Z[\Z_{16}])\cong \Z_{2}$. By
\reth{K0cyc}, $\widetilde{K}_0(\Z[\Z_{n}])$ is isomorphic to $\Z_3$ if $n\in \brak{18,22}$, and to $\Z_2^{5}$ if $n=20$. If $G=\dic{4m}$, where $2\leq m\leq 11$, or if $G=\tstar$ or $\ostar$ then $\widetilde{K}_{0}(\Z[G])$ is given by \reth{swan}.

\item We determine $K_{-1}(\Z[G])$, where $G$ is one of the groups appearing in Table~\ref{tab:finite}. We consider several cases.

\begin{enumerate}[(i)]
\item If $G=\Z_{m}$, where $m\in \brak{1,2,3,4,5,7,8,9,11,16}$, then $K_{-1}(\Z[G])$ is trivial by \repr{kminusonecyclic}(\ref{it:kminusonecyclica}).

\item Let $G=\Z_{m}$, where $m\in \brak{6,10,14,18,22}$. Then $m$ is of the form $m=2p^{q}$, where $q\in \N$ and $p$ is an odd prime. By \repr{kminusonecyclic}(\ref{it:kminusonecyclicb}), $K_{-1}(\Z[G])\cong \Z^{r}$, where $r=\sum_{j=1}^{q} \Bigl[\Z_{p^{j}}^{\ast}: \ang{\overline{2}}_{\Z_{p^{j}}^{\ast}} \Bigr]$. A straightforward computation shows that $K_{-1}(\Z[\Z_{6}])\cong K_{-1}(\Z[\Z_{10}])\cong K_{-1}(\Z[\Z_{22}]) \cong \Z$, and $K_{-1}(\Z[\Z_{14}])\cong K_{-1}(\Z[\Z_{18}])\cong \Z^{2}$.

\item If $G=\Z_{m}$, where $m\in \brak{12,20}$, the results for $K_{-1}(\Z[G])$ are obtained from \repr{kminusonecyclic}(\ref{it:kminusonecyclicc}).

\item $K_{-1}(\Z[G])$ for $G=\dic{4m}$, where $2\leq m \leq 11$: we distinguish the following cases.

\begin{enumerate}[\textbullet]
\item If $m\in \brak{2,4,8}$, the results are a consequence of \repr{kminusonequat2k}.
\item If $m\in \brak{3,5,7,11}$, the results follow from \reth{kminusonedic4prime} and \repr{lambdacalc} (observe that if $m\in \brak{3,5,11}$, $\overline{2}$ generates $\Z_{m}^{\ast}$, so $\lambda=1$, while if $m=7$, $\overline{-1}\notin \ang{\overline{2}}$, but $\ord{\ang{\overline{2}}}=3$, so $\lambda=1$ also). 
\item If $m\in \brak{6,9,10}$, by~Propositions~\ref{prop:torsiondicextra} and~\ref{prop:rankdicextra}, it follows that $K_{-1}(\Z[\dic{36}])\cong \Z^{2}$, and  $K_{-1}(\Z[\dic{4m}]) \cong \Z_{2} \oplus \Z^{2}$ if $m\in \brak{6,10}$.
\end{enumerate}

\item If $G=\tstar$ or $\ostar$ then the results follow from \repr{kminusonebinpoly}.
\end{enumerate}
\end{enumerate}
This completes the proof of the results given in Table~\ref{tab:finite}.
\end{proof}

\chapter[The braid group $B_{4}(\St)$, and the conjugacy classes of its maximal virtually cyclic subgroups]{The braid group $B_{4}(\St)$, and the conjugacy classes of its maximal virtually cyclic subgroups}\label{chap:b4s2}

\chaptermark{$B_{4}(\St)$ and its maximal virtually cyclic subgroups}

In this chapter, we focus our attention on the braid group $B_{4}(\St)$ of the sphere on four strings. The aim is to understand the structure of its maximal virtually cyclic subgroups. These results will be used in \rechap{kthb4St} to compute the lower algebraic $K$-theory of $B_{4}(\St)$, and to prove \reth{B4}. 
 
In \resec{gensB4}, we start by recalling some properties of $B_{4}(\St)$. We then study the algebraic description of the finite subgroups of $B_{4}(\St)$ given by \reth{finitebn}, which enables us to prove in \repr{b4amalg} that $B_{4}(\St)$ may be expressed as an amalgamated product of $\tstar$ and $\quat[16]$ along their common normal subgroup that is isomorphic to $\quat$. This will alow us to show that $B_{4}(\St)$ is hyperbolic in the sense of Gromov (as we shall see in \repr{semigamma}, $B_{4}(\St)$ is virtually free). In order to do this, in \resec{maxvcb4}, we study the structure of the maximal virtually cyclic subgroups of $B_{4}(\St)$, our main result being \reth{maxvcb4}. Using~\cite{JLMP}, in \resec{conjclmaxvc}, we show that the maximal infinite virtually cyclic subgroups of $B_{4}(\St)$ possess an infinite number of conjugacy classes. 

\section{Generalities about $B_{4}(\St)$}\label{sec:gensB4}

In this section, we state several results concerning $B_{4}(\St)$. Some basic facts and results about Artin (pure)  braid groups and (pure) braid groups of surfaces may be found in Appendix~\ref{chap:braids}. As we shall see, $B_{4}(\St)$ is rather special, and possesses some very interesting properties that will allow us to calculate its lower $K$-theoretical groups. Unfortunately, if $n\geq 5$, $B_{n}(\St)$ does not share these properties. We start by recalling a presentation of $B_{4}(\St)$. 

\begin{thm}[\cite{FVB}]\label{th:fvb}
The group $B_{4}(\St)$ admits the following presentation:
\begin{enumerate}
\item[\underline{\textbf{generators:}}] $\sigma_{1}$, $\sigma_{2}$, $\sigma_{3}$. 
\item[\underline{\textbf{relations:}}] 
\begin{gather}
\sigma_{1}\sigma_{3}=\sigma_{3}\sigma_{1}\notag\\
\sigma_{1}\sigma_{2}\sigma_{1}=\sigma_{2}\sigma_{1}\sigma_{2}\notag\\
\sigma_{2}\sigma_{3}\sigma_{2}=\sigma_{3}\sigma_{2}\sigma_{3}\notag\\
\sigma_{1}\sigma_{2}\sigma_{3}^2\sigma_{2}\sigma_{1}=1.\label{eq:surfacerel}
\end{gather}
\end{enumerate}
\end{thm}

The first three `Artin relations' (also known as braid relations) will be used freely and without further comment in what follows. Since the given generators together with these three relations constitute a presentation of the Artin braid group $B_{4}$ (see~\reqref{Artin1} and~\reqref{Artin2}), $B_{4}(\St)$ is thus a quotient of $B_{4}$. 

\begin{rem}\label{rem:b4S2ab}
It follows easily from \reth{fvb} that the Abelianisation of $B_{4}(\St)$ is isomorphic to $\Z_{6}$, and that the Abelianisation homomorphism $\map{\pi}{B_{4}(\St)}[\Z_{6}]$ identifies the three generators to the single generator $\overline{1}$ of $\Z_{6}$.
\end{rem}

We may determine generators of representatives of the conjugacy classes of the finite subgroups of $B_4(\St)$ in terms of the generators of \reth{fvb} as follows. First, according to Murasugi~\cite[Theorem~A]{Mu}, any finite order element of $B_{4}(\St)$ is conjugate to a power of one of the following elements:
\begin{equation}\label{eq:defalpha}
\left\{
\begin{aligned}
\alpha_{0}&=\sigma_{1}\sigma_{2}\sigma_{3} \quad\text{(of order $8$)}\\
\alpha_{1}&=\sigma_{1}\sigma_{2}\sigma_{3}^2 \quad\text{(of order $6$)}\\
\alpha_{2}&=\sigma_{1}\sigma_{2}^2 \quad\text{(of order $4$)}.
\end{aligned}\right.
\end{equation}
Let 
\begin{equation}\label{eq:defgarside}
\garside[4]=\sigma_{1}\sigma_{2}\sigma_{3} \sigma_{1}\sigma_{2}\sigma_{1}
\end{equation}
denote the `half twist' braid on four strings. It is a square root of the full twist braid $\ft[4]$ described in~\reqref{fulltwist} (see also Figures~~\ref{fig:twists}(\subref{fig:halftwist}) and~(\subref{fig:fulltwist}) for illustrations of the half and full twist braid on six strings). The braid $\ft[4]$ generates the centre of $B_{4}(\St)$ and is the unique element of $B_{4}(\St)$ of order $2$ (this is true in general, see~\cite{GVB}). Using~\reqref{defalpha}, this latter fact implies that:
\begin{equation}\label{eq:ftorder2}
\ft[4]=\alpha_{0}^{4}=\alpha_{1}^{3}=\alpha_{2}^{2}.
\end{equation}
Let $Q=\ang{\alpha_{0}^2,\garside[4]}$. By~\cite[Theorem~1.3(3)]{GG3}, $Q$ is isomorphic to $\quat$, and is a normal subgroup of $B_{4}(\St)$. Further, it is well known (see~\cite[Lemma~29]{GG2} for example) that:
\begin{equation}\label{eq:actalpha0}
\text{$\alpha_{0}\sigma_{i}\alpha_{0}^{-1}=\sigma_{i+1}$ for $i=1,2$, and $\alpha_{0}^{2}\sigma_{3}\alpha_{0}^{-2}=\sigma_{1}$,}
\end{equation}
and that:
\begin{equation}\label{eq:actgarside}
\garside[4] \sigma_{i} \garside[4]^{-1}=\sigma_{4-i} \quad \text{for all $i=1,2,3$}
\end{equation}
Note that relations~\reqref{actalpha0} and~\reqref{actgarside} hold in $B_{4}$, and are special cases of more general relations in $B_{n}(\St)$.

\begin{rems}\mbox{}\label{rem:maxfin}
\begin{enumerate}[(a)]
\item\label{it:maxfina} By \reth{finitebn}, the isomorphism classes of the maximal finite subgroups of $B_{4}(\St)$ are $\tstar$ and $\quat[16]$.

\item\label{it:maxfinb} Within $B_{4}(\St)$, there is a single conjugacy class of each isomorphism class of $\tstar$ and $\quat[16]$~\cite[Proposition~1.5(1)]{GG2}. These subgroups may be realised algebraically as follows:
\begin{enumerate}[(i)]
\item $\quat[16]$ may be realised in $B_{4}(\St)$ as the subgroup $\ang{\alpha_{0},\garside[4]}$ ($\alpha_{0}$ is of order $8$ and $\garside[4]$ is of order $4$)~\cite{GG5}. In particular,
\begin{equation}\label{eq:conjgar}
\garside[4] \alpha_{0}\garside[4]^{-1}=\alpha_{0}^{-1} \quad \text{and} \quad \alpha_{0}^{4}=\ft[4].
\end{equation}
\item\label{it:consttstar} By~\cite[Remark~3.2]{GG2}, $\tstar$ may be realised in $B_{4}(\St)$ as the subgroup $\bigl\langle \sigma_{1}\sigma_{3}^{-1}, \garside[4] \bigr\rangle\rtimes \ang{\alpha_{1}^2}\cong \quat \rtimes \Z_{3}$. Note that the first factor of the semi-direct product is $Q$, since 
by equations~\reqref{defalpha} and~\reqref{defgarside}, we have: 
\begin{equation}\label{eq:sig1sig3}
\alpha_{0}^{-2}\garside[4]= \sigma_{3}^{-1}\sigma_{2}^{-1} \sigma_{1}^{-1} \sigma_{3}^{-1}\sigma_{2}^{-1} \sigma_{1}^{-1} \ldotp \sigma_{1} \sigma_{2} \sigma_{3} \sigma_{1} \sigma_{2} \sigma_{1}= \sigma_{1} \sigma_{3}^{-1}.
\end{equation}
\end{enumerate}
The action in the semi-direct product is given by $\alpha_{1}^2 \garside[4]\alpha_{1}^{-2}=\sigma_{1} \sigma_{3}^{-1}$ and $\alpha_{1}^2 \sigma_{1} \sigma_{3}^{-1}\alpha_{1}^{-2}=\garside[4] \sigma_{1} \sigma_{3}^{-1}=\alpha_{0}^{-2}$. The only other isomorphism class of finite non-Abelian subgroups of $B_{4}(\St)$ is that of $\quat$: the subgroups $Q=\ang{\alpha_{0}^2,\garside[4]}$ and $Q'=\ang{\alpha_{0}^2,\alpha_{0}\garside[4]}$ of $\ang{\alpha_{0},\garside[4]}$ are isomorphic to $\quat$ and realise the two conjugacy classes of $\quat$ in $B_{4}(\St)$~\cite[Proposition~1.5(2) and Theorem~1.6]{GG2}.

\item\label{it:maxfinc} By \reth{finitebn}, the remaining finite subgroups are cyclic, and as we mentioned previously, are realised up to conjugacy by powers of the $\alpha_{i}$, $i\in\brak{0,1,2}$. For each finite cyclic subgroup, there is a single conjugacy class, with the exception of $\Z_{4}$, which is realised by both of the non-conjugate subgroups $\ang{\alpha_{0}^2}$ and $\ang{\alpha_{2}}$~\cite[Proposition~1.5(2) and Theorem~1.6]{GG2}.
\end{enumerate}
\end{rems}

As we mentioned above, $Q$ is a normal subgroup of $B_{4}(\St)$. From this, we obtain the following decomposition of $B_{4}(\St)$ as an amalgamated product of two finite groups.
\begin{prop}\label{prop:b4amalg}
$B_{4}(\St) \cong \quat[16] \bigast_{\quat} \tstar$.
\end{prop}

\begin{proof}
Let $\Gamma=B_{4}(\St)/Q$. Since $\sigma_{1}\sigma_{3}^{-1}\in Q$ and $\bigl\langle\sigma_{1}\sigma_{3}^{-1}\bigr\rangle$ is not normal in $B_{4}(\St)$ by \rerems{maxfin}(\ref{it:maxfinb})(\ref{it:consttstar}), it follows that the normal closure of  $\sigma_{1}\sigma_{3}^{-1}$ in $B_{4}(\St)$ is $Q$, and that a presentation of $\Gamma$ is obtained by adjoining the relation $\sigma_{1}=\sigma_{3}$ to the presentation of $B_{4}(\St)$ given in \reth{fvb}. Thus $\Gamma$ is generated by elements $\overline{\sigma_{1}}$ and $\overline{\sigma_{2}}$, subject to the two relations $\overline{\sigma_{1}}\,\overline{\sigma_{2}}\, \overline{\sigma_{1}}=\overline{\sigma_{2}} \,\overline{\sigma_{1}}\, \overline{\sigma_{2}}$ and $(\overline{\sigma_{1}}\, \overline{\sigma_{2}} \,\overline{\sigma_{1}})^2=1$. Let $\Lambda =\setbigangr{a,b}{a^2=b^3=1}$ denote the free product $\Z_{2}\bigast \Z_{3}$, and consider the map $\map{\varphi}{\Lambda}[\Gamma]$ defined on the generators of $\Lambda$ by $\varphi(a)=\overline{\sigma_{1}}\,\overline{\sigma_{2}}\, \overline{\sigma_{1}}$ and $\varphi(b)=\overline{\sigma_{1}}\,\overline{\sigma_{2}}$. Since $(\varphi(b))^3= (\overline{\sigma_{1}}\,\overline{\sigma_{2}})^3= (\overline{\sigma_{1}}\,\overline{\sigma_{2}}\, \overline{\sigma_{1}})^2=(\varphi(a))^2=1$, $\varphi$ extends to a homomorphism that is surjective since $\varphi(b^{-1}a)= \overline{\sigma_{1}}$ and $\varphi(a^{-1}b^2)= \overline{\sigma_{2}}$. Conversely, the map $\map{\psi}{\Gamma}[\Lambda]$ defined on the generators of $\Gamma$ by $\psi(\overline{\sigma_{1}})= b^{-1}a$ and $\psi(\overline{\sigma_{2}})= a^{-1}b^2$ extends to a homomorphism since:
\begin{align*}
\psi(\overline{\sigma_{1}})\, \psi(\overline{\sigma_{2}})\, \psi(\overline{\sigma_{1}})&=b^{-1}aa^{-1}b^2 b^{-1}a=a\\
&= a^{-1}b^2 b^{-1}a a^{-1}b^2= \psi(\overline{\sigma_{2}}) \,\psi(\overline{\sigma_{1}})\, \psi(\overline{\sigma_{2}})\quad \text{and}\\
(\psi(\overline{\sigma_{1}})\, \psi(\overline{\sigma_{2}})\, \psi(\overline{\sigma_{1}}))^2 &=a^2=1,
\end{align*}
and is surjective because $\psi(\overline{\sigma_{1}}\,\overline{\sigma_{2}}\, \overline{\sigma_{1}})=a$ and $\psi(\overline{\sigma_{1}}\,\overline{\sigma_{2}})=b$. Hence $\varphi$ is an isomorphism, and $\Gamma\cong \Z_{2}\bigast \Z_{3}$.

Let $G=\quat[16] \bigast_{\quat} \tstar$, and consider the following presentation of $G$ with generators $u,v,p,q,r$ that are subject to the relations:
\begin{equation*}
\left\{
\begin{aligned}
p^2&=q^2,\, qpq^{-1}=p^{-1},\, rpr^{-1}=q,\, rqr^{-1}=pq,\, r^3=1\\
u^4&=v^2,\, vuv^{-1}=u^{-1},\, u^2=p,\, v=q,
\end{aligned}
\right.
\end{equation*}
so that $\ang{p,q,r}\cong \tstar$, $\ang{u,v}\cong \quat[16]$, and $\ang{p,q,r} \cap \ang{u,v}=H$, where $H=\ang{p,q}\cong \quat$. It follows from this presentation that $H$ is normal in $G$ and $G/H\cong \Z_{2}\bigast \Z_{3}$. Let $\map{f}{G}[B_{4}(\St)]$ be the map defined on the generators of $G$ by $f(u)=\alpha_{0}^{-1}$, $f(p)=\alpha_{0}^{-2}$, $f(v)=f(q)=\garside[4]$ and $f(r)=\alpha_{1}^2$. Using \rerems{maxfin}(\ref{it:maxfinb}), we see that $f$ respects the relations of $G$, and so extends to a homomorphism that sends $H$ isomorphically onto $Q$. Further, $\alpha_{1}^3=\ft[4]$ by \req{ftorder2}. Thus $\alpha_{1}=\ft[4] \alpha_{1}^{-2}$, and since $B_{4}(\St)=\ang{\alpha_{0},\alpha_{1}}$ by~\cite[Theorem~3]{GG4}, we conclude that $f$ is surjective. We thus obtain the following commutative diagram of short exact sequences:
\begin{equation*}
\xymatrix{1\ar[r] & H \ar[r]\ar[d]^{f\left\lvert_{H}\right.}_{\cong} & G  \ar[r]\ar[d]^{f} & G/H \ar[r]\ar[d]^{\widehat{f}} & 1\\
1\ar[r] & Q \ar[r] & B_{4}(\St) \ar[r] & \Gamma  \ar[r] & 1,}
\end{equation*}
where $\widehat{f}$ is the homomorphism induced by $f$ on the quotients. Since $f$ is surjective, $\widehat{f}$ is too, and the isomorphisms $G/H\cong \Z_{2}\bigast \Z_{3}\cong \Gamma$ and the fact that $\Z_{2}\bigast \Z_{3}$ is Hopfian (see~\cite{DN} for example) imply that $\widehat{f}$ is an isomorphism. The result is then a consequence of the $5$-Lemma.
\end{proof}

\begin{rem}\label{rem:b4S2hyperbolic}
By \repr{b4amalg}, $B_{4}(\St)$ is isomorphic to an amalgam of finite groups. Using Bass-Serre theory of groups acting on trees, it follows that it is a virtually free group, and so is hyperbolic in the sense of Gromov (see~\cite{Grom} and~\cite[Section~1.1]{JLMP}). This important fact will be crucial in the computation of the lower algebraic $K$-theory of $B_{4}(\St)$.
\end{rem}
   
\section{Maximal virtually cyclic subgroups of $B_{4}(\St)$}\label{sec:maxvcb4}

As we mentioned at the beginning of \resec{classvc}, an infinite virtually cyclic group $\Gamma$ is isomorphic to one of the following:
\begin{enumerate}[(I)]
\item a semi-direct product of the form $\Gamma\cong F\rtimes_\alpha \Z$, where $F$ is a finite group and the action $\alpha$ belongs to $\operatorname{Hom}(\Z,\aut{F})$. Such a group $\Gamma$ surjects onto $\Z$ with finite kernel $F$.
\item an amalgamated product of the form $\Gamma\cong G_1\bigast_F G_2$, where $G_1$ and $G_{2}$ are finite groups containing a common subgroup $F$ of index $2$ in both $G_{1}$ and $G_{2}$. Such a group $\Gamma$ surjects onto the infinite dihedral group $\dih{\infty}$ with finite kernel $F$.
\end{enumerate}
We shall say that these infinite virtually cyclic groups are of \emph{Type~I} or of \emph{Type~II} respectively.

Recall from \rerems{maxfin}(\ref{it:maxfina}) and~(\ref{it:maxfinb}) that up to isomorphism, the maximal finite subgroups of $B_{4}(\St)$ are $\quat[16]$ and $\tstar$, and that there exists a single conjugacy class of each. Since $\operatorname{Out}(\quat)\cong \sn[3]$, there are three isomorphism classes of Type~I groups of the form $\quat\rtimes \Z$, and that we denote by $\quat\rtimes_{j} \Z$, where $j\in\brak{1,2,3}$, and for which the action is of order $j$~\cite[Definition~4(1)(e)]{GG2}. More precisely, if we take the presentation of $\quat$ given by~\reqref{presdic} and adjoin a new generator $z$:
\begin{enumerate}[(i)]
\item $\quat\rtimes_{1} \Z$ is the group obtained by adding the relations $[z,x]= [z,y]=1$, where $[u,v]=uvu^{-1}v^{-1}$ denotes the commutator of the elements $u$ and $v$, so $\quat\rtimes_{1} \Z\cong \quat \times \Z$.
\item $\quat\rtimes_{2} \Z$ is the group obtained by adding the relations $zxz^{-1}=y$ and $zyz^{-1}=x$ (so $zxyz^{-1}=(xy)^{-1}$).
\item $\quat\rtimes_{3} \Z$ is the group obtained by adding the relations $zxz^{-1}=y$ and $zyz^{-1}=xy$ (so $zxyz^{-1}=x$).
\end{enumerate}
Up to a finite number of exceptions, the isomorphism classes of the infinite virtually cyclic subgroups of $B_{n}(\St)$ were classified in~\cite{GG2} for all $n\geq 4$. In the case $n=4$, the classification is as follows.
\begin{thm}[{\cite[Theorem~5]{GG2}}]\label{th:vcb4S2}
Every infinite virtually cyclic subgroups of $B_{4}(\St)$ is isomorphic to one of the following groups:
\begin{enumerate}[(a)]
\item\label{it:vcb4S2a} \emph{subgroups of Type~I:} $\Z_{k}\times \Z$, $k\in \brak{1,2,4}$; $\Z_{4}\rtimes\Z$ for the non-trivial action; and $\quat \rtimes_{j} \Z$ for $j\in \brak{1,2,3}$.
\item\label{it:vcb4S2b} \emph{subgroups of Type~II:} $\Z_{4} \bigast_{\Z_{2}} \Z_{4}$, $\Z_{8} \bigast_{\Z_{4}} \Z_{8}$, $\Z_{8} \bigast_{\Z_{4}} \quat$, $\quat \bigast_{\Z_{4}} \quat$, $\quat[16] \bigast_{\quat} \quat[16]$.
\end{enumerate}
\end{thm}

For each of the Type~II subgroups given in \reth{vcb4S2}(\ref{it:vcb4S2b}), abstractly there is a single isomorphism class, with the exception of $\quat[16] \bigast_{\quat} \quat[16]$ for which there are two isomorphism classes~\cite[Proposition~11]{GG2}. In this exceptional case, we recall the following result concerning the structure of the two classes, and their realisation in $B_{4}(\St)$.
\begin{prop}[{\cite[Propositions~11 and~78]{GG2}}]\label{prop:q16q8}
Abstractly, there are exactly two isomorphism classes of the amalgamated product $\quat[16] \bigast_{\quat} \quat[16]$, possessing the following presentations:
\begin{equation}\label{eq:Gamma1}
\hspace*{-4mm}\Gamma_{1}\!=\!\setbigangr{a,b,x,y}{a^{4}\!=\!b^{2},\, x^{4}\!=\!y^{2},\; bab^{-1}\!=\!a^{-1},\; yxy^{-1}\!=\!x^{-1},\; x^{2}\!=\!a^{2},\; y\!=\!b}
\end{equation}
and 
\begin{equation}\label{eq:Gamma2}
\hspace*{-2mm}\Gamma_{2}\!=\!\setbigangr{a,b,x,y}{a^4\!=\!b^2,\; x^4\!=\!y^2,\; bab^{-1}\!=\!a^{-1},\; yxy^{-1}\!=\!x^{-1},\; x^2\!=\!b,\; y\!=\!a^2b}.
\end{equation}
Further, for $i\in\brak{1,2}$, $B_{4}(\St)$ possesses a subgroup $G_{i}$ isomorphic to $\Gamma_{i}$.
\end{prop}

With \rechap{kthb4St} in view, most of the rest of this chapter will be devoted to the problem of deciding which subgroups of $B_{4}(\St)$ are maximal within the family of virtually cyclic subgroups. In what follows, we refer to such a subgroup as a \emph{maximal} virtually cyclic group (thus the word `maximal' will be used to qualify the notion of virtually cyclic group). 

\begin{thm}\label{th:maxvcb4}\mbox{}
\begin{enumerate}[(a)]
\item\label{it:maxvcb4a} Let $G$ be a maximal infinite virtually cyclic subgroup of $B_{4}(\St)$. Then $G$ is isomorphic to one of the following groups: $\quat \rtimes \Z$ for one of the three possible actions, or $\quat[16] \bigast_{\quat} \quat[16]$.
\item\label{it:maxvcb4b} If $G$ is a subgroup of $B_{4}(\St)$ isomorphic to $\tstar$ then it is maximal as a virtually cyclic subgroup.
\item\label{it:maxvcb4c} For each $j\in \brak{1,2,3}$, there are subgroups of $B_{4}(\St)$ isomorphic to $\quat\rtimes_{j}\Z$ that are maximal as virtually cyclic subgroups, and others that are non maximal. 
\item\label{it:maxvcb4d} There exist subgroups of $B_{4}(\St)$ isomorphic to $\quat[16] \bigast_{\quat} \quat[16]$ that are maximal as virtually cyclic subgroups, and others that are non maximal. 
\end{enumerate}
\end{thm}

The proof of \reth{maxvcb4} is long, and will be split into three sections, \resec{proofab}, where we shall prove parts~(\ref{it:maxvcb4a}) and~(\ref{it:maxvcb4b}), \resec{proofcd}, where we shall prove parts~(\ref{it:maxvcb4c}) and~(\ref{it:maxvcb4d}), with the exception of the case $j=1$ in part~(\ref{it:maxvcb4c}), and \resec{proofcdexcep}, where we prove part~(\ref{it:maxvcb4c}) in this exceptional case. As we mentioned in~\rerem{b4S2hyperbolic}, $B_{4}(\St)$ is hyperbolic in the sense of Gromov. The following proposition implies that there are no infinite ascending chains of infinite virtually cyclic subgroup of $B_{4}(\St)$.

\begin{prop}[{\cite[Propositions~5,6 and Remark~7]{JuLe}}]\label{prop:juanleary}
Every infinite virtually cyclic subgroup of a Gromov hyperbolic group is contained in a unique maximal virtually cyclic subgroup.
\end{prop}

\subsection{Proof of parts~(\ref{it:maxvcb4a}) and~(\ref{it:maxvcb4b}) of \reth{maxvcb4}}\label{sec:proofab}

The statement of the following proposition is that of parts~(\ref{it:maxvcb4a}) and~(\ref{it:maxvcb4b}) of \reth{maxvcb4}.

\begin{prop}\label{prop:maxvcb4a}\mbox{}
\begin{enumerate}[(a)]
\item\label{it:maxvcb4aa} Let $G$ be a maximal virtually cyclic subgroup of $B_{4}(\St)$. Then $G$ is isomorphic to one of the following groups: $\tstar$, $\quat \rtimes \Z$ for one of the three possible actions, or $\quat[16] \bigast_{\quat} \quat[16]$.
\item If $G$ is a subgroup of $B_{4}(\St)$ isomorphic to $\tstar$ then it is maximal as a virtually cyclic subgroup.
\end{enumerate}
\end{prop}

Before proving \repr{maxvcb4a}, note that if $G$ is infinite in part~(\ref{it:maxvcb4aa}), we will prove that $G$ cannot be isomorphic to one of the other infinite virtually cyclic groups of $B_{4}(\St)$ given in \reth{vcb4S2}. The question of whether there actually  exist maximal virtually cyclic subgroups of $B_{4}(\St)$ isomorphic to $\quat \rtimes \Z$ or to $\quat[16] \bigast_{\quat} \quat[16]$ will be dealt with in \resec{proofcd}.

\begin{proof}[Proof of \repr{maxvcb4a}]
Suppose that $G$ is a maximal virtually cyclic subgroup of $B_{4}(\St)$. 
\begin{enumerate}[(a)]
\item \begin{enumerate}[(i)]
\item\label{it:maxvcb4aai} First assume that $G$ is finite. Then $G$ is a maximal finite subgroup of $B_{4}(\St)$, so is isomorphic to either $\quat[16]$ or $\tstar$ by \reth{finitebn}. Suppose that $G\cong \quat[16]$. Since $B_{4}(\St)$ possesses a single conjugacy class of subgroups isomorphic to $\quat[16]$ by \rerems{maxfin}(\ref{it:maxfinb}), by \repr{q16q8}, there exists a subgroup of $B_{4}(\St)$ isomorphic to the amalgamated product $\quat[16] \bigast_{\quat} \quat[16]$, of which one of the factors is $G$, so $G$ is not maximal as a virtually cyclic subgroup. So $G$ must be isomorphic to $\tstar$.

\item Now assume that $G$ is infinite. We separate the cases where $G$ is of Type~I and Type~II respectively.
\begin{enumerate}[(A)]
\item We first suppose that $G$ is of Type~I, so $G= F\rtimes \Z$, for some action of $\Z$ on $F$, where $F$ is finite and is the torsion subgroup of $G$. Suppose that $F$ is either trivial or is isomorphic to $\Z_{2}$ or $\Z_{4}$, and let $u$ be a generator of the $\Z$-factor of $G$. Up to conjugation, we claim that $F\subset \ang{\alpha_{0}^{2}}$. If $F$ is  trivial or isomorphic to $\Z_{2}$ then $F\subset \ang{\ft[4]}\subset \ang{\alpha_{0}^{2}}$ since $\alpha_{0}^{4}=\ft[4]$. So suppose that $F\cong \Z_{4}$. By \rerems{maxfin}(\ref{it:maxfinc}), $B_{4}(\St)$ admits two conjugacy classes of subgroups isomorphic to $\Z_{4}$, generated respectively by $\alpha_{0}^{2}$ and $\alpha_{2}$. But since $u$ normalises $F$ and the normaliser of $\ang{\alpha_{2}}$ in $B_{4}(\St)$ is finite~\cite[Proposition~8(b)]{GG2}, it follows that $F$ is conjugate to $\ang{\alpha_{0}^{2}}$. This proves the claim, and so conjugating $G$ if necessary, we may suppose that $F\subset \ang{\alpha_{0}^{2}}$. Since $Q=\ang{\alpha_{0}^{2},\garside[4]}$ is normal in $B_{4}(\St)$ and $Q\cong \quat$ by~\cite[Theorem~1.3(3)]{GG3}, the subgroup $\ang{\alpha_{0}^{2},\garside[4],u}$ is isomorphic to one of the three Type~I groups $\quat\rtimes_{j} \Z$ of \reth{vcb4S2}(\ref{it:vcb4S2a}), where $j\in \brak{1,2,3}$, and admits $\ang{\alpha_{0}^{2},u}$ as a proper subgroup. Now $G$ is a subgroup of $\ang{\alpha_{0}^{2},u}$, so $G$ is non maximal as a virtually cyclic subgroup of $B_{4}(\St)$. The result in this case is then a consequence of \reth{vcb4S2}(\ref{it:vcb4S2a}).

\item Now suppose that $G$ is a Type~II subgroup of $B_{4}(\St)$ that is non isomorphic to $\quat[16] \bigast_{\quat} \quat[16]$. By \reth{vcb4S2}(\ref{it:vcb4S2b}), we may write $G=G_{1} \bigast_{H} G_{2}$, where either:
\begin{enumerate}[(1)]
\item\label{it:amalprodi} $G_{1}$ and $G_{2}$ are subgroups of $B_{4}(\St)$ isomorphic to $\quat$ or $\Z_{8}$, and $H=G_{1}\cap G_{2}$ is isomorphic to $\Z_{4}$, or
\item\label{it:amalprodii} $G_{1}$ and $G_{2}$ are subgroups of $B_{4}(\St)$ isomorphic to $\Z_{4}$, and $H=G_{1}\cap G_{2}=\ang{\ft[4]}$. 
\end{enumerate}
Note that $G_{1}$ and $G_{2}$ are not necessarily isomorphic. By \rerems{maxfin}, in $B_{4}(\St)$, there are two conjugacy classes of subgroups isomorphic to $\quat$ represented by $Q$ and $Q'$, one conjugacy class of subgroups isomorphic to $\Z_{8}$, represented by $\ang{\alpha_{0}}$, and two conjugacy classes of subgroups isomorphic to $\Z_{4}$, represented by $\ang{\alpha_{0}^{2}}$ and $\ang{\alpha_{0} \garside[4]}$ (this is because the elements $\alpha_{0} \garside[4]$ and $\alpha_{2}$ generate conjugate subgroups of order $4$). Conjugating $G$ if necessary, we may suppose that $G_{1}$ is equal to $Q$, $Q'$ or $\ang{\alpha_{0}}$ in case~(\ref{it:amalprodi}), and is equal to $\ang{\alpha_{0}^{2}}$ or $\ang{\alpha_{0} \garside[4]}$ in case~(\ref{it:amalprodii}). Furthermore, there exists $\lambda\in B_{4}(\St)$ such that $G_{2}=\lambda G_{2}' \lambda^{-1}$, where $G_{2}'$ is equal to $Q$, $Q'$ or $\ang{\alpha_{0}}$ in case~(\ref{it:amalprodi}), and is equal to $\ang{\alpha_{0}^{2}}$ or $\ang{\alpha_{0} \garside[4]}$ in case~(\ref{it:amalprodii}). Set $L=\ang{\alpha_{0}, \garside[4]}$. Then $G_{1}$ and $G_{2}'$ are subgroups of $L$, and $Q$ is a subgroup of $L$ that is normal in $B_{4}(\St)$, so $Q$ is a subgroup of $L \cap \lambda L \lambda^{-1}$. Since $L\cong \lambda L \lambda^{-1}\cong \quat[16]$ and $G=\ang{G_{1}\cup G_{2}} \subsetneqq \ang{L\cup \lambda L \lambda^{-1}}$, it follows that $\ang{L\cup \lambda L \lambda^{-1}}$ is infinite and $L \cap \lambda L \lambda^{-1}=Q$ because $Q$ is of index $2$ in both $L$ and $\lambda L \lambda^{-1}$. We conclude from~\cite[Lemma~15]{GG8} that $\ang{L\cup \lambda L \lambda^{-1}}\cong L \bigast_{Q} L\cong \quat[16] \bigast_{\quat} \quat[16]$. Thus $G$ is a non-maximal virtually cyclic subgroup of $B_{4}(\St)$, and it follows from \reth{vcb4S2}(\ref{it:vcb4S2b}) that any maximal virtually cyclic subgroup of $B_{4}(\St)$ of Type~II must be isomorphic to $\quat[16] \bigast_{\quat} \quat[16]$.
\end{enumerate}
\end{enumerate}

\item By \reth{vcb4S2}, none of the infinite virtually cyclic subgroups of $B_{4}(\St)$ admit subgroups isomorphic to $\tstar$, so any subgroup of $B_{4}(\St)$ isomorphic to $\tstar$ is maximal as a virtually cyclic subgroup. Combined with part~(\ref{it:maxvcb4aa})(\ref{it:maxvcb4aai}) of the proof, this shows in fact that $G$ is a finite maximal virtually cyclic subgroup if and only if $G\cong \tstar$.\qedhere
\end{enumerate}
\end{proof}

This completes the proof of parts~(\ref{it:maxvcb4a}) and~(\ref{it:maxvcb4b}) of \reth{maxvcb4}.

\subsection{Proof of parts~(\ref{it:maxvcb4c}) and~(\ref{it:maxvcb4d}) of \reth{maxvcb4}}\label{sec:proofcd}

We now turn to parts~(\ref{it:maxvcb4c}) and~(\ref{it:maxvcb4d}) of \reth{maxvcb4}, which may be regarded as a converse of part~(\ref{it:maxvcb4a}) in the case that $G$ is infinite. We first prove part~(\ref{it:maxvcb4c}) of \reth{maxvcb4} with the exception of the existence of $\quat\times \Z$ as a maximal virtually cyclic subgroup of $B_{4}(\St)$, which will be dealt with in \resec{proofcdexcep}. Before doing so, we state and prove the following lemma.

\begin{lem}\label{lem:over3}
Let $\map{\pi}{B_{4}(\St)}[\Z_{6}]$ denote the Abelianisation homomorphism described in \rerem{b4S2ab}.
\begin{enumerate}[(a)]
\item\label{it:over3a} If $H$ is a subgroup of $B_{4}(\St)$ that is isomorphic to either $\Z_{8}$, $\quat$ or $\quat[16]$ then $\pi(H)\subset \ang{\overline{3}}$.
\item\label{it:over3b} If $G$ is a subgroup of $B_{4}(\St)$ that is isomorphic to an amalgamated product of one of the groups $\quat[16] \bigast_{\quat} \quat[16]$, $\quat \bigast_{\Z_{4}} \quat$, $\quat \bigast_{\Z_{4}} \Z_{8}$ or $\Z_{8} \bigast_{\Z_{4}} \Z_{8}$ then $\pi(G)\subset \ang{\overline{3}}$. 
\end{enumerate}
\end{lem}
 
\begin{proof}\mbox{}
\begin{enumerate}[(a)]
\item Consider the subgroup $K=\ang{\alpha_{0},\garside[4]}$. As we mentioned in the proof of \repr{maxvcb4a}, $K$ contains representatives of the conjugacy classes of all subgroups of $B_{4}(\St)$ that are isomorphic to $\Z_{8}$, $\quat$ or $\quat[16]$. So there exists $\lambda\in B_{4}(\St)$ such that $\lambda H \lambda^{-1}\subset K$. Now $\pi(\alpha_{0})=\overline{3}$ and $\pi(\garside[4])=\overline{0}$, thus $\pi(K)\subset \ang{\overline{3}}$, which yields the result.
\item If $G$ is a subgroup of $B_{4}(\St)$ that is isomorphic to one of the given amalgamated products then by \rerems{maxfin}(\ref{it:maxfinc}), the factors appearing in the amalgamation are subgroups of conjugates of $K$, and thus $\pi(G)\subset \pi(K)\subset \ang{\overline{3}}$ by part~(\ref{it:over3a}).\qedhere
\end{enumerate}
\end{proof}

To prove \reth{maxvcb4}(\ref{it:maxvcb4c}), for each $j\in\brak{1,2,3}$, we shall exhibit two subgroups of $B_{4}(\St)$ that are isomorphic $\quat \rtimes_{j}\Z$, one of which is maximal as a virtually cyclic subgroup of $B_{4}(\St)$, and the other of which is non maximal.  For the case $j=1$, the proof of the existence of a maximal virtually cyclic subgroup of $B_{4}(\St)$ that is isomorphic to $\quat \times \Z$ is long, and will be treated separately in \resec{proofcdexcep}. With the exception of this case, the statement of the following proposition is that of parts~(\ref{it:maxvcb4c}) and~(\ref{it:maxvcb4d}) of \reth{maxvcb4}.

\begin{prop}\label{prop:maxq8timesz}\mbox{}
\begin{enumerate}
\item For each $j\in \brak{1,2,3}$, there are subgroups of $B_{4}(\St)$ isomorphic to $\quat\rtimes_{j}\Z$ that are non maximal as virtually cyclic subgroups. 
\item For each $j\in \brak{2,3}$, there are subgroups of $B_{4}(\St)$ isomorphic to $\quat\rtimes_{j}\Z$ that are maximal as virtually cyclic subgroups.
\item\label{it:q16q8max} There exist subgroups of $B_{4}(\St)$ isomorphic to $\quat[16] \bigast_{\quat} \quat[16]$ that are maximal as virtually cyclic subgroups, and others that are non maximal. 
\end{enumerate}
\end{prop}

\begin{proof}[Proof of \repr{maxq8timesz}]
\mbox{}
\begin{enumerate}
\item By \repr{q16q8}, for $i=1,2$, $B_{4}(\St)$ possesses a subgroup $G_{i}$ that is isomorphic to the amalgamated product $\Gamma_{i}$ given by equations~\reqref{Gamma1} and~\reqref{Gamma2}, and so admits a presentation given by the corresponding equation. The amalgamating subgroup $\Gamma=\ang{a^{2},b}=\ang{x^{2},y}$ is isomorphic to $\quat$, and the element $a^{-1}x$ is a product of elements chosen alternately from the two sets 
$\ang{a,b}\setminus \ang{a^2,b}$ and $\ang{x,y}\setminus \ang{x^2,y}$, so is of infinite order by standard properties of amalgamated products. Consider the subgroup $H_{i}=\ang{\Gamma_{i}\cup \brak{a^{-1}x}}$ of $G_{i}$. One may check that $\ang{a^{-1}x}$ acts by conjugation on $\ang{a^{2},b}$. If $i=1$ then:
\begin{equation}\label{eq:actxa}\left\{
\begin{aligned}
a^{-1}x \ldotp a^{2}\ldotp x^{-1}a &= a^{2}\\
a^{-1}x \ldotp b \ldotp x^{-1}a &=a^{-1}xyx^{-1}a= a^{-1}xyx^{-1}y^{-1} y a=a^{-1}x^{2}ya=a ba b^{-1}b=b\\
a^{-1}x \ldotp a^{2}b\ldotp x^{-1}a &=a^{2}b,
\end{aligned}\right.
\end{equation}
and hence $H_{1}\cong \quat\times \Z$. If $i=2$, a similar computation shows that conjugation by $a^{-1}x$ permutes cyclically $a^2,b^{-1}$ and $a^{-2}b^{-1}$, and thus $H_{2}\cong \quat\rtimes_{3} \Z$. In each case, $H_{i}\subsetneqq G_{i}$ because $[G_{i}: H_{i}]=2$. Now $G_{i}$ is isomorphic to $\quat[16] \bigast_{\quat} \quat[16]$, and so $H_{i}$ is non maximal as a virtually cyclic subgroup of $B_{4}(\St)$, which proves the statement for $j\in \brak{1,3}$. It thus remains to treat the case $j=2$. Using \req{actxa}, note that in $G_{1}$, the action by conjugation of $xa^{-1}x$ on $\Gamma$ is as follows:
\begin{equation}\label{eq:actxax}\left\{
\begin{aligned}
xa^{-1}x \ldotp a^{2}\ldotp x^{-1}ax^{-1} &= xa^{2}x^{-1}=a^{2}\\
xa^{-1}x \ldotp b \ldotp x^{-1}ax^{-1} &=xyx^{-1}=xyx^{-1}y^{-1}y= x^{2}y= a^{2} b\\
xa^{-1}x \ldotp a^{2}b\ldotp x^{-1}ax^{-1} &=a^{4}b=b^{-1}.
\end{aligned}\right.
\end{equation}
Now $xa^{-1}x$ is of infinite order, so we conclude from \req{actxax} that the subgroup $\ang{\Gamma \cup\brak{xa^{-1}x}}$ is isomorphic to $\quat \rtimes_{2} \Z$. Furthermore, this subgroup is contained (strictly) in $G_{1}$, so is non maximal.

\item First let $j=2$. Consider the subgroup $H=\ang{Q\cup\brak{\sigma_{1}}}$ of $B_{4}(\St)$. By \repr{juanleary}, $H$ is contained in a maximal virtually cyclic subgroup $M$ of $B_{4}(\St)$. Since $Q$ is normal in $B_{4}(\St)$ and $\sigma_{1}$ is of infinite order, $H$ must be isomorphic to a semi-direct product of the form $\quat\rtimes_{k}\Z$ for some $k\in\brak{1,2,3}$. To determine $k$, we study the action by conjugation of $\sigma_{1}$ on $Q$. Using \req{sig1sig3}, we have:
\begin{equation}\label{eq:actsig1}\left\{
\begin{aligned}
\sigma_{1}\ldotp \alpha_{0}^{2} \ldotp\sigma_{1}^{-1} &= \sigma_{1} \alpha_{0}^{2} \sigma_{1}^{-1}\alpha_{0}^{-2}\alpha_{0}^{2}= \sigma_{1} \sigma_{3}^{-1}\alpha_{0}^{2}\quad\text{by \req{actalpha0}}\\
&= \alpha_{0}^{-2}\garside[4] \alpha_{0}^{2}=\alpha_{0}^{-4}\garside[4]=\garside[4]^{-1} \quad\text{by \req{conjgar}}\\
\sigma_{1}\ldotp \garside[4]^{-1} \ldotp\sigma_{1}^{-1} &= \sigma_{1} \garside[4]^{-1}\sigma_{1}^{-1}\garside[4] \garside[4]^{-1}= \sigma_{1} \sigma_{3}^{-1} \garside[4]^{-1} \quad\text{by \req{actgarside}}\\
&=\alpha_{0}^{-2}\quad\text{by \req{sig1sig3}}\\
\sigma_{1}\ldotp \alpha_{0}^{2} \garside[4] \ldotp\sigma_{1}^{-1} &=\garside[4]^{-1} \alpha_{0}^{2}= \alpha_{0}^{2}\garside[4].
\end{aligned}\right.
\end{equation}
Since $\sigma_{1}$ is of infinite order, $H$ is thus isomorphic to $\quat \rtimes_{2} \Z$ because the action fixes the subgroup $\ang{\alpha_{0}^{2}\garside[4]}$ of order $4$ of $Q$, and exchanges $\ang{\alpha_{0}^2}$ and $\garside[4]$. But $\pi(\sigma_{1})=\overline{1} \notin \ang{\overline{3}}$, so $H$ is not contained in any subgroup of the form $\quat \bigast_{\Z_{4}} \quat$, $\quat \bigast_{\Z_{4}} \Z_{8}$ or $\quat[16] \bigast_{\quat} \quat[16]$ by \relem{over3}(\ref{it:over3b}). It cannot be contained either in a subgroup isomorphic to $\quat\times \Z$ or $\quat\rtimes_{3} \Z$ because the actions on $Q$ are not compatible. This implies that $M$, which is maximal in $B_{4}(\St)$ as a virtually cyclic subgroup, must also be isomorphic to $\quat \rtimes_{2} \Z$.

Now let $j=3$. As in the case $j=2$, if there exists a subgroup $L$ of $B_{4}(\St)$ that is isomorphic to $\quat\rtimes_{3} \Z$, it cannot be contained in a subgroup of $B_{4}(\St)$ isomorphic to $\quat \times \Z$ or to $\quat\rtimes_{2} \Z$. Moreover, by \relem{over3}(\ref{it:over3b}), if $\pi(L) \nsubset \ang{\overline{3}}$ then $L$ is not contained in any subgroup of $B_{4}(\St)$ isomorphic to $\quat \bigast_{\Z_{4}} \quat$, $\Z_{8} \bigast_{\Z_{4}} \Z_{8}$, $\quat \bigast_{\Z_{4}} \Z_{8}$ or $\quat[16] \bigast_{\quat} \quat[16]$. As in the previous paragraph, we conclude using \repr{juanleary} that $L$ is contained in a maximal virtually cyclic subgroup of $B_{4}(\St)$ that must also be isomorphic to $\quat\rtimes_{3} \Z$. To prove the result, we exhibit such a subgroup $L$. Consider the action by conjugation of $\sigma_{2}$ on $Q$:
\begin{equation}\label{eq:actsig2}\left\{
\begin{aligned}
\sigma_{2}\ldotp \alpha_{0}^{2} \ldotp\sigma_{2}^{-1} &= \alpha_{0}\ldotp \alpha_{0}^{-1}\sigma_{2} \alpha_{0}\ldotp \alpha_{0} \sigma_{2}^{-1}\alpha_{0}^{-1}\ldotp\alpha_{0}\\
&= \alpha_{0}\sigma_{1}\sigma_{3}^{-1}\alpha_{0} \quad\text{by \req{actalpha0}}\\
&=\alpha_{0}\alpha_{0}^{-2}\garside[4]\alpha_{0}\quad\text{by \req{sig1sig3}}\\
&= \alpha_{0}^{-2} \garside[4]= (\alpha_{0}^{2}\garside[4])^{-1}\quad\text{by \req{conjgar}}\\
\sigma_{2}\ldotp \garside[4] \ldotp\sigma_{2}^{-1} &= \sigma_{2} \garside[4]\sigma_{2}^{-1}\garside[4]^{-1}\ldotp \garside[4]= \garside[4]\\
\sigma_{2}\ldotp \alpha_{0}^{2} \garside[4] \ldotp\sigma_{2}^{-1} &=\alpha_{0}^{-2} \garside[4]\ldotp \garside[4]=\alpha_{0}^{2}.
\end{aligned}\right.
\end{equation}
In particular, $\sigma_{2}^{4} x \sigma_{2}^{-4}=x$ for all $x\in Q$. This implies that the action by conjugation of $z=\sigma_{2}^{7}\sigma_{1}$ on the elements of $Q$ is the same as that of $\sigma_{2}^{3}\sigma_{1}$. By equations~\reqref{actsig1} and~\reqref{actsig2}, this action is as follows:
\begin{align*}
& \alpha_{0}^{2} \stackrel{\sigma_{1}}{\longmapsto} \garside[4]^{-1} \stackrel{\sigma_{2}^{3}}{\longmapsto} \garside[4]^{-1}\\
& \garside[4]^{-1}\stackrel{\sigma_{1}}{\longmapsto} \alpha_{0}^{-2} \stackrel{\sigma_{2}^{3}}{\longmapsto} (\alpha_{0}^{2} \garside[4])^{-1}\\
& (\alpha_{0}^{2} \garside[4])^{-1}\stackrel{\sigma_{2}^{3}\sigma_{1}}{\longmapsto} \garside[4]^{-1} \alpha_{0}^{-2} \garside[4]=\alpha_{0}^{2}.
\end{align*}
Hence the action by conjugation of $z$ on $Q$ is of order $3$. Further, $\pi(z)=\overline{2}\notin \ang{\overline{3}}$, which shows that $L=\ang{Q\cup\brak{z}}$ is not contained in any subgroup isomorphic to an amalgamated product of the form $\quat \bigast_{\Z_{4}} \quat$, $\Z_{8} \bigast_{\Z_{4}} \Z_{8}$, $\quat \bigast_{\Z_{4}} \Z_{8}$ or $\quat[16] \bigast_{\quat} \quat[16]$ by \relem{over3}(\ref{it:over3b}). Observe that by~\reqref{sessn}, the permutation associated to $z$ is $(1,2,3)$, and so $z^{3}\in P_{4}(\St)$. To prove that $L\cong \quat\rtimes_{3}\Z$, it remains to show that $z$ is of infinite order. To achieve this, we shall write $z^{3}$ in terms of the direct product decomposition~\reqref{dirprodp4S2} of $P_{4}(\St)$, which comes down to expressing $z^{3}$ in terms of the basis $(A_{1,4}, A_{2,4})$ of the free group $\pi_{1}(\St\setminus \brak{z_{1},z_{2},z_{3}},z_{4})$ that is the kernel of the homomorphism $(p_{4,3})_{\ast}$ of~\reqref{fnsesS2}. Geometrically, this homomorphism is given by forgetting the last string (see~\reqref{fnses} and~\reqref{fnsesspec}). As mentioned in Appendix~\ref{chap:braids}, the group $P_{4}(\St)$ is generated by the set $\brak{A_{i,j}}_{1\leq i <j\leq 4}$, where $A_{i,j}$ is defined by~\reqref{defaij}, $A_{i,i+1}=\sigma_{i}^{2}$ for $i\in \brak{1,2,3}$, and the $A_{i,j}$ satisfy the `surface relations'~\reqref{surfrelpnS2} (the relations are not complete). For the convenience of the reader, we write out these relations in full:
\begin{align}
&A_{1,2}A_{1,3}A_{1,4}=1 \label{eq:relS21}\\
&A_{1,2}A_{2,3}A_{2,4}=1 \label{eq:relS25}\\
&A_{1,3}A_{2,3}A_{3,4}=1 \label{eq:relS24}\\
&A_{1,4}A_{2,4} A_{3,4}=1 \label{eq:relS22}.
\end{align}
Using~\reqref{fulltwist} and~\reqref{defaij}, one may also see that:
\begin{equation}\label{eq:relS23}
A_{1,2}A_{1,3}A_{1,4} A_{2,3}A_{2,4}A_{3,4}=\ft[4].
\end{equation}
The reader may also convince himself or herself of the validity of this relation by drawing a picture similar to that of Figure~\ref{fig:twists}(\subref{fig:fulltwist}). It follows from relations~\reqref{relS21},~\reqref{relS22} and~\reqref{relS23} that
\begin{equation}\label{eq:A23}
A_{2,3}=\ft[4] A_{3,4}^{-1}A_{2,4}^{-1}=\ft[4] A_{1,4},
\end{equation}
from relations~\reqref{relS24},~\reqref{relS22} and~\reqref{A23} that
\begin{equation}\label{eq:A13}
A_{1,3}=A_{3,4}^{-1} A_{2,3}^{-1}= A_{1,4} A_{2,4} A_{1,4}^{-1} \ft[4],
\end{equation}
and from relations~\reqref{relS25} and~\reqref{A23} that
\begin{equation}\label{eq:A12}
A_{1,2}= A_{2,4}^{-1}A_{2,3}^{-1}= A_{2,4}^{-1} A_{1,4}^{-1} \ft[4].
\end{equation}
If $i\in\brak{1,2}$, it follows from the braid relations in $B_{4}(\St)$ that $\sigma_{i}\sigma_{i+1}\sigma_{i}^{-1}=\sigma_{i+1}^{-1}\sigma_{i}\sigma_{i+1}$, and hence $\sigma_{i}\sigma_{i+1}^{k}\sigma_{i}^{-1}=\sigma_{i+1}^{-1}\sigma_{i}^{k}\sigma_{i+1}$ for all $k\in \Z$. We thus obtain:
\begin{align*}
z^{3}&= (\sigma_{2}^{7}\sigma_{1})^{3}=A_{2,3}^{3} \sigma_{2}\sigma_{1}\sigma_{2}^{7}\sigma_{1}^{-1}\ldotp \sigma_{1}^{2} \sigma_{2}^{7} \sigma_{1}= A_{2,3}^{3} \sigma_{1}^{7}\sigma_{2} \sigma_{1}^{2} \sigma_{2}^{7} \sigma_{1}\\
&= A_{2,3}^{3}A_{1,2}^{3} \sigma_{1}\sigma_{2} \sigma_{1}^{2} \sigma_{2}^{7} \sigma_{1}= A_{2,3}^{3}A_{1,2}^{3} \sigma_{1}\ldotp \sigma_{2} \sigma_{1}^{2} \sigma_{2}^{-1}\ldotp \sigma_{2}^{8} \sigma_{1}\\
&= A_{2,3}^{3}A_{1,2}^{3} A_{2,3} \sigma_{1} \sigma_{2}^{8} \sigma_{1}= A_{2,3}^{3}A_{1,2}^{3} A_{2,3} A_{1,2}\sigma_{1}^{-1} \sigma_{2}^{8} \sigma_{1}\\
&= A_{2,3}^{3}A_{1,2}^{3} A_{2,3} A_{1,2}\sigma_{2} \sigma_{1}^{8} \sigma_{2}^{-1}= A_{2,3}^{3}A_{1,2}^{3} A_{2,3} A_{1,2}A_{1,3}^{4}\\
&=(\ft[4] A_{1,4})^{3} (A_{2,4}^{-1} A_{1,4}^{-1} \ft[4])^{3} \ft[4] A_{1,4} A_{2,4}^{-1} A_{1,4}^{-1} \ft[4] (A_{1,4} A_{2,4} A_{1,4}^{-1} \ft[4])^{4}\\
&= A_{1,4}^{3}A_{2,4}^{-1} A_{1,4}^{-1}A_{2,4}^{-1} A_{1,4}^{-1} A_{2,4}^{2} A_{1,4}^{-1}
\end{align*}
by equations~\reqref{A23}~--~\reqref{A12}. But $(A_{1,4}, A_{2,4})$ is a basis of the free group $\pi_{1}(\St\setminus \brak{z_{1},z_{2},z_{3}},z_{4})$, so $z^{3}\neq 1$, and since $z^{3}\in \ang{A_{1,4}, A_{2,4}}$, it is of infinite order. We conclude that $L\cong \quat \rtimes_{3} \Z$, which completes the proof in this case.

\item The existence of subgroups of $B_{4}(\St)$ isomorphic to $\quat[16]\bigast_{\quat}\quat[16]$ that are non maximal as virtually cyclic subgroups is actually a consequence of the structure of the amalgamated product. Indeed, consider the following short exact sequence:
\begin{equation*}
1 \to \quat\to \quat[16]\bigast_{\quat}\quat[16] \stackrel{p}{\to} \Z_{2}\bigast \Z_{2}\to 1.
\end{equation*}
Now $\Z_{2}\bigast \Z_{2}$ is isomorphic to the infinite dihedral group $\dih{\infty}=\Z \rtimes \Z_{2}$. So for all $n\in \N$, $n\geq 2$, the subgroup $n\Z\rtimes \Z_{2}$ is abstractly isomorphic to $\Z\rtimes \Z_{2}$ while being a proper subgroup (in other words, it is non co-Hopfian). Thus $p^{-1}(n\Z\rtimes \Z)$ is isomorphic to $\quat[16]\bigast_{\quat}\quat[16]$ while being a proper subgroup (of index $n$). In particular, since $B_{4}(\St)$ contains a subgroup $\Gamma$ that is isomorphic to $\quat[16]\bigast_{\quat}\quat[16]$, $\Gamma$ admits proper subgroups that are also isomorphic to $\quat[16]\bigast_{\quat}\quat[16]$, and any one of these subgroups is a non-maximal virtually cyclic subgroup that is isomorphic to $\quat[16]\bigast_{\quat}\quat[16]$. Conversely, let $G$ be a subgroup of $B_{4}(\St)$ that is isomorphic to $\quat[16]\bigast_{\quat}\quat[16]$. By \repr{juanleary}, $G$ is a contained in a subgroup $M$ of $B_{4}(\St)$ that is maximal as a virtually cyclic subgroup. But \reth{vcb4S2} implies that the only isomorphism class of infinite virtually cyclic subgroups of $B_{4}(\St)$ that contains $\quat[16]$ is $\quat[16]\bigast_{\quat}\quat[16]$, and so we conclude that $M\cong \quat[16]\bigast_{\quat}\quat[16]$, which completes the proof.\qedhere
\end{enumerate}
\end{proof}
 
This proves parts~(\ref{it:maxvcb4c}) and~(\ref{it:maxvcb4d}) of \reth{maxvcb4}, with the exception of the statement of part~(\ref{it:maxvcb4c}) that pertains to the existence of maximal virtually cyclic subgroups in the case $j=1$.
 
\subsection{Proof of the existence of maximal subgroups $\quat\times \Z$ in part~(\ref{it:maxvcb4c}) of \reth{maxvcb4}}\label{sec:proofcdexcep}
 
We now complete the proof of \reth{maxvcb4}(\ref{it:maxvcb4c}) by proving the existence of maximal virtually cyclic subgroups of $B_{4}(\St)$ that are isomorphic to $\quat\times \Z$.

\begin{prop}\label{prop:maxq8timesza}
The group $B_{4}(\St)$ contains maximal virtually cyclic subgroups that are isomorphic to $\quat\times \Z$.
\end{prop}

In order to prove \repr{maxq8timesza}, we will first require two lemmas. As before, let $Q$ denote the normal subgroup $\ang{\alpha_{0}^{2},\garside[4]}$ of $B_{4}(\St)$, and let $H_{1}=\ang{\garside[4]}$, $H_{2}=\ang{\alpha_{0}^{2}}$ and $H_{3}=\ang{\alpha_{0}^{2} \garside[4]}$ be the three subgroups of $Q$ isomorphic to $\Z_{4}$. Then $B_{4}(\St)$ acts transitively on the set $\mathcal{H}=\brak{H_{1},H_{2},H_{3}}$ by conjugation, and this action gives rise to the permutation representation $\map{\psi}{B_{4}(\St)}[{\sn[3]}]$ that satisfies the following relation:
\begin{equation*}
\text{for all $1\leq i,j\leq 3$, and for all $\beta\in B_{4}(\St)$, $\bigl( \beta H_{i}\beta^{-1}=H_{j}\bigr)$} \Longleftrightarrow \bigl( \psi(\beta)(i)=j \bigr).
\end{equation*}
Note that the homomorphism $\psi$ is surjective, that $\psi(\sigma_{1})=(1,2)$ by \req{actsig1}, and that $\psi(\sigma_{2})=(2,3)$ by \req{actsig2}. Since $\sigma_{1}\sigma_{3}^{-1}\in Q$ by \req{sig1sig3}, and the action of the elements of $Q$ on $\mathcal{H}$ is trivial, it follows that $\psi(\sigma_{3})=\psi(\sigma_{1})$. If $\beta$ is of infinite order then $\ang{Q\cup \brak{\beta}}\cong \quat\rtimes \Z$, and the order of the action of $\Z$ on $\quat$ is that of the element $\psi(\beta)$. The first step is to describe $\ker{\psi}$ whose elements of infinite order will give rise to subgroups of $B_{4}(\St)$ isomorphic to $\quat\times \Z$.

\begin{lem}\label{lem:kerpsi}
$\ker{\psi}$ is isomorphic to the direct product of $Q$ with a free group $\F[2](x,y)$ of rank $2$, for which a basis $(x,y)$ is given by:
\begin{equation}\label{eq:defxy}
\text{$x=\alpha_{0}^{2}\garside[4]\sigma_{1}^{2}$ and $y=\garside[4]\sigma_{2}^{2}$.}
\end{equation}
\end{lem}

\begin{proof}
By \rerems{maxfin}(\ref{it:maxfinb}) and \repr{b4amalg}, $B_{4}(\St)$ is isomorphic to the group $\tstar \bigast_{\quat} \quat[16]$, where the $\tstar$-factor $G_{1}$ of $B_{4}(\St)$ is generated by $Q$ and $\alpha_{1}^2$, and the $\quat[16]$-factor $G_{2}$ of $B_{4}(\St)$ is generated by $Q$ and $\alpha_{0}$, so $G_{1}\cap G_{2}=Q$. Consider the canonical projection: 
\begin{equation*}
\map{\rho}{B_{4}(\St)}[B_{4}(\St)/Q].
\end{equation*}
As in the proof of \repr{b4amalg}, we identify the quotient $B_{4}(\St)/Q$ with the free product $\Z_{3}\bigast \Z_{2}$, the $\Z_{3}$- (resp. $\Z_{2}$-) factor being generated by $a=\rho(\alpha_{1})$ (resp.\ $b=\rho(\alpha_{0})$). Consider the surjective homomorphism $\map{\widehat{\psi}}{\Z_{3}\bigast \Z_{2}}[{\sn[3]}]$ defined by $\widehat{\psi}(a)=(1,3,2)$ and $\widehat{\psi}(b)=(1,3)$. Since $\psi(\alpha_{0})= \psi(\sigma_{1}\sigma_{2}\sigma_{3})=(1,2)(2,3)(1,2)=(1,3)$, $\psi(\alpha_{1})= \psi(\sigma_{1}\sigma_{2}\sigma_{3}^{2})=(1,2)(2,3)=(1,3,2)$ and $B_{4}(\St)=\ang{\alpha_{0},\alpha_{1}}$ by~\cite[Theorem~3]{GG4}, it follows that $\widehat{\psi}\circ \rho=\psi$, so $\rho$ induces a homomorphism $\map{\widehat{\rho}}{\ker{\psi}}[\operatorname{Ker}\bigl(\widehat{\psi}\mspace{1mu}\bigr)]$ of the respective kernels. We thus obtain the following commutative diagram of short exact sequences:
\begin{equation}\label{eq:z3z2}
\begin{xy}*!C\xybox{%
\xymatrix{%
& 1\ar[d] & 1 \ar[d] & \\
& \ker{\widehat{\rho}\mspace{1mu}}\ar[d]\ar[r]& Q  \ar[d]  &  \\
1 \ar[r]  & \ker{\psi} \ar[r] \ar[d]^{\widehat{\rho}} & B_{4}(\St) \ar[d]^{\rho} \ar[r]^(.57){\psi} & \sn[3]
\ar@{=}[d] \ar[r] & 1\\
1 \ar[r] & \operatorname{Ker}\bigl(\widehat{\psi}\mspace{1mu}\bigr) \ar[r] \ar[d] & \Z_{3}\bigast \Z_{2} \ar[d]\ar[r]^(.57){\widehat{\psi}} & \sn[3] \ar[r] & 1,\\
& 1 & 1 & }}
\end{xy}
\end{equation}
as well as the equality $\ker{\widehat{\rho}\mspace{1mu}}=Q$. Taking $\brak{1,a,a^{2},b,ab,a^{2}b}$ to be the Schreier transversal for $\widehat{\psi}$ and applying the Reidemeister-Schreier rewriting process~\cite{J}, we see that $\operatorname{Ker}\bigl(\widehat{\psi}\bigr)$ is a free group of rank $2$ with basis $\left((ab)^{2}, (ba)^{2}\right)$, which implies that $\ker{\psi}\cong \quat \rtimes \F[2]$ by the commutative diagram~\reqref{z3z2}. To determine the action of $\operatorname{Ker}\bigl(\widehat{\psi}\mspace{1mu}\bigr)$ on $Q$, note by~\reqref{actsig1} and~\reqref{actsig2} that $\sigma_{1}^{2}$ and $\sigma_{2}^{2}$ belong to $\ker{\psi}$, and that:
\begin{align*}
\rho(\sigma_{1}^{2})&= \rho(\sigma_{3}^{2})=\bigl(\rho(\alpha_{0}^{-1}\alpha_{1})\bigr)^{2}=(ba)^{2}\\
\rho(\sigma_{2}^{2})&= \rho(\alpha_{0} \sigma_{1}^{2} \alpha_{0}^{-1})=(ab)^{2},
\end{align*}
so $\widehat{\rho}(\sigma_{1}^{2})=(ba)^{2}$ and $\widehat{\rho}(\sigma_{2}^{2})=(ab)^{2}$. The same equations imply that the actions by conjugation of $\sigma_{1}^{2}$ and $\sigma_{2}^{2}$ on $Q$ yield elements of $\inn{Q}$, namely conjugation by $\alpha_{0}^{2}\garside[4]$ and by $\garside[4]$ respectively. Let $\map{s}{\operatorname{Ker}\bigl(\widehat{\psi}\mspace{1mu}\bigr)}[\ker{\psi}]$ be the section for $\widehat{\rho}$ defined on the basis of $ \operatorname{Ker}\bigl(\widehat{\psi}\mspace{1mu}\bigr)$ by $s\left((ba)^{2}\right)=x$ and $s\left((ab)^{2}\right)=y$. The action of these two elements on $Q$ is thus trivial, which shows that $\ker{\psi}\cong \quat \times \F[2]$ as required.
\end{proof}

Using the definition of $\psi$, a transversal of $\ker{\psi}$ in $B_{4}(\St)$ is seen to be:
\begin{equation}\label{eq:transv}
\mathcal{T}=\brak{e,\sigma_{1}, \sigma_{2},\sigma_{1}\sigma_{2}\sigma_{1}, \sigma_{1}\sigma_{2}, \sigma_{2}\sigma_{1}}.
\end{equation}
We now determine the action by conjugation of these coset representatives on $x$ and $y$. 

\begin{lem}\label{lem:conjtau}
Let $\tau\in \mathcal{T}\setminus \brak{e}$. Then 
\begin{equation*}
\tau x \tau^{-1}=
\begin{cases}
x & \text{if $\tau=\sigma_{1}$}\\
\ft[4] x^{-1} y^{-1} & \text{if $\tau=\sigma_{2}$}\\
\ft[4] y & \text{if $\tau=\sigma_{1} \sigma_{2}\sigma_{1}$}\\
\ft[4] y & \text{if $\tau=\sigma_{1}\sigma_{2}$}\\
\ft[4] x^{-1} y^{-1} & \text{if $\tau=\sigma_{2}\sigma_{1}$}
\end{cases} \;\text{and}\; 
\tau y \tau^{-1}=
\begin{cases}
y^{-1}x^{-1} & \text{if $\tau=\sigma_{1}$}\\
y & \text{if $\tau=\sigma_{2}$}\\
\ft[4] x & \text{if $\tau=\sigma_{1} \sigma_{2}\sigma_{1}$}\\
y^{-1} x^{-1} & \text{if $\tau=\sigma_{1}\sigma_{2}$}\\
\ft[4] x & \text{if $\tau=\sigma_{2}\sigma_{1}$.}
\end{cases}
\end{equation*}
\end{lem}

\begin{proof}
The action by conjugation of $\sigma_{1}$ and $\sigma_{2}$ on $\sigma_{1}^{2}$ and $\sigma_{2}^{2}$ is given by:
\begin{align*}
\sigma_{1} \sigma_{1}^{2} \sigma_{1}^{-1}&= \text{$\sigma_{1}^{2}$ and $\sigma_{2} \sigma_{2}^{2} \sigma_{2}^{-1}=\sigma_{2}^{2}$}\\
\sigma_{1} \sigma_{2}^{2} \sigma_{1}^{-1}&= \sigma_{2}^{-1}\sigma_{1}^{2}\sigma_{2}= \sigma_{2}^{-2}\ldotp \sigma_{2}\sigma_{1}^{2}\sigma_{2}= \sigma_{2}^{-2} \sigma_{3}^{-2}\;\text{by \req{surfacerel}}\\
&=\sigma_{2}^{-2} \alpha_{0}^{2} \sigma_{1}^{-2}\alpha_{0}^{-2}\;\text{by \req{actalpha0}}\\
&= \sigma_{2}^{-2} \sigma_{1}^{-2} \ldotp \sigma_{1}^{2}\alpha_{0}^{2} \sigma_{1}^{-2}\alpha_{0}^{-2}=\ft[4]\sigma_{2}^{-2} \sigma_{1}^{-2} \;\text{by equations~\reqref{conjgar} and~\reqref{actsig1}}\\
\sigma_{2} \sigma_{1}^{2} \sigma_{2}^{-1}&= \sigma_{2} \sigma_{1}^{2} \sigma_{2}^{-1}=\sigma_{2} \sigma_{1}^{2} \sigma_{2} \ldotp\sigma_{2}^{-2}= \ft[4] \sigma_{1}^{-2}\sigma_{2}^{-2} \,\text{in a similar manner.}
\end{align*}
Using also equations~\reqref{actsig1} and~\reqref{actsig2} as well as the fact that $x$ and $y$ commute with the elements of $Q$, we see that:
\begin{align*}
\sigma_{1} x \sigma_{1}^{-1} &=\sigma_{1} \alpha_{0}^{2}\garside[4]\sigma_{1}^{2} \sigma_{1}^{-1}= \alpha_{0}^{2}\garside[4]\sigma_{1}^{2}=x\\
\sigma_{1} y \sigma_{1}^{-1} &=\sigma_{1} \garside[4]\sigma_{2}^{2} \sigma_{1}^{-1}=\alpha_{0}^2 \ft[4]\sigma_{2}^{-2} \sigma_{1}^{-2}=\alpha_{0}^2 \ft[4] y^{-1}\garside[4]x^{-1}\alpha_{0}^2\garside[4]=y^{-1}x^{-1}\\
\sigma_{2} x \sigma_{2}^{-1}&=\sigma_{2} \alpha_{0}^{2}\garside[4]\sigma_{1}^{2} \sigma_{2}^{-1}= \alpha_{0}^2 \ft[4] \sigma_{1}^{-2}\sigma_{2}^{-2}= \alpha_{0}^2 \ft[4] x^{-1}\alpha_{0}^2\garside[4]y^{-1}\garside[4]=\ft[4]x^{-1}y^{-1}\\
\sigma_{2} y \sigma_{2}^{-1}&=\sigma_{2} \garside[4]\sigma_{2}^{2} \sigma_{2}^{-1}=\alpha_{0}^2 \ft[4]\sigma_{2}^{-2} \sigma_{1}^{-2}=\ft[4]\sigma_{2}^{-2}=y,
\end{align*}
from which we deduce that:
\begin{align*}
&(\sigma_{1}\sigma_{2}\sigma_{1}) x (\sigma_{1}\sigma_{2}\sigma_{1})^{-1}=\ft[4]y, & &(\sigma_{1}\sigma_{2} )x\sigma_{2}^{-1}\sigma_{1}^{-1}= \ft[4]y,\\
&(\sigma_{1}\sigma_{2}\sigma_{1}) y (\sigma_{1}\sigma_{2}\sigma_{1})^{-1}=\ft[4]x, &
&(\sigma_{1}\sigma_{2}) y\sigma_{2}^{-1}\sigma_{1}^{-1}=y^{-1}x^{-1},\\
&(\sigma_{2}\sigma_{1}) x\sigma_{1}^{-1} \sigma_{2}^{-1}= \ft[4] x^{-1}y^{-1}, &
&(\sigma_{2}\sigma_{1}) y\sigma_{1}^{-1} \sigma_{2}^{-1}= \ft[4]x.
\end{align*}
We thus obtain the relations given in the statement.
\end{proof}

\begin{proof}[Proof of \repr{maxq8timesza}]
To prove the proposition, we must show that there exists a maximal virtually cyclic subgroup of $B_{4}(\St)$ that is isomorphic to $\quat\times \Z$. Let $z\in B_{4}(\St)$ be an element of infinite order, and suppose that $\Gamma=\ang{Q\cup\brak{z}}\cong \quat \rtimes_{j} \Z$, where $j\in \brak{2,3}$. Our aim is to obtain necessary conditions on the generators of the infinite cyclic factor of those subgroups of $\Gamma$ that are isomorphic to $\quat\times \Z$. Thus will enable us to construct subgroups of $B_{4}(\St)$ that are isomorphic to $\quat\times \Z$ but are not contained in any subgroup isomorphic to $\quat\rtimes_{j}\Z$, where $j\in \brak{2,3}$. With this in mind, let $\Delta$ be a subgroup of $\Gamma$ that is isomorphic to $\quat \times \Z$. Since the finite-order elements of $\Gamma$ are precisely the elements of $Q$, the subgroup of $\Delta$ that is isomorphic to $\quat$ is $Q$. The remaining elements of $\Gamma$, of the form $q\ldotp z^{k}$, where $q\in Q$ and $k\in \Z\setminus \brak{0}$, are of infinite order. In order that such an element belong to the centraliser of $Q$ (and thus form a subgroup isomorphic to $\quat\times \Z$), the fact that the action of $z$ on $Q$ is of order $j$ implies that $k$ must be a multiple of $j$, and thus $\Delta=\ang{Q\cup \brak{q\ldotp z^{\lambda j}}}= \ang{Q\cup \brak{z^{\lambda j}}}\subset \ang{Q\cup \brak{z^{j}}}$ for some $\lambda \in \Z\setminus \brak{0}$. In particular, $\ang{Q\cup \brak{z^{j}}}$ is the maximal subgroup of $\Gamma$ that is isomorphic to $\quat\times\Z$.

Since the action by conjugation of $z$ on $Q$ is of order $j$, it follows from the definition of $\psi$ that $z$ belongs to one of the cosets $\tau\ldotp\ker{\psi}$ of $B_{4}(\St)$ where $\tau\in \mathcal{T}\setminus \brak{e}$, $\mathcal{T}$ being the transversal of \req{transv}. More precisely, $z\in \tau \ldotp \ker{\psi}$, where $\tau\in \brak{\sigma_{1},\sigma_{2}, \sigma_{1} \sigma_{2}\sigma_{1}}$ if $j=2$, and $\tau\in \brak{\sigma_{1}\sigma_{2}, \sigma_{2}\sigma_{1}}$ if $j=3$. Further, by \relem{kerpsi} there exist $v\in \ker{\psi}$, $u\in \F[2](x,y)$ and $q_{1}\in Q$ such that $z=\tau v$ and $v=q_{1}u$. Let us write $u=u(x,y)$ as a freely reduced word in $\F[2](x,y)$:
\begin{equation*}
u= x^{\epsilon_{1}} y^{\delta_{1}}\cdots x^{\epsilon_{r}} y^{\delta_{r}},
\end{equation*}
where $\epsilon_{i},\delta_{i}\in \Z$ for all $i=1,\ldots,r$, and $\delta_{1},\epsilon_{2},\ldots, \delta_{r-1},\epsilon_{r}$ are non-zero. If $v\in \ker{\psi}$, let $\overline{v}$ denote the image of $v$ under projection onto the $\F[2](x,y)$-factor, followed by Abelianisation of $\F[2](x,y)$. We now compute $\overline{z^j}$. We have that:
\begin{align*}
z^j &=(\tau q_{1}u)^j\\
&=\begin{cases}
\underbrace{\tau q_{1}\tau^{-1}}_{\in Q} \underbrace{\tau u \tau q_{1} \tau^{-1}u^{-1}\tau^{-1}}_{\in Q} (\tau u)^2 & \text{if $j=2$}\\
\underbrace{\tau q_{1}\tau^{-1}}_{\in Q} \underbrace{\tau u \tau q_{1} \tau^{-1}u^{-1}\tau^{-1}}_{\in Q} \underbrace{\tau u \tau u \tau q_{1} \tau^{-1}u^{-1}\tau^{-1}u^{-1}\tau^{-1}}_{\in Q}(\tau u)^{3} & \text{if $j=3$}
\end{cases}\\
&= q' (\tau u)^j, \quad\text{where $q'\in Q$.}
\end{align*}
Now
\begin{equation*}
(\tau u)^j=\begin{cases}
\tau u \tau^{-1}\ldotp \tau^2 \ldotp u & \text{if $j=2$}\\
(\tau u \tau^{-1})(\tau^{2} u \tau^{-2}) \ldotp\tau^{3} \ldotp u & \text{if $j=3$.}
\end{cases}
\end{equation*}
Applying \relem{conjtau}, and using \req{defxy} as well as the fact that $x$ and $y$ commute with the elements of $Q$, it follows that there exists $q''\in Q$ such that $(\tau u)^j=q'^{-1}q'' w$, where $w\in \F[2](x,y)$ is given by: 
\begin{equation*}\label{eq:tauusquare}
\!\begin{cases}
x^{\epsilon_{1}} (y^{-1}x^{-1})^{\delta_{1}}\cdots x^{\epsilon_{r}} (y^{-1}x^{-1})^{\delta_{r}} x x^{\epsilon_{1}} y^{\delta_{1}}\cdots x^{\epsilon_{r}} y^{\delta_{r}} & \!\!\text{if $\tau=\sigma_{1}$}\\
(x^{-1}y^{-1})^{\epsilon_{1}} y^{\delta_{1}}\cdots (x^{-1}y^{-1})^{\epsilon_{r}} y^{\delta_{r}} y x^{\epsilon_{1}} y^{\delta_{1}}\cdots x^{\epsilon_{r}} y^{\delta_{r}} & \!\!\text{if $\tau=\sigma_{2}$}\\
y^{\epsilon_{1}} x^{\delta_{1}}\cdots y^{\epsilon_{r}} x^{\delta_{r}} x^{\epsilon_{1}} y^{\delta_{1}}\cdots x^{\epsilon_{r}} y^{\delta_{r}} & \!\!\text{if $\tau=\sigma_{1}\sigma_{2}\sigma_{1}$}\\
y^{\epsilon_{1}} (y^{-1}x^{-1})^{\delta_{1}}\cdots y^{\epsilon_{r}} (y^{-1}x^{-1})^{\delta_{r}} 
(y^{-1}x^{-1})^{\epsilon_{1}} x^{\delta_{1}}\cdots (y^{-1}x^{-1})^{\epsilon_{r}} x^{\delta_{r}} x^{\epsilon_{1}} y^{\delta_{1}}\cdots x^{\epsilon_{r}} y^{\delta_{r}} & \!\!\text{if $\tau=\sigma_{1}\sigma_{2}$}\\
(x^{-1}y^{-1})^{\epsilon_{1}} x^{\delta_{1}}\cdots (x^{-1}y^{-1})^{\epsilon_{r}} x^{\delta_{r}}
y^{\epsilon_{1}} (x^{-1}y^{-1})^{\delta_{1}}\cdots y^{\epsilon_{r}} (x^{-1}y^{-1})^{\delta_{r}} x^{\epsilon_{1}} y^{\delta_{1}}\cdots x^{\epsilon_{r}} y^{\delta_{r}} & \!\!\text{if $\tau=\sigma_{2}\sigma_{1}$.}
\end{cases}
\end{equation*}
We have also used the fact that:
\begin{equation*}
(\sigma_{1}\sigma_{2}\sigma_{1})^{2}=(\sigma_{1}\sigma_{2})^3=(\sigma_{2}\sigma_{1})^3 =(\sigma_{1}\sigma_{3}^{-1})^{2}=\ft[4]
\end{equation*}
by equations~\reqref{surfacerel},~\reqref{conjgar} and~\reqref{sig1sig3}.
Since $z^{j}=q''w$, and $q''$ commutes with $w$, relative to the basis $(\overline{x},\overline{y})$ of the Abelianisation $\Z^{2}$ of $\F[2](x,y)$, we obtain: 
\begin{equation*}
\overline{z^{\lambda j}}=\begin{cases}
\lambda\ldotp \bigl(2(\epsilon_{1}+\cdots+ \epsilon_{r})-(\delta_{1}+\cdots +\delta_{r})+1,0\bigr) & \text{if $\tau=\sigma_{1}$}\\
\lambda\ldotp \bigl(0,2(\delta_{1}+\cdots +\delta_{r})-(\epsilon_{1}+\cdots+ \epsilon_{r})+1\bigr) & \text{if $\tau=\sigma_{2}$}\\
\lambda(\epsilon_{1}+\cdots+ \epsilon_{r}+\delta_{1}+\cdots +\delta_{r})\ldotp (1, 1) & \text{if $\tau=\sigma_{1}\sigma_{2}\sigma_{1}$}\\
(0,0) & \text{if $\tau=\sigma_{1}\sigma_{2}$ or $\tau=\sigma_{2}\sigma_{1}$,}
\end{cases}
\end{equation*}
for all $\lambda \in \Z\setminus\brak{0}$. We conclude that if $\Delta$ is a subgroup of $\Gamma=\ang{Q\cup\brak{z}}\cong \quat\rtimes_{j}\Z$ that is isomorphic to $\quat\times \Z$ then $\Delta=\ang{Q\cup \brak{q\ldotp z^{\lambda j}}}$, where $\overline{z^{\lambda j}}\in \setl{(a,b)\in \Z^2}{ab(a-b)=0}$ relative to the basis $(\overline{x},\overline{y})$ of the Abelianisation $\Z^{2}$ of $\F[2](x,y)$. 

To complete the proof of the proposition, we shall exhibit an element $w\in B_{4}(\St)$ for which:
\begin{enumerate}[(a)]
\item\label{it:conda} $w$ is a non-trivial element of $\F[2](x,y)$ such that $\overline{w}=(c,d)$, where $cd(c-d)\neq 0$.
\item\label{it:condd} $\pi(w)\notin \ang{\overline{3}}$.
\end{enumerate}
Since $\F[2](x,y) \subset \ker{\psi}$, the first condition implies that such an element $w$ is a suitable generator of the $\Z$-factor of a subgroup of $B_{4}(\St)$ that is  isomorphic to $\quat\times \Z$, but which from the above discussion, is not contained in any subgroup that is isomorphic to $\quat\rtimes_{j} \Z$ for $j\in \brak{2,3}$. By \relem{over3}(\ref{it:over3b}), the second condition implies that $\ang{Q\cup\brak{w}}$ is not contained in any subgroup of $B_{4}(\St)$ that is isomorphic to $\quat[16]\bigast_{\quat} \quat[16]$, $\quat \bigast_{\Z_{4}} \quat$, $\Z_{8} \bigast_{\Z_{4}} \Z_{8}$ or $\quat\bigast_{\Z_{4}} \Z_{8}$.

Take $w=xy^3$, and let $\Delta=\ang{Q\cup\brak{w}}$. Then $\Delta\cong \quat\times \Z$ since $w\in \ker{\psi}$ is an element of infinite order, and by \repr{juanleary}, there exists a maximal infinite virtually cyclic subgroup $M$ of $B_{4}(\St)$ that contains $\Delta$. Clearly condition~(\ref{it:conda}) above holds, and \req{defxy} implies that condition~(\ref{it:condd}) is also satisfied. It follows from the previous paragraph and \reth{vcb4S2}(\ref{it:vcb4S2b}) that $M\cong \quat\times \Z$, which completes the proof of \repr{maxq8timesza}. In conjunction with \repr{maxq8timesz}, this also proves parts~(\ref{it:maxvcb4c}) and~(\ref{it:maxvcb4d}) of \reth{maxvcb4}.
\end{proof}

This proves part~(\ref{it:maxvcb4c}) of \reth{maxvcb4} in the the exceptional case, and bringing together Propositions~\ref{prop:maxvcb4a},~\ref{prop:maxq8timesz} and~\ref{prop:maxq8timesza}, completes the proof of \reth{maxvcb4}.
  
\section{Conjugacy classes of maximal infinite virtually cyclic subgroups in $B_{4}(\St)$}\label{sec:conjclmaxvc}

In order to determine the number of conjugacy classes of maximal infinite virtually cyclic subgroups, we follow the procedure given in \cite[Section~2.5]{JLMP} based on the action of $B_{4}(\St)$ on a suitable tree. As in the proof of \repr{b4amalg}, we identify $B_{4}(\St)$ with $\quat[16] \bigast_{\quat} \tstar$, and the quotient  $B_{4}(\St)/Q$ with the modular group $\operatorname{PSL}(2,\Z)\cong \mathbb{Z}_{2} \bigast\mathbb{Z}_{3}=\setangr{a,b}{a^2=b^3=1}$. Thus we have the following short exact sequence:
\begin{equation}\label{eq:b4mod}
1\to Q \to B_{4}(\St) \stackrel{\rho}{\to} \mathbb{Z}_{2} \bigast\mathbb{Z}_{3}   \to 1,
\end{equation}
$\rho$ being the quotient map as in the proof of \repr{b4amalg}. There is a well-known action of $\operatorname{PSL}(2,\Z)$ on the tree $T$ of Figure~\ref{fig:treeT}, where the edge stabilisers are trivial and the vertex stabilisers are $\Z_2$ and $\Z_3$.
 \begin{figure*}[tbp]
\centering
\centering{\includegraphics[scale=0.82]{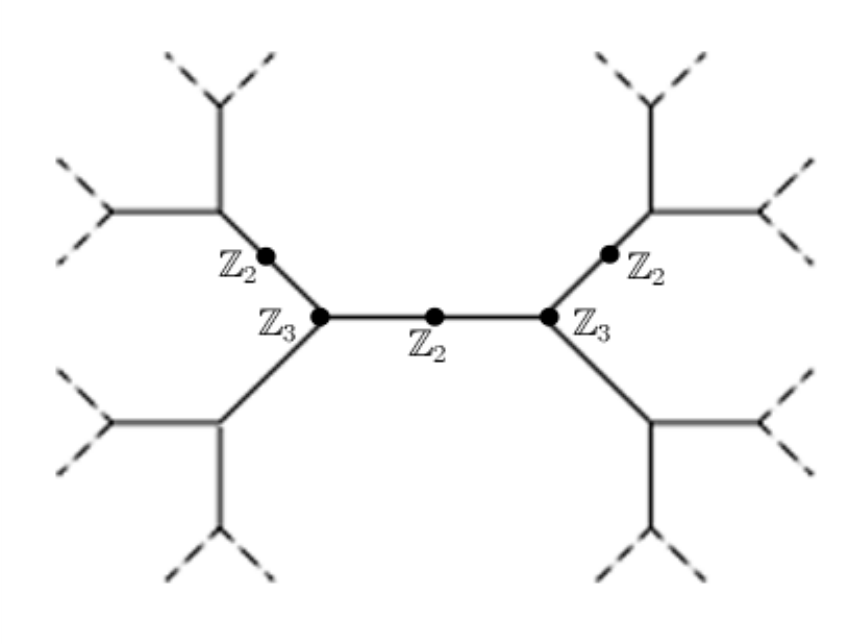}}
\caption{The tree $T$, showing the edge and vertex stabilisers under the action of $\operatorname{PSL}(2,\Z)$.}\label{fig:treeT}
\end{figure*}
The quotient of $T$ by this action is the graph:
\begin{equation*}
 \Z_2\bullet\hspace{-2mm}\underset{}{\rule[.6mm]{3cm}{.5mm}} \hspace{-2mm}\bullet\Z_3.
\end{equation*}
It follows from the short exact sequence~\reqref{b4mod} that $B_4(\St)$ acts on $T$ via $\rho$, and since $\ker{\rho}\cong \quat$, the quotient graph of  this action is:
\begin{equation*}
\quat[16]\bullet\hspace{-2mm}\underset{\quat}{\rule[.8mm]{3cm}{.5mm}} \hspace{-2mm}\bullet \tstar.
\end{equation*}
We now apply the Reidemeister-Schreier rewriting process to the Abelianisation homomorphism $\map{\widetilde{\pi}}{\mathbb{Z}_{2} \bigast\mathbb{Z}_{3}}[\Z_{6}]$. A computation similar to that given in the proof of \relem{kerpsi} shows that the commutator subgroup $\Gamma_{2}(\mathbb{Z}_{2} \bigast\mathbb{Z}_{3})$ of $\mathbb{Z}_{2} \bigast\mathbb{Z}_{3}$ is a free group, which we denote by $\mathbf{F}_2$, of rank two with basis $\bigl([a,b],[a,b^2]\bigr)$.

\begin{prop}\label{prop:semigamma}
Let $\widetilde{\mathbf{F}}=\rho^{-1}(\mathbf{F}_2)$. Then there exists a (free) subgroup $\widetilde{\mathbf{F}}_k$ of $\widetilde{\mathbf{F}}$ of rank $k\geq 2$ that is normal and of finite index in $B_4(\St)$. 
\end{prop}

\begin{rem}\label{rem:semigamma}
The above construction gives rise to the following commutative diagram of short exact sequences:
\begin{equation*}
\xymatrix{%
& & 1\ar[d] &  1\ar[d] &\\
1 \ar[r]  & Q \ar[r] \ar@{=}[d] & \widetilde{\mathbf{F}} \ar[d] \ar[r]^{\rho\left\lvert_{\widetilde{\mathbf{F}}}\right.} & \mathbf{F}_2
\ar[d] \ar[r] & 1\\
1 \ar[r] & Q \ar[r] & B_{4}(\St) \ar[d]^{\pi}\ar[r]^{\rho} &  \mathbb{Z}_{2} \bigast\mathbb{Z}_{3} \ar[d]^{\widetilde{\pi}} \ar[r] & 1.\\
& & \Z_{6}\ar[d] \ar@{=}[r] &  \Z_{6}\ar[d] &\\
& & 1 &  1 &}
\end{equation*}
We see that $\Gamma_{2}(B_{4}(\St))=\widetilde{\mathbf{F}}\cong \quat \rtimes \mathbf{F}_2$, which yields an alternative proof of the decomposition given in~\cite[Theorem~1.3(3)]{GG3}.
\end{rem}

\begin{proof}[Proof of \repr{semigamma}]
By \rerem{semigamma}, $\widetilde{\mathbf{F}}$ is isomorphic to a semi-direct product of the form $\quat\rtimes \mathbf{F}_2$. Let $\map{s}{\mathbf{F}_2}[\widetilde{\mathbf{F}}]$ be a section for $\rho\left\lvert_{\widetilde{\mathbf{F}}}\right.$. Since $s(\mathbf{F}_2)$ is of finite index in $B_4(\St)$, it suffices to take $\widetilde{\mathbf{F}}_k$ to be the intersection of the conjugates of $s(\mathbf{F}_2)$ in $B_{4}(\St)$.
\end{proof}

The group $\mathbf{F}_2$ acts freely on $T$, the resulting quotient space being a graph $\Gamma_1$ that is homotopy equivalent to a wedge of two circles. The group $\widetilde{\mathbf{F}}_k$ also acts freely on $T$ in the same way as its image $\rho(\widetilde{\mathbf{F}}_k)$ in $\mathbb{Z}_{2} \bigast\mathbb{Z}_{3}$, the quotient graph $\Gamma=T/\widetilde{\mathbf{F}}_k$ being a finite-sheeted covering space of  $\Gamma_1$.
 
By~\cite[Section~2.3]{JLMP}, there is a bijective correspondence between:
\begin{enumerate}[(a)]
\item the maximal infinite virtually subgroups of $B_4(\St)$, and
\item the stabilisers of geodesics in $T$ with infinite stabiliser.
\end{enumerate}
In order to determine the number of conjugacy classes of the maximal infinite virtually cyclic subgroups of $B_4(\St)$, we observe that since the action of $\quat$ on the quotient $T/\widetilde{\mathbf{F}}_k$ is trivial, it follows that $\pi_1((T/\widetilde{\mathbf{F}}_k)^{\quat})$ is free of rank $k\geq 2$. Therefore,  there are infinitely many conjugacy classes of maximal infinite virtually cyclic subgroups of the form $\quat\rtimes_{j}\Z$ for $j\in \brak{1,2,3}$ and of the form $\quat[16]\bigast_{\quat}\quat[16]$, see~\cite[Section~2.5]{JLMP} for more details.

\chapter{Lower algebraic $K$-theory groups of the group ring $\Z[B_4(\St)]$}\label{chap:kthb4St}

\chaptermark{$K$-theory the group ring $\Z[B_4(\St)]$}

As we mentioned in \resec{classvc}, $B_n(\St)$ is finite for all $n\leq 3$. For these values of $n$, the corresponding $K$-groups were given in Table~\ref{tab:finite}. This chapter is devoted to the computation of the lower $K$-groups of $\Z[B_{4}(\St)]$. The aim is to prove \reth{B4}, whose statement we recall here.
\begin{varthm}[\reth{B4}]
The group $B_4(\St)$ has the following lower algebraic $K$-groups:
\begin{gather*}
\wh{B_4(\St)}\cong\Z\oplus  \operatorname{Nil}_1,\\
\widetilde{K}_{0}(\Z[B_4(\St)])\cong\Z_2\oplus  \operatorname{Nil}_0,\quad\text{and}\\
K_{-1}(\Z[B_4(\St)])\cong\Z_2\oplus\Z,\\
\text{$K_{-i}(\Z[B_4(\St)])= 0$ for all $i\geq 2$},
\end{gather*}
where for $i=0,1$, the groups $\operatorname{Nil}_i$ are isomorphic to a countably-infinite direct sum of $\Z_2$, $\Z_4$ or $\Z_2\oplus \Z_4$.
\end{varthm}

The main fact that allows this computation is that $B_{4}(\St)$ is hyperbolic in the sense of Gromov (see \rerem{b4S2hyperbolic}) because it is an amalgam of finite groups by \repr{b4amalg}. Hence the Farrell-Jones fibred isomorphism conjecture holds for this group, and so we may perform the  $K$-theoretical calculations using \resec{calcingred} and~\cite{JLMP,JuLe}. All of these calculations are based on the knowledge of the lower $K$-theory groups of the virtually cyclic subgroups of $B_{4}(\St)$. In \resec{calcingred}, we recall some general facts about the lower $K$-groups of infinite virtually cyclic groups. In \resec{prelimkth}, we discuss the lower $K$-groups of the finite subgroups of $B_{4}(\St)$ and how they fit together with the infinite virtually cyclic subgroups of $B_{4}(\St)$ to give the lower $K$-groups of $\Z[B_{4}(\St)]$, up to computing the $\operatorname{Nil}_{i}$ groups. Finally, in \resec{nilgrpcomp}, we determine these groups, and we put together all of these ingredients to complete our calculations to prove \reth{B4}.

\section{The lower $K$-theory of infinite virtually cyclic groups}\label{sec:calcingred}

In this section, we provide the ingredients needed to compute the lower algebraic $K$-groups of infinite virtually cyclic groups. For a virtually cyclic group $\Gamma$ of Type~I, the algebraic $K$-groups of $\Z[\Gamma]$ are described by the Bass-Heller-Swan formula with $\alpha=1$, which asserts that for a finite group $\pi$, there is a natural decomposition~\cite{Ba}:
\begin{equation}\label{eq:bhs0}
K_i(\Z[\pi\times \Z]) = K_i(\Z[\pi]) \oplus K_{i-1}(\Z[\pi]) \oplus 2 \nkone[i]{\pi}\quad\text{for all $i\in\Z$,}
\end{equation}
where the \emph{$i\up{th}$ Bass Nil group} of $\pi$, denoted by $\nkone[i]{\pi}$, is defined to be the kernel of the homomorphism in $K$-groups induced by the evaluation $\map{e}{\Z[\pi][t]}[\Z[\pi]]$ at $t=0$. In the reduced version, \req{bhs0} takes the form:
\begin{align*}
\wh{\pi\times \Z} &= \wh{\pi} \oplus \widetilde{K}_{0}(\Z[\pi]) \oplus 2\nkone{\pi}, \quad\text{and}\\
\widetilde{K}_{0}(\Z[\pi\times \Z]) &= \widetilde{K}_{0}(\Z[\pi]) \oplus K_{-1}(\Z[\pi]) \oplus 2\nkone[0]{\pi}.
\end{align*}

If $\alpha\neq 1$, the group ring $\Z[\Gamma]$ is equal to $\Z[F\rtimes_{\alpha}\Z]\cong \Z[F]_{\alpha}[t,t^{-1}]$, the latter being the \emph{twisted} Laurent polynomial ring of $\Z[F]$, and the twisting is given by the action of $\alpha$. In this case, the Bass Nil groups are replaced by the \emph{Farrell-Hsiang Nil groups} $\nkonealt[i]{\Z[F], \alpha}\oplus \nkonealt[i]{\Z[F], \alpha^{-1}}$~\cite{FH}.

For virtually cyclic groups $\Gamma$ of Type~II, the fundamental work of Waldhausen gives rise to the following exact sequence~\cite{Wal}:
\begin{multline*}
\cdots\to K_{n}(\Z [F])\to K_{n}(\Z [G_{1}])\oplus K_{n}(\Z [G_{2}])\to K_{n}(\Z [\Gamma])/\operatorname{\textbf{Nil}}_n^W\to \\
K_{n-1}(\Z [F])\to K_{n-1}(\Z [G_{1}])\oplus K_{n-1}(\Z [G_{2}])\to K_{n-1}(\Z [G])/\operatorname{\textbf{Nil}}_{n-1}^W\to \cdots,
\end{multline*}
where $\operatorname{\textbf{Nil}}_n^W$ denotes the \emph{Waldhausen Nil groups}, denoted in~\cite{Wal} by:
\begin{equation*}
\operatorname{\textbf{Nil}}_n^W=\operatorname{Nil}^W_n(\Z[F];\Z[G_1\setminus F],\Z[G_2\setminus F]).
\end{equation*}
If $\Gamma$ is an infinite virtually cyclic group of Type~II, there is a surjection
$f\colon\thinspace \Gamma\twoheadlongrightarrow \dih{\infty}$ whose kernel $F$ is finite. Let $T$ be the unique infinite cyclic subgroup of $\dih{\infty}$ of index $2$.  Then the subgroup $\widetilde{\Gamma}=f^{-1}(T)\subset \Gamma$ is an infinite virtually cyclic group of Type~I, and $\widetilde{\Gamma}$ is of the form $F\rtimes_\alpha T$. In this situation, it was recently established that the Waldhausen Nil groups may be identified with the Farrell-Hsiang Nil groups as follows~\cite{DKR,LO2}:
\begin{equation*}
\operatorname{\textbf{Nil}}^W_n=\operatorname{Nil}^W_n(\Z[F];\Z[G_1\setminus F],\Z[G_2\setminus F])\cong \nkonealt[n]{\Z[F], \alpha}\cong \nkonealt[n]{\Z[F], \alpha^{-1}}.
\end{equation*}
In negative degrees, the Nil groups are described as follows.

\begin{thm}[{\cite[Theorem~2.1]{FJ}}]\label{th:farrelljones}
Let $\Gamma$ be an infinite virtually cyclic group. Then:
\begin{enumerate}[(a)]
\item $K_{-1}(\Z [\Gamma])$ is a finitely-generated Abelian group.
\item\label{it:farrelljonesb} $K_{-1}(\Z [\Gamma])$ is generated by the images of $K_{-1}(\Z [F])$ under the maps induced by the inclusions $F\subset \Gamma$, where $F$ runs over the representatives of the conjugacy classes of finite subgroups of $\Gamma$.
\item\label{it:farrelljones} $K_{-i}(\Z [\Gamma])=0$ for all $i\geq 2$.
\end{enumerate}
\end{thm}

In summary, in order to compute the $K$-groups of an infinite virtually cyclic group, we need to understand the $K$-groups of the corresponding finite kernel $F$ and of the associated Bass or Farrell-Hsiang Nil groups. 

\section{Preliminary $K$-theoretical calculations for $\mathbb{Z}[B_4(\St)]$}\label{sec:prelimkth}

Using the hyperbolicity of $B_4(\St)$ and the results of~\cite{JuLe} and~\cite[Example~3.2]{JLMP}, we may compute $K_n(\mathbb{Z}[B_4(\St)])$, obtaining:
\begin{equation*}
K_n(\mathbb{Z}[B_4(\St)])\cong A_n\oplus B_n\oplus \biggl(\bigoplus_{V\in\mathcal{V}}\coker[n]{V}\biggr),
\end{equation*}
where:
\begin{align*}
A_n &=\operatorname{Coker}\bigl( K_n(\mathbb{Z}[\quat])\to K_n(\mathbb{Z}[\quat[16]])\oplus K_n( \mathbb{Z}[\tstar])\bigr)\;\text{and}\\
B_n &=\operatorname{Ker}\bigl( K_{n-1}(\mathbb{Z}[\quat])\to K_{n-1}(\mathbb{Z}[\quat[16]])\oplus K_{n-1}(\mathbb{Z}[\tstar])\bigr).
\end{align*}
The group $\coker[n]{V}$ corresponds to the various Nil groups described in \resec{calcingred} (and will be determined in what follows), and the sum is over the family $\mathcal{V}$ of conjugacy classes of maximal infinite virtually cyclic subgroups of $B_4(\St)$. Using the pseudo-isotopy functor instead of $K$, we obtain similar formul{\ae} for the Whitehead and $\widetilde{K}_0$-groups:
\begin{align*}
\wh{B_{4}(\St)}&=\begin{cases}  
\operatorname{Coker}\bigl(\wh{\quat}\to \wh{\quat[16]} \oplus \wh{\tstar}\bigr) &\\
\oplus&\\
\ker{\widetilde{K}_{0}(\Z[\quat])\to \widetilde{K}_0(\Z[\quat[16]]) \oplus \widetilde{K}_0(\Z[\tstar])}&\\
         \oplus&\\
         \operatorname{Nil}_1
         \end{cases}\\       
\widetilde{K}_{0}(\Z[B_{4}(\St)])& =  \begin{cases}
\coker{\widetilde{K}_{0}(\Z[\quat])\to \widetilde{K}_{0}(\Z[\quat[16]]) \oplus \widetilde{K}_{0}(\Z[\tstar])}&\\
\oplus&\\
\ker{\widetilde{K}_{-1}(\Z[\quat])\to \widetilde{K}_{-1}(\Z[\quat[16]]) \oplus \widetilde{K}_{-1}(\Z[\tstar])}&\\
\oplus&\\
\operatorname{Nil}_0,
\end{cases}\\
\intertext{and the $K_{-1}$-group is given by:}
K_{-1}(\Z[B_{4}(\St)])&= \operatorname{Coker}\bigl(K_{-1}(\Z[\quat])\to K_{-1}(\Z[\quat[16]]) \oplus K_{-1}(\Z[\tstar])\bigr),
\end{align*}
where for $i=0,1$, $\operatorname{Nil}_i$ splits as a direct sum of Bass or Farrell-Hsiang Nil groups over representatives of $\mathcal{V}$ (see \resec{calcingred}).  From \reth{tablas}, we have the following isomorphisms:
\begin{align*}
&\wh{\quat}=0 && \widetilde{K}_{0}(\Z[\quat])\cong\Z_2 && K_{-1}(\Z[\quat])=0\\
&\wh{\quat[16]}\cong\Z && \widetilde{K}_{0}(\Z[\quat[16]])\cong\Z_2 && K_{-1}(\Z[\quat[16]])\cong\Z_2\\
&\wh{\tstar}=0 && \widetilde{K}_{0}(\Z[\tstar])\cong\Z_2 && K_{-1}(\Z[\tstar])\cong \Z.
\end{align*}
Moreover, by~\cite[Lemma~14.6]{Sw3}, the induction
$\widetilde{K}_{0}(\Z[\quat])\to \widetilde{K}_{0}(\Z[\quat[16]])$
is zero and the homomorphism
$\widetilde{K}_{0}(\Z[\quat])\to \widetilde{K}_{0}(\Z[\tstar])$
is an isomorphism by hyper-elementary induction (cf.~\cite[Theorem~14.1(1)]{Sw3}). 
Furthermore, by \rerem{negkfinite} and \reth{farrelljones}(\ref{it:farrelljones}), we obtain the following isomorphisms:
\begin{equation}\label{eq:marca}
\left\{ \begin{aligned}
\wh{B_4(\St)} &\cong\Z\oplus \operatorname{Nil}_1\\
\widetilde{K}_{0}(\Z[B_4(\St)]) &\cong\Z_2\oplus \operatorname{Nil}_0\\
K_{-1}(\Z[B_4(\St)]) &\cong\Z_2\oplus\Z\\
K_{-i}(\Z[B_4(\St)]) & \text{$\cong 0$ for all $i\geq 2$,}
\end{aligned}\right.
\end{equation}
which proves \reth{B4} up to the computation of the $\operatorname{Nil}_i$ terms in the first two isomorphisms. To complete the proof, we must compute the Nil groups that appear in the contribution of the conjugacy classes of maximal infinite virtually cyclic subgroups of $B_{4}(\St)$.

\section{Nil group computations}\label{sec:nilgrpcomp}

In this section, we compute the Bass Nil groups $\nkone[i]{\quat}$ for $i=0,1$, as well as the twisted versions. In the non-twisted case, we obtain the following result. 

\begin{prop}\label{prop:nkq8}
For $i=0,1$, the groups $\nkone[i]{\quat}$ are isomorphic to a countable, infinite direct sum of copies of $\Z_2$, $\Z_4$ or $\Z_2\oplus \Z_4$.
\end{prop}

In order to prove \repr{nkq8}, we first consider the ring $R$ of Lipschitz quaternions of the form $a+bi+cj+dk$, where $a,b,c,d\in \Z$ and $i,j,k$ are the quaternionic roots of $-1$, and compute its $NK_0$ and $NK_1$ groups. Recall that the ring $S$ of Hurwitz quaternions consists of the quaternions of the form $(a+bi+cj+dk)/2$ where $a,b,c$ and $d$ are integers that are either all even or all odd. Hence:
\begin{equation*}
\text{$R=\Z[i,j,k]$ and $S=\Z[i,j,k]+\Z  \Bigl[\frac{1+i+j+k}{2}\Bigr]$.}
\end{equation*}
By~\cite[Example~5.1]{Chat}, $S$ is a non-commutative principal ideal domain, and so is a regular ring. Let $M=(1+i)S$. Observe that $M\subset R\subset S$, and that $R/M$ and $S/M$ are the fields of two and four elements respectively. From this, it follows that $S/R$ is the group with $2$ elements. These computations involve the double relative term $K_1(R,S,M)$ for the injection $R\to S$ and ideal $M=(1+i)S$ that is described as follows~\cite[Theorem~0.2]{GW}:
\begin{equation}\label{eq:doublek}
\hspace*{-2mm}K_1(R,S,M)\!\cong\! (S/R)\!\otimes\! (M/M^2)/\!\setl{b\otimes cz\!+\!c\otimes zb\!-\!bc\otimes z\!}{\!b,c\in S,\, z\in M}.
\end{equation}
A straightforward computation yields:
\begin{align*}
M &=\setl{a+bi+cj+dk}{\text{$a+b$ and $c+d$ even}} + \Z [i+k]\\
M^2&=\setl{-2b+2ai+2dj+2ck}{\text{$a+b$ and $c+d$ even}} + 2\Z [i+j].
\end{align*}
Define the elements $b\in S/R$ and $u, w, z \in M/M^2$ to be the following cosets:
\begin{equation*}
\text{$b=\overline{\frac{1+i+j+k}{2}}$, $u=\overline{1+i}$, $w=\overline{j+k}$, and $z=\overline{1+i+j+k}$.}
\end{equation*}
Notice that the group generated by $u$ and $w$ is isomorphic to $\Z_2\oplus\Z_2$. On the other hand, $b$ is the only non-trivial element of $S/R$,  $b^2=b$ and the relations in~\reqref{doublek} become:
\begin{equation*}
b\otimes z=0.
\end{equation*} 
These identities imply that $K_1(R,S,M)=0$, and so by~\cite[Lemma~2.1]{We}, we obtain: 
\begin{equation}\label{eq:nk1R}
0=NK_1(R,S,M)\cong K_1(R,S,M)\otimes x\Z[x].
\end{equation}

\begin{thm}\label{th:Rreg}
The groups  $\nkonealt[0]{R}$  and  $\nkonealt[1]{R}$ are  trivial.
\end{thm}

\begin{proof}
As mentioned above, $R/M$ and $S/M$ are the fields of two and four elements respectively. Consider the following commutative square, where the right-hand vertical morphism is surjective:
\begin{equation*}
\begin{CD}
R@>>> S\\
@VVV    @VVV\\
R/M@>>>S/M.
\end{CD}
\end{equation*}
Now $R/M$, $S/M$ and $S$ are all regular rings, so their corresponding Nil groups vanish, and it follows from the associated Mayer-Vietoris sequence that $\nkonealt[0]{R}=0$. On the other hand, using the fact that $0=NK_1(R/M)=NK_1(S/M)=NK_1(S)$ and the description of the double relative group $K_1(R,S,M)$ given in~\cite[Theorem~1.1]{GW}, we obtain the isomorphism: 
\begin{equation*}
NK_1(R,S,M)\cong NK_1(R).
\end{equation*}
It follows from this that $NK_1(R)$ is trivial by \reqref{nk1R}. This completes the proof of the theorem.
\end{proof}

To prove \repr{nkq8}, we will require some general properties of Nil groups.

\begin{rem}\label{rem:infnils}
We recall the following facts about the Nil groups $NK_{i}$ and $\operatorname{\textbf{Nil}}^{W}_i$ for all $i\in\Z$: 
\begin{enumerate}
\item if $NK_{i}$ (resp.\ $\operatorname{\textbf{Nil}}^{W}_i$) is non trivial, it is an infinitely-generated group~\cite[Theorem~A]{LPW}.
\item if $A\subset NK_{i}$ (resp.\ $A\subset\operatorname{\textbf{Nil}}^{W}_i$) is a finite subgroup then $NK_{i}$ (resp.\ $\operatorname{\textbf{Nil}}^{W}_i$) contains an infinite direct sum of copies of $A$~\cite[Theorem~B]{LPW}.
\end{enumerate}
\end{rem}

In what follows, if $m\in \N$, $C_m$ will denote the cyclic group of order $m$.

\begin{proof}[Proof of \repr{nkq8}]
We first consider the case $i=0$. Let $\quat$ be equipped with the following presentation:
\begin{equation*}
\quat=\setbigangr{x,y}{x^2=y^2, yxy^{-1}=x^{-1}}.
\end{equation*}
By~\cite[Theorem 50.31, p.~266]{CR2}, the group ring $\Z[\quat]$ fits into the following Cartesian square:
\begin{equation}\label{eq:cartq8}
\begin{xy}*!C\xybox{%
\xymatrix{%
\Z[\quat] \ar[r]^(.42)f \ar[d]^q & \Z[C_2\times C_2] \ar[d]^{p}\\
R \ar[r] &\FF_2[C_2\times C_2],}}
\end{xy}
\end{equation}
where $q$ is defined on the generators of $\quat$ by $q(x)=i$ and $q(y)=j$, and $f$ is induced by the homomorphism $\quat\to C_2\times C_2$ given by taking the quotient of $\quat$ by its centre. The Cartesian square~\reqref{cartq8} gives rise to the following Mayer-Vietoris sequence:
\begin{multline}\label{eq:mayerviet}
\nkone[2]{\quat}\!\to\! \nkonealt[2]{R}\oplus \nkone[2]{C_2\times C_2}\!\to\! \\
 \nkoneftwo[2]{C_2\times C_{2}} \!\to\! \nkone{\quat}\!\to\! \cdots.\\
\end{multline}
By~\cite[Lemmas~5.3 and~5.4]{LO3},~\cite[Theorem~1.3]{We} and~\cite[Lemma~2.2]{WeK}, we have:
\begin{align}
\nkone{C_2\times C_2} & \!\cong\! \Omega_{\FF_2[x]}\!\cong\! \FF_2[x]\,\mathrm{d}x, \; \nkone[0]{C_2\times C_2} \!\cong\!  V\!=\!x\FF_2[x]\notag\\ 
\nkoneftwo{C_2} &\!\cong\! (1+x\epsilon\FF_2[x])^{\times}\!\cong\! V, \; \nkoneftwo[0]{C_2}\!=\! 0\notag\\ 
\nkoneftwo[0]{C_2\times C_2} &\!=\! 0.\label{eq:nkzeroc2c2}
\end{align}
Since $\FF_2[C_2\times C_2]\cong \FF_2[\epsilon,\nu]/(\epsilon^2,\nu^2)$, the ideal $I=\ang{\epsilon,\nu,\epsilon\nu}$ is nilpotent in $\FF_2[\epsilon,\nu]$, and it follows that $\nkoneftwo[0]{C_2\times C_2}\cong NK_0(\FF_2)=0$, which yields \req{nkzeroc2c2}. In a similar fashion, we have $\nkoneftwo[0]{C_2}= 0$. As Abelian groups, $\Omega_{\FF_2}$ and $V$ are both countable infinite direct sums of copies of $\Z_{2}$. As we saw in \reth{Rreg}, the ring $R$ has trivial Nil groups in degrees $0$ and $1$. On the other hand, observe that $\FF_2[C_2\times C_2]\cong \FF_2[\epsilon,\nu]/(\epsilon^2,\nu^2)$, hence by~\cite[Proposition~7.8]{Ba} and~\cite[Theorem~3.3]{Mar}, we have: 
\begin{equation*}
\nkoneftwo{C_2\times C_2} \cong (1+x\epsilon\FF_2[x])^{\times}\times(1+x\nu\FF_2[x])^{\times}\times(1+x\epsilon\nu\FF_2[x])^{\times}\cong V^3.
\end{equation*}
The Mayer-Vietoris sequence~\reqref{mayerviet} thus reduces to:
\begin{equation}\label{eq:mayerv}
\begin{gathered}
\nkone[2]{\quat}\to \nkonealt[2]{R}\oplus \nkone[2]{C_2\times C_2}\to \nkoneftwo[2]{C_2\times C_{2}} \stackrel{\delta}{\to}\\ \nkone{\quat}\stackrel{(f_{\ast 1})}{\to}
\nkone{C_2\times C_2}\stackrel{p_{\ast}}{\to} \nkoneftwo{C_2\times C_{2}} \stackrel{\tau}{\to}\\
\nkone[0]{\quat}\stackrel{f_{\ast 0}}{\to}
\nkone[0]{C_2\times C_2} \to 0
\end{gathered}
\end{equation}
(the labelled homomorphisms are discussed below).   
The  homomorphism 
\begin{equation*}
\map{p_{\ast}}{\nkone{C_2\times C_2}}[\nkoneftwo{C_2\times C_{2}}] 
\end{equation*}
is trivial since $\Z[C_2\times C_2]$ is reduced and $\nkoneftwo{C_2\times C_{2}}$ is Artinian~\cite[Theorem~3.3]{Mar}. Thus the part of~\reqref{mayerv} involving $\nkone[0]{\quat}$ is:
\begin{equation*}
0\to \nkoneftwo{C_2\times C_{2}} \stackrel{\tau}{\to} \nkone[0] {\quat}\stackrel{f_{\ast 0}}{\to}
\nkone[0]{C_2\times C_2}\to 0.
\end{equation*}
Since both  $\nkoneftwo{C_2\times C_{2}}$ and $\nkone[0]{C_2\times C_2}$ are non-trivial  infinite sums of copies of $\Z_2$, it follows by exactness that  $\nkone[0]{\quat}$ is an infinite direct sum of copies 
of $\Z_2,\Z_4$ or $\Z_2\oplus \Z_4$. 
 
We now turn to the case $i=1$. Consider the following homomorphism:
\begin{equation*}
\map{f_{\ast 1}}{\nkone{\quat}}[\nkone{C_2\times C_{2}}].
\end{equation*}
Exactness of the sequence~\reqref{mayerv}  and the fact that $p_{\ast}$ is the trivial homomorphism imply that $f_{\ast 1}$ is a surjection. On the other hand, since once more both $\nkoneftwo[2]{C_2\times C_2}$ and $\nkone[1]{C_2\times C_2}$ are infinite direct sum of copies of $\Z_2$, the latter by~\cite[Theorem~1.2]{AGH}, the result follows as before.
\end{proof}

In order to complete the proof of \reth{B4}, it remains to determine the twisted Nil groups of $\quat$. Recall from \resec{maxvcb4} that up to isomorphism, there are two non-trivial semi-direct products of the form $\quat\rtimes \Z$, namely $\quat\rtimes_{j} \Z$, where $j\in \brak{2,3}$. In what follows, we shall use the notation $\nkonetwist[i]{\quat}{\alpha} =NK_i(\Z[\quat],\alpha)$.

\begin{prop}\label{prop:twisnil}
Let $i\in \brak{0,1}$.
\begin{enumerate}[(a)]
\item For the action $\alpha$ of $\Z$ on $\quat$ of order $3$,
\begin{equation}\label{eq:twistalpha}
\nkonetwist[i]{\quat}{\alpha} \cong \nkonetwist[i]{\quat}{\alpha^{-1}} \cong \nkone[i]{\quat}.
\end{equation}
\item For the action $\alpha$ of $\Z$ on $\quat$ of order $2$, the twisted Nil groups are isomorphic to infinitely many copies of $\Z_2, \Z_4$ or $\Z_2\oplus \Z_4$.
\end{enumerate}
\end{prop}

\begin{proof}\mbox{}
\begin{enumerate}[(a)]
\item Since the action of $\Z$ on $\quat$ is of order three, there is a surjective homomorphism
$\varphi\colon\thinspace \quat\rtimes_{3} \Z \twoheadlongrightarrow \quat\rtimes \Z_3\cong \tstar$ defined by taking the $\Z$-factor modulo $3$.
We use the technique of induction on hyper-elementary subgroups~\cite[proof of Theorem~3.2]{FH} that asserts that:
\begin{equation*}
\operatorname{\textit{NK}}_i(\Z[\quat\rtimes_{3}\Z])\cong \lim_{H\in \operatorname{Hyp}} \operatorname{\textit{NK}}_i(\Z[\varphi^{-1}(H)]),
\end{equation*}
where $\operatorname{Hyp}$ denotes the set of hyper-elementary subgroups of $\quat\rtimes \Z_3$, and the limit is with respect to the morphisms induced by  conjugation and inclusion in the category $\operatorname{Hyp}$. Following the proof of \repr{torsbinpoly}, we see that the hyper-elementary subgroups of $\quat\rtimes \Z_3$ are isomorphic to one of $\Z_6, \Z_3, \Z_2, \Z_4$ or $\quat$, and their inverse images by $\varphi$ are isomorphic to $\Z_2\times \Z ,\Z, \Z_2\times\Z, \Z_4\times\Z$ and $\quat\times \Z$ respectively. With the exception of the last two, the corresponding group rings of these groups have trivial Nil groups. Further, the subgroups of $\quat\rtimes_3\Z$ that are isomorphic to $\Z_4\times\Z$ are pairwise conjugate, and there is only one maximal element of the form $\quat\times \Z$ in the limit. We thus obtain \req{twistalpha} using \repr{nkq8}.

\item Consider the action $\alpha$ of $\Z$ on $\quat$ of order $2$ given by exchanging the generators $x$ and $y$ of $\quat$. Comparing with the Cartesian square~\reqref{cartq8}, we observe that this action may be transposed in all the rings of~\reqref{cartq8}, thus giving rise to the following Cartesian square of twisted polynomial rings:
\begin{equation*}
\begin{CD}
\Z[\quat]_\alpha[t]@>f>> \Z[C_2\times C_2]_\alpha[t]\\
@VqVV    @VVpV\\
R_\alpha[t] @>>> \FF_2[C_2\times C_2]_\alpha[t],
\end{CD}
\end{equation*}
where the induced action of $\alpha$ exchanges the generators in all group rings, and exchanges $i$ and $j$ in $R$. By~\cite[Theorem 1.6]{F}, the Farrell-Hsiang group $\operatorname{\textit{NK}}_s$ of $R$ also vanishes for $s=0,1$. Moreover, let $I=\ang{\epsilon,\nu}$ be the nilpotent ideal generated by $\epsilon $ and $\nu$ in $\FF_2[C_2\times C_2]_\alpha [t]\cong \FF_2[\epsilon, \nu]_\alpha [t]$ since $\FF_2[\epsilon, \nu]_\alpha [t]/I\cong\FF_2$. By an argument similar to that given in the proof of \repr{nkq8}, it follows that 
$\nkonetwistftwo[0]{C_2\times C_2}{\alpha}=0$. Hence we obtain the following long exact sequence:
\begin{equation}\label{eq:twistseq}
\begin{gathered}
\nkonetwist[2]{\quat}{\alpha}\to \operatorname{\textit{NK}}_2^\alpha(R)\oplus\nkonetwist[2]{C_2\times C_2}{\alpha} \to \nkonetwistftwo[2]{C_2\times C_2}{\alpha} \stackrel{\delta}{\to}\\
 \nkonetwist{\quat}{\alpha}
\to\nkonetwist[1]{C_2\times C_2}{\alpha} \stackrel{p_{\ast}}{\to} \nkonetwistftwo{C_2\times C_2}{\alpha} \to\\ \nkonetwist[0]{\quat}{\alpha}\to\nkonetwist[0]{C_2\times C_2}{\alpha} \to 0.
\end{gathered}
\end{equation}

We first study the groups $K_s^{\alpha}(S[C_2\times C_2])$. Let $G$ be the amalgamated product defined as follows. Consider the non-trivial semi-direct products $G_1=C_4\rtimes C_2$ and $G_2=C_4\rtimes C_2$, where the cyclic groups of order four are generated by $a\in G_1$ and $b\in G_2$, and the cyclic groups of order $2$ are generated by $x\in G_1$ and $y\in G_2$. Let $D_2$ be the group $C_2\times C_2$ generated by $u$ and $v$, let $D_2\lhra G_1$ be the inclusion given by $u\longmapsto a^2$, $v\longmapsto x$, and let $D_2\lhra G_2$ be the inclusion given by $u\longmapsto y$ and $v\longmapsto b^2$. Then the amalgamated product $G_1 \bigast_{D_2}G_2$ is a virtually cyclic group. By~\cite{DKR}, the Farrell-Hsiang Nil groups $K^{\alpha}_{\ell}(S[C_2\times C_2])$ are isomorphic to the corresponding Waldhausen Nil groups: 
\begin{equation*}
\operatorname{\textbf{Nil}}^{W}_{\ell}(S[D_2]:S [G_1\setminus D_2], S [G_2\setminus D_2])
\end{equation*}
for all $s\in \Z$ and all rings $S$. 

Now, for the rings $S=\Z$ or $ \FF_2$, these Waldhausen Nil groups are isomorphic to infinite direct sums of copies of $\Z_2$. For $S=\Z$ and $\ell=0,1$, see~\cite[Theorem~5.2]{LO3} and~\cite[Section~7.2]{BLR}. 
Hence $K^{\alpha}_{\ell}(S[C_2\times C_2])$ is isomorphic to an infinite direct sum of copies of $\Z_2$ for $\ell=0,1,2$ and for $S=\Z$ or $\FF_2$. From the exact sequence~\reqref{twistseq}, for $\ell=0,1$, $\nkonetwist[0]{\quat}{\alpha}$ fits into an exact sequence of the form:
\begin{equation*}
0\to A_{\ell+1}\to \nkonetwist[\ell]{\quat}{\alpha} \to B_{\ell}\to 0,
\end{equation*}
where both $A_{\ell+1}$ and $B_{\ell}$ are isomorphic to infinite direct sums of copies of $\Z_2$. The result follows using \rerem{infnils}.\qedhere
\end{enumerate}
\end{proof}

Summing up, Propositions~\ref{prop:nkq8} and~\ref{prop:twisnil} give rise to the $\operatorname{Nil}_i$ summands of \req{marca}, and the decompositions of the statement of \reth{B4} then follow. 

\begin{appendices}

\chapter{The fibred isomorphism conjecture}\label{chap:fic}

\section*{The setup}

Let  $\map{\mathcal{S}}{\text{TOP}}[\text{$\Omega$-SPECTRA}]$ be a covariant homotopy functor. Let $\mathbf{F}$ be the category of continuous surjective maps: objects in $\mathbf{F}$ are continuous surjective maps $\map{p}{E}[B]$, where $E,B$ are objects in $\text{TOP}$, and morphisms between pairs of  maps $\map{p_{1}}{E_{1}}[B_{1}]$ and $\map{p_{2}}{E_{2}}[B_{2}]$ consist of continuous maps $\map{f}{E_{1}}[E_{2}]$ and $\map{g}{B_{1}}[B_{2}]$ that make the following diagram commute:
\begin{equation}\label{eq:diagfic}
\begin{CD}
E_1@>f>>E_2\\
@Vp_1VV @Vp_2VV\\
B_1@>g>>B_2.
\end{CD}\tag{A1}
\end{equation}
Within this framework, Quinn constructed a functor between $\mathbf{F}$ and $\text{$\Omega$-SPECTRA}$~\cite{Q}. The value of this $\Omega$-spectrum at the object $($\map{p}{E}[B]$)$ is denoted by $\mathbb{H}(B; {\mathcal{S}} (p))$, and the value at the object $(E \to \ast)$ is ${\mathcal S}(E)$. The map of spectra $\map{\mathbb{A}}{\mathbb{H}(B_{1}; {\mathcal S}(p_{1}))} [\mathbb{H}(B_{2}; \mathcal{S}(p_{2}))]$ associated to the commutative diagram~\reqref{diagfic} is known as the \emph{Quinn assembly map}. Other ingredients for the fibred isomorphism conjecture may be found in~\cite{FJC}.

\section*{The conjecture}

Given a discrete group $\Gamma$, let $E_{\mathcal{VC}}\Gamma$ be a universal $\Gamma$-space for the family of virtually cyclic
subgroups of $\Gamma$, let $\mathcal{B}_{\mathcal{VC}}\Gamma$ denote the orbit space $E_{\mathcal{VC}}\Gamma/\Gamma$, and let $X$ be a space on which $\Gamma$ acts freely and properly discontinuously. If $(f,g)$ is the following morphism in $\mathbf{F}$:
\begin{equation*}
\begin{CD}
E_{\mathcal{VC}}\Gamma \times_{\Gamma}X @>f>>X/\Gamma\\
@Vp_{1}VV @Vp_{2}VV\\
\mathcal{B}_{\mathcal{VC}}\Gamma @>g>>\ast
\end{CD}
\end{equation*}
then the \emph{Fibred Isomorphism Conjecture} for the functor
$\mathcal{S}$ and the group $\Gamma$ is the assertion that
\begin{equation*}
\map{\mathbb{A}}{\mathbb{H}(\mathcal{B}_{\mathcal{VC}}\Gamma ; \mathcal{S}(p_{1}))}[\mathcal{S}(X/ \Gamma)]
\end{equation*}
is a homotopy equivalence, and hence the induced map
\begin{equation*}
\map{\mathbb A_{\ast}}{\pi_{n} (\mathbb {H}(\mathcal{B}_{\mathcal{VC}}\Gamma ; \mathcal{S}(p_{1}))}[\pi_{n}(\mathcal{S}(X/\Gamma))]
\end{equation*}
is an isomorphism for all $n\in \mathbb{Z}$. This conjecture was stated in~\cite{FJC} for the functors
$\mathcal{S}=\mathcal{P}_{\ast}(\cdot)$, $\mathcal{K}(\cdot)$ and $\mathcal{L}^{-\infty}$, the pseudoisotopy, algebraic $K$-theory and $\mathcal{L}^{-\infty}$-theory functors respectively. In this paper, we use the functor $\mathcal{S}=\mathcal{K}_{\ast}(\cdot)$. The validity of this conjecture for $K$-theory and braid groups of $\St$ is proved in~\cite{JS}. Other cases in which the conjecture holds may be found in~\cite{Weg}.

\chapter{Braid groups}\label{chap:braids}

In this appendix, we recall briefly some basic facts and results about braid groups for the convenience of the reader. More information about braid groups may be found in~\cite{BCHWW,Bi,GM,H}. We refer the reader to~\cite{GJP} for a recent survey on surface braid groups.

If $n\geq 1$, the $n$-string Artin braid group, denoted $B_{n}$, may be defined by the following presentation~\cite{A1}:
\begin{enumerate}
\item[\underline{\textbf{generators}:}] $\sigma_1,\ldots, \sigma_{n-1}$ (known as the \emph{Artin generators}).
\item[\underline{\textbf{relations}:}] (known as the \emph{Artin relations}) 
\begin{gather}
\text{$\sigma_{i}\sigma_{j}=\sigma_{j}\sigma_{i}$ if $\lvert i-j\rvert\geq 2$
and $1\leq i,j\leq n-1$}\label{eq:Artin1}\\
\text{$\sigma_{i}\sigma_{i+1}\sigma_{i}=\sigma_{i+1}\sigma_{i}\sigma_{i+1}$ for
all $1\leq i\leq n-2$.}\label{eq:Artin2}
\end{gather}
\end{enumerate} 
The generator $\sigma_i$ may be regarded geometrically as the braid with a single positive crossing of the $i\th$ string with the $(i+1)\textsuperscript{st}$ string, while all other strings remain vertical (see Figure~\ref{fig:braid3}). It is convenient to view a geometric braid as being a collection of pairwise-disjoint arcs (or strings) in the Cartesian product $\mathbb{D}^{2}\times [0,1]$, where $\mathbb{D}^{2}$ is the $2$-disc, and each string joins two points of the form $(x,0)$ to $(y,1)$, where $x$ and $y$ belong to a set $X$ of $n$ distinguished basepoints lying in the interior of $\mathbb{D}^{2}$. The group operation in $B_{n}$ corresponds to concatenation of these geometric braids. The group $B_{1}$ is trivial, $B_{2}$ is infinite cyclic generated by $\sigma_{1}$, and for all $n\geq 2$, $B_{n}$ is infinite. For all $n\in \N$, $B_{n}$ is torsion free~\cite{Dy}.
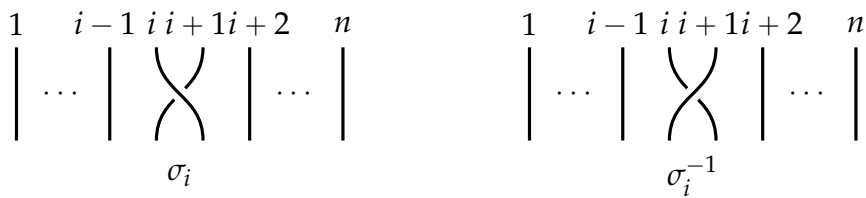
\begin{figure}[h]
\hfill
\begin{tikzpicture}[scale=0.62, very thick]

\foreach \k in {5}
{\draw (\k,3) .. controls (\k,2) and (\k-1,2) .. (\k-1,1);};

\foreach \k in {4}
{\draw[white,line width=6pt] (\k,3) .. controls (\k,2) and (\k+1,2) .. (\k+1,1);
\draw (\k,3) .. controls (\k,2) and (\k+1,2) .. (\k+1,1);};

\foreach \k in {15}
{\draw (\k,3) .. controls (\k,2) and (\k+1,2) .. (\k+1,1);};

\foreach \k in {16}
{\draw[white,line width=6pt] (\k,3) .. controls (\k,2) and (\k-1,2) .. (\k-1,1);
\draw (\k,3) .. controls (\k,2) and (\k-1,2) .. (\k-1,1);};

\foreach \k in {1,3,6,8,12,14,17,19}
{\draw (\k,1)--(\k,3);};

\foreach \k in {2,7,13,18}
{\node at (\k,2) {$\cdots$};};

\foreach \k in {1,12}
{\node at (\k,3.5) {$1$};
\node at (\k+1.9,3.5) {$i-1$};
\node at (\k+2.9,3.52) {$i$};
\node at (\k+3.85,3.5) {$i+1$};
\node at (\k+5.2,3.5) {$i+2$};
\node at (\k+7,3.5) {$n$};};

\node at (4.5,0.25) {$\sigma_{i}$};
\node at (15.5,0.25) {$\sigma_{i}^{-1}$};

\end{tikzpicture}
\hspace*{\fill}
\caption{The braid $\sigma_{i}$ and its inverse.}\label{fig:braid3}
\end{figure}
The map $\map{\sigma}{B_{n}}[\sn]$ defined on the generators by $\sigma(\sigma_{i})=(i,i+1)$ for all $1\leq i\leq n-1$ may be seen to be a surjective homomorphism. \label{page:sigma} Its kernel, denoted by $P_{n}$, is known as the $n$-string pure Artin braid group. Thus a braid $\beta\in B_{n}$ is \emph{pure} if for all $x\in X$, there is a string of $\beta$ that joins $(x,0)$ to $(x,1)$. The `half twist' braid $\garside$ is defined by:
\begin{equation*}
\garside=\prod_{i=1}^{n-1} \sigma_{1} \cdots \sigma_{n-i}.
\end{equation*}
Using the braid relations, one may check that the square $\ft$ of $\garside$, known as the `full twist' braid is given by:
\begin{equation}\label{eq:fulltwist}
\ft=(\sigma_1\cdots\sigma_{n-1})^n\in B_n.
\end{equation}
The braids $\garside$ and $\ft$ are illustrated in Figure~\ref{fig:twists}(\subref{fig:halftwist}) and~(\subref{fig:fulltwist}) in the case $n=6$. One may check that $\ft$ is a pure braid. If $n\geq 3$, $Z(B_{n})=Z(P_{n})=\ang{\ft}$, where $Z(G)$ denotes the centre of the group $G$~\cite{Ch}. 
\begin{figure}[h]
\hfill
\begin{subfigure}[b]{0.45\textwidth}
\begin{tikzpicture}[scale=0.75, very thick]
\foreach \j in {2,3,...,6} 
{\foreach \k in {2,...,\j}
{\draw (\k,\j) .. controls (\k,\j-0.5) and (\k-1,\j-0.5) .. (\k-1,\j-1);};
{\draw[white,line width=5pt] (1,\j) .. controls (1,\j-0.5) and (\j,\j-0.5) .. (\j,\j-1);
\draw (1,\j) .. controls (1,\j-0.5) and (\j,\j-0.5) .. (\j,\j-1);};
if \j>2 then \draw (\j,1) -- (\j,\j-1);
\draw (\j,6) -- (\j,6.5);
\draw (1,6) -- (1,6.5);
\foreach \j in {1,2,3,...,6} 
{\draw (\j,1) -- (\j,-0.35);}
};
\end{tikzpicture}
\hspace*{\fill}
\caption{The half twist braid $\garside[6]$ of $B_{6}$.}\label{fig:halftwist}
\end{subfigure}
\hfill\begin{subfigure}[b]{0.45\textwidth}
\begin{tikzpicture}[scale=0.75, very thick]
\foreach \j in {1,2,3,...,6} 
{\foreach \k in {2,...,6}
{\draw (\k,\j) .. controls (\k,\j-0.5) and (\k-1,\j-0.5) .. (\k-1,\j-1);};
{\draw[white,line width=5pt] (1,\j) .. controls (1,\j-0.5) and (6,\j-0.5) .. (6,\j-1);
\draw (1,\j) .. controls (1,\j-0.5) and (6,\j-0.5) .. (6,\j-1);};
\draw (\j,6) -- (\j,6.5);
\draw (\j,-0.25) -- (\j,0);
};
\end{tikzpicture}
\hspace*{\fill}
\caption{The full twist braid $\ft[6]$ of $B_{6}$.}\label{fig:fulltwist}
\end{subfigure}
\caption{The braids $\garside[6]$ and $\ft[6]$ of $B_{6}$.}\label{fig:twists}
\end{figure}
The Artin pure braid group is generated by the set $\brak{A_{i,j}}_{1\leq i<j\leq n}$~\cite[Lemma~I.4.2]{H}, where:
\begin{equation}\label{eq:defaij}
A_{i,j}=\sigma_{j-1}\cdots \sigma_{i+1} \sigma_{i}^{2} \sigma_{i+1}^{-1} \cdots \sigma_{j-1}^{-1}. 
\end{equation}
Geometrically, $A_{i,j}$ may be represented by a braid all of whose strings are vertical, with the exception of the $j$\up{th} string that wraps around the $i$\up{th} string as in Figure~\ref{fig:aij}.
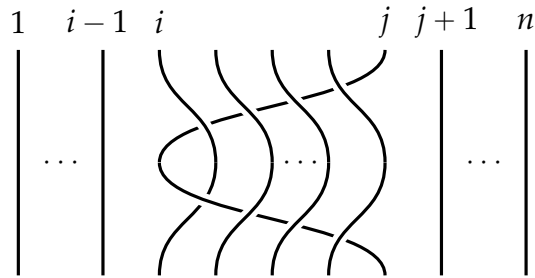
\begin{figure}[h]
\hfill
\begin{tikzpicture}[scale=0.75, very thick]
\draw (7,6.5) .. controls (7,5.5) and (3,5.5) .. (3,4.5);
\foreach \j in {3,...,6} 
{\draw[white,line width=5pt] (\j,6.5).. controls (\j,5.5) and (\j+1,5.5) .. (\j+1,4.5);
\draw (\j,6.5).. controls (\j,5.5) and (\j+1,5.5) .. (\j+1,4.5);
\draw[white,line width=5pt] (3,4.5) .. controls (3,3.5) and (7,3.5) .. (7,2.5);};
\draw (4,4.5) .. controls (4,3.5) and (3,3.5) .. (3,2.5);
\draw[white,line width=5pt] (3,4.5) .. controls (3,3.5) and (7,3.5) .. (7,2.5);
\draw (3,4.5) .. controls (3,3.5) and (7,3.5) .. (7,2.5);
\foreach \j in {4,...,6} 
{\draw[white,line width=5pt] (\j+1,4.5).. controls (\j+1,3.5) and (\j,3.5) .. (\j,2.5);
\draw (\j+1,4.5).. controls (\j+1,3.5) and (\j,3.5) .. (\j,2.5);
};
\foreach \j in {0.5,2,8,9.5}
{\draw (\j,6.5)-- (\j,2.5);};
\foreach \k in {0.75,5,8.25}
{\node at (\k+0.55,4.5) {$\cdots$};};
\node at (0.5,7) {$1$};
\node at (1.9,7) {$i-1$};
\node at (3,7) {$i$};
\node at (7,7) {$j$};
\node at (8.1,7) {$j+1$};
\node at (9.5,7) {$n$};
\end{tikzpicture}
\hspace*{\fill}
\caption{The element $A_{i,j}$ of $B_{n}$.}\label{fig:aij}
\end{figure}
In particular, for all $i=1,\ldots,n-1$, $A_{i,i+1}=\sigma_{i}^{2}$.

The Artin braid groups admit many different generalisations, one being that of \emph{surface braid groups}. If $M$ is a surface, orientable or not, with or without boundary, and with a finite number (possibly zero) of punctures, the $n$-string braid group $B_{n}(M)$ may be defined geometrically simply by replacing $\mathbb{D}^{2}$ by $M$. The subgroup $P_{n}(M)$ of $n$-string pure braids is defined in a manner similar to that for $P_{n}$. A number of presentations of $B_{n}(M)$ and $P_{n}(M)$ may be found in the literature, see~\cite{Be,GG1,GM0} for example. 

Braid groups may also defined topologically in terms of configuration spaces as follows. Let $F_n(M)$ denote the \emph{$n\th$ configuration space} of $M$ defined by:
\begin{equation*}
F_n(M)=\set{(p_1,\ldots,p_n)\in M^n}{\text{$p_i\neq p_j$ for all $i,j\in\brak{1,\ldots, n}$, $i\neq j$}}.
\end{equation*}
We equip $F_n(M)$ with the topology induced by the product topology on $M^n$. A transversality argument shows that $F_n(M)$ is a connected $2n$-dimensional open manifold. There is a natural free action of the symmetric group $\sn$ on $F_n(M)$ given by permutation of coordinates. The resulting orbit space $F_n(M)/\sn$ shall be denoted by $D_n(M)$, the \emph{$n\th$ permuted configuration space} of $M$, and may be thought of as the configuration space of $n$ \emph{unordered} points. The associated canonical projection $\map{p}{F_n(M)}[D_n(M)]$ is thus a regular $n!$-fold covering map~\cite[p.~14]{H}. Fox and Neuwirth showed that $P_n(M)\cong \pi_1(F_n(M))$ and $B_n(M) \cong \pi_1(D_n(M))$~\cite{FoN}. If $n=1$ then $F_{1}(M)=M$, and thus $B_{1}(M)=P_{1}(M)=\pi_{1}(M)$, so braid groups generalise the notion of fundamental group. The map $p$ gives rise to the following short exact sequence:
\begin{equation}\label{eq:sessn}
1\to P_{n}(M) \to B_{n}(M) \stackrel{p_{\ast}}{\to} \sn \to 1.
\end{equation}
In the case where $M$ is the disc, $p_{\ast}$ is the surjective homomorphism $\sigma$ described on page~\pageref{page:sigma}.

This topological definition is very useful in practice, and may be used as follows to obtain fibrations involving the configurations spaces, and (short) exact sequences bringing into play the homotopy groups of these spaces. Suppose that $M$ is a surface with empty boundary, and let $m>n\geq 1$. 
Then the map $\map{p_{m,n}}{F_{m}(M)}[F_{n}(M)]$ given by $p_{m,n}(x_{1},\ldots,x_{m}) = (x_{1},\ldots,x_{n})$ that forgets the last $m-n$ coordinates 
is a locally-trivial fibration, known as the \emph{Fadell-Neuwirth fibration}, with fibre $F_{m-n}(M\setminus \brak{z_{1},\ldots,z_{n}})$, where $(z_{1},\ldots,z_{n})$ is a basepoint of $F_{n}(M)$~\cite{FaN}. The fibre is known to be an Eilenberg--Mac~Lane space of type $K(\pi,1)$. Taking the long exact sequence in homotopy of the fibration, and using Fox and Neuwirth's isomorphisms mentioned above, we obtain the \emph{Fadell-Neuwirth short exact sequence of surface pure braid groups}:
\begin{equation}\label{eq:fnses}
1 \to \pi_{1}(F_{m-n}(M\setminus \brak{z_{1},\ldots,z_{n}})) \to P_{m}(M) \xrightarrow{(p_{m,n})_{\ast}} P_{n}(M) \to 1.
\end{equation}
The homomorphism $(p_{m,n})_{\ast}$ induced by the map $p_{m,n}$ may be visualised geometrically as the map that `forgets' the last $m-n$ strings of a braid in $P_{m}(M)$. Due to the fact that the higher homotopy groups of the braid groups of $\St$ and $\rp$ are non trivial, in order to obtain the short exact sequence~\reqref{fnses} for these two surfaces, we need to suppose additionally that $n\geq 3$ (resp.\ $n\leq 2$). In particular, if $m=n+1$, then~\reqref{fnses} becomes:
\begin{equation}\label{eq:fnsesspec}
1 \to \pi_{1}(M\setminus \brak{z_{1},\ldots,z_{n}}) \to P_{n+1}(M) \xrightarrow{(p_{n+1,n})_{\ast}} P_{n}(M) \to 1.
\end{equation}

The braid groups of $\St$ and $\rp$ are of particular interest, partly because they are the only surface braid groups to possess torsion, and as we explained in the introduction, the methods of~\cite{AFR,FR} cannot be applied to study their lower algebraic $K$-theory. The isomorphism classes of the maximal finite subgroups of $B_{n}(\St)$ are given in \reth{finitebn}. An analogous result for the braid groups of $\rp$ may be found in~\cite{GGjlms}. The braid groups of the sphere were initially studied by Fadell, Van Buskirk and Gillette~\cite{FVB,GVB,VB}. A presentation of $B_{n}(\St)$ due to the first two of these authors is given in \reth{fvb}. From a geometric point of view, the space $\St \times [0,1]$ in which geometric braids of the sphere are defined may be visualised as that between two concentric spheres (see~\cite[pp.~41,~42~and~45]{H} or~\cite[Figure~2.1(c), p.~193]{MK} for example), and the geometric representation of the generators of that presentation is as in Figure~\ref{fig:braid3}. Using such figures, the reader may convince himself or herself of the validity of the relations given in \reth{fvb}, in particular the `surface relation'~\reqref{surfacerel}. Other properties of $B_{4}(\St)$ that we use in this manuscript are given in \resec{gensB4}. The full twist braid $\ft$ also plays an important r\^{o}le in $B_{n}(\St)$. If $n\geq 3$, it is the unique element of $B_{n}(\St)$ of order $2$, it is the unique non-trivial torsion element of $P_{n}(\St)$, and it generates the centre of $B_{n}(\St)$~\cite{GVB,GGcamb}. The pure braid group $P_{4}(\St)$ is generated by the set $\brak{A_{i,j}}_{1\leq i<j\leq 4}$, where in terms of the generators $\sigma_{1}, \sigma_{2}$ and $\sigma_{3}$ of $B_{4}(\St)$, $A_{i,j}$ is given by~\reqref{defaij}, and its geometric representation within $\St \times [0,1]$ is as in Figure~\ref{fig:aij}. If $m\geq 1$, a presentation of $P_{m}(\St)$ may be obtained using techniques similar to those of~\cite[Proposition~7]{GG4}. Note that if one takes $n=0$ in that proposition, one does indeed obtain a presentation of $P_{m}(\St)$ whose generating set is $\brak{A_{i,j}}_{1\leq i<j\leq m}$, and whose relations are given by those of~\cite[Lemma~I.4.2]{H} for $P_{m}$, and by the `surface relations' that are of the form:
\begin{equation}\label{eq:surfrelpnS2}
\Biggl( \prod_{i=1}^{j-1} A_{i,j} \Biggr)\Biggl( \prod_{k=j+1}^{m} A_{j,k} \Biggr)=1
\end{equation}
for all $1\leq j\leq m$.

Taking $M=\St$ and $n=3$ in~\reqref{fnsesspec} yields:
\begin{equation}\label{eq:fnsesS2}
1 \to \pi_{1}(\St\setminus \brak{z_{1},z_{2},z_{3}}) \to P_{4}(\St) \xrightarrow{(p_{4,3})_{\ast}} P_{3}(\St) \to 1.
\end{equation}
The kernel is a free group of rank~$2$ that may be identified with the subgroup of $P_{4}(\St)$ generated by $(A_{1,4}, A_{2,4})$, and the quotient $P_{3}(\St)$ is equal to $\ang{\ft[3]}$, and is isomorphic to $\Z_{2}$. The map $\map{s}{P_{3}(\St)}[P_{4}(\St)]$ defined by $s(\ft[3])=\ft[4]$ is a homomorphism, and is a section for $(p_{4,3})_{\ast}$ since removal of the last string of $\ft[4]$ in $P_{4}(\St)$ yields the braid $\ft[3]$ in $P_{3}(\St)$, \emph{i.e.} $(p_{4,3})_{\ast}(\ft[4])=\ft[3]$. So the short exact sequence~\reqref{fnsesS2} splits, and since $\ft[4]\in Z(P_{4}(\St))$, it follows that:
\begin{equation}\label{eq:dirprodp4S2}
P_{4}(\St) \cong \F[2] \times \Z_{2}.
\end{equation}
From this, it follows also that $Z(P_{4}(\St))=\ang{\ft[4]}$.

\end{appendices}


\backmatter

\end{document}